\documentclass[a4paper,11pt]{amsart}

\usepackage[all]{xy}

\usepackage[utf8]{inputenc}							% font and encoding
\usepackage[T1]{fontenc}
\usepackage[english]{babel}

\usepackage{amsfonts}								% ams-packages
\usepackage{amsmath}
\usepackage{amssymb}
\usepackage{amsthm}
\usepackage{mathtools}

\makeatletter										% table of contents.
\def\l@subsection{\@tocline{2}{0pt}{2.5pc}{5pc}{}}	%
\makeatother										%

\usepackage{graphicx}								% images and drawing
\usepackage{tikz}
\usepackage{tikz-cd}

\usepackage{hyperref}								% hyperlinks
	\hypersetup{colorlinks=false, urlcolor=black, linkcolor=black}
\DeclareRobustCommand{\SkipTocEntry}[5]{}

\usepackage{enumitem}								% arbitrary list items

\renewcommand{\phi}{\varphi}

\newcommand{\N}{\mathbb{N}}						% Nonnegative integers
\newcommand{\Z}{\mathbb{Z}}						% Integers
\newcommand{\R}{\mathbb{R}}						% Reals
\newcommand{\eps}{\varepsilon}					% Epsilon shortcut
\DeclareMathOperator{\mass}{\mathbf{M}}

\newcommand{\dd}								% Differential d
	{\mathop{}\!\mathrm{d}}						
\newcommand{\ddn}[1]							% Powers of a differential d
	{\mathop{}\!\mathrm{d^{#1}}}

\newcommand{\abs}[1]							% Absolute value
	{\left| #1 \right|}
\newcommand{\smallabs}[1]						% Small abs bars which won't scale
	{\lvert #1 \rvert}	
\newcommand{\norm}[1]							% Norm 
	{\left\lVert #1 \right\rVert}	
\newcommand{\smallnorm}[1]						% Small norm bars which won't scale
	{\lVert #1 \rVert}						
\newcommand{\ip}[2]								% Inner product
	{\left< #1 , #2 \right>}

\DeclareMathOperator{\proj}{pr}					% Projections
\DeclareMathOperator{\vol}{vol}					% Volume (and the volume form)
\DeclareMathOperator{\spt}{spt}					% Support
\DeclareMathOperator*{\esssup}{ess\,sup}		% Essential supremum
\DeclareMathOperator{\diam}{diam}				% Diameter
\DeclareMathOperator{\len}{len}					% Length

\let\Im\relax
\DeclareMathOperator{\Im}{Im}					% Note: Image (pair is \ker).

\newcommand{\loc}{\mathrm{loc}}

\DeclareMathOperator{\hausd}{\mathcal{H}}

\newcommand{\cA}{\mathcal{A}}
\newcommand{\cB}{\mathcal{B}}

\newcommand{\cW}{\mathcal{W}}

\newtheorem{thm}{Theorem}[section]{\bf}{\it}
\newtheorem{lemma}[thm]{Lemma}
\newtheorem{prop}[thm]{Proposition}
\newtheorem{cor}[thm]{Corollary}

\theoremstyle{definition}
\newtheorem{defn}[thm]{Definition}

\theoremstyle{remark}
\newtheorem{rem}[thm]{Remark}

\numberwithin{equation}{section}

\begin{document}

\title[Notes on quasiregular maps between manifolds]{Notes on quasiregular maps between Riemannian manifolds}
\author{Ilmari Kangasniemi}

\maketitle

\begin{abstract}
	These notes provide an exposition on obtaining the well-known standard results of quasiregular maps on Riemannian manifolds, given the corresponding theory in the Euclidean setting. We recall several different approaches to first-order Sobolev spaces between Riemannian manifolds, and show that they result in equivalent definitions of quasiregular maps. We explain how e.g.\ Reshetnyak's theorem, degree and local index theory, and the quasiregular change of variables formula are transferred into the manifold setting from Euclidean spaces. Finally, we conclude with a proof of the basic fact that pull-backs with quasiregular maps preserve Sobolev differential forms of the conformal exponent.
\end{abstract}

\section{Introduction}

In these notes, we present a detailed exposition on how to obtain the basic analytical theory of quasiregular maps between Riemannian manifolds, given the corresponding theory of quasiregular maps between Euclidean spaces.

Higher-dimensional quasiregular and quasiconformal mappings are generally studied in three main contexts: on domains in $\R^n$, on oriented Riemannian manifolds, and in sufficiently regular metric measure spaces. The Euclidean theory is discussed in multiple detailed introductory monographs: see e.g.\ Rickman \cite{Rickman_book}, Reshetnyak \cite{Reshetnyak-book}, Iwaniec and Martin \cite{Iwaniec-Martin_book}, and Gehring, Martin and Palka \cite{Gehring-Martin-Palka_book}. We also mention the closely related texts by Hencl and Koskela \cite{Hencl-Koskela_book} and Heinonen, Kilpeläinen and Martio \cite{Heinonen-Kilpelainen-Martio_book}, as well as the monograph on the planar theory by Astala, Iwaniec and Martin \cite{Astala-Iwaniec-Martin_Book}. 

The metric version of the theory is more recent. The quasiconformal side of the metric theory is better established, and an introduction to the theory can be found in the monograph of Heinonen \cite{Heinonen_book}; we also mention the related monograph by Heinonen, Koskela, Shanmugalingam and Tyson \cite{Heinonen-Koskela-Shanmugalingam-Tyson}. The non-homeomorphic side of the theory is currently under active development; see e.g.\ \cite{Kirsila_GenManifoldsMFD}, \cite{Guo-Williams} and \cite{Lindquist-Pankka}.

However, while the Riemannian theory has existed for almost as long as the Euclidean one, introductory material for it appears far more uncommon. The aforementioned monographs \cite{Reshetnyak-book} and \cite{Gehring-Martin-Palka_book} contain some discussion of the manifold theory, but are regardless overwhelmingly focused on the Euclidean theory. One detailed account is in a paper by Suominen \cite{Suominen_QC-Manifold}, but it is limited to the quasiconformal case and uses a modulus approach. The more typical analytic approach is considerd in the monograph of Holopainen and Pankka \cite{Holopainen-Pankka_Notes} and in lecture notes by Pankka \cite{Pankka_Lectnotes_Fribourg}; however, these monographs are less detailed in order to cover more deep results. 

A possible reason for this apparent shortage of introductory material is that much of the basic Riemannian theory is relatively easy to derive directly from the Euclidean one using charts. However, the required prior information and the countless small details in this process of arriving at the Riemannian theory make an exposition easily long enough to encompass an entire set of notes.

The purpose of these notes is to provide exactly that: a detailed account of how to build the foundation for the Riemannian version of the analytical theory of quasiregular maps. We do not aim for a wide variety of deep results. Instead, we try to be as thorough as possible with the details, in hopes of allowing the reader to work with the Riemannian quasiregular theory without second-guessing the validity of every basic detail. Indeed, the text is largely motivated by the countless small questions the author had to sort out when learning the subject.

The first part of these notes is mainly built around the fact that the analytic definition of quasiregular maps has a first order Sobolev space, which in the Riemannian setting would have to be defined for mappings between two manifolds. Various authors use different equivalent approaches to this part of the definition in the Riemannian setting. Hence, we go over several of these possible approaches in detail, and show their equivalence.

The other main part of the notes is an exposition on how to transfer the following fundamental parts of the analytic theory to the manifold setting: Reshetnyak's theorem, normal domains, degree and local index theory, Lusin conditions, and the quasiregular change of variables theorem. Finally, we conclude with the quasiregular pull-back formula for Sobolev differential forms of the conformal exponent. In the process of reaching this result, we also end up discussing Sobolev differential forms in general.

\subsection*{Structure of these notes}

In Sections \ref{sect:diff_geo_prelims} and \ref{sect:meas_diff_geo_prelims}, we recall necessary parts of differential and Riemannian geometry for setting up the theory. In particular, Section \ref{sect:diff_geo_prelims} is a very brief recap of the smooth Riemannian geometry used in the notes, mainly to clarify the precise notation and terminology used. Section \ref{sect:meas_diff_geo_prelims} on the other hand reviews measurable Riemannian geometry, where slightly more detail is given due to us finding the textbook references for this part of the theory to be less commonplace. 

In Section \ref{sect:eucl_sobolev_prelims}, we review the necessary part of the theory of first order Sobolev spaces in the Euclidean setting. Due to its importance to us, special attention is placed placed on the Sobolev chain rule with $C^1$-smooth Lipschitz and bilipschitz maps. In the following Section \ref{sect:sobolev_manifold_domain}, we then define first order Sobolev maps with a Riemannian manifold domain. 

Sections \ref{sect:sobolev_cont_functions}, \ref{sect:sobolev_nash_embedding} and \ref{sect:sobolev_intrinsic_derivative} then catalogue various approaches to first order Sobolev functions for which the target space is also a Riemannian manifold. In Section \ref{sect:sobolev_cont_functions}, we consider a simpler case in which the map under consideration is already known to be continuous, in which case we may use a similar chart-based approach as is used in moving to manifold domains. In Section \ref{sect:sobolev_nash_embedding}, we discuss an approach using the Nash embedding theorem, which yields an embedding of the target manifold into a higher dimensional Euclidean space. Finally, in Section \ref{sect:sobolev_intrinsic_derivative}, we discuss an approach based on post-composition with compactly supported $C^1$-functionals, which has been studied in detail in \cite{Convent-VanSchaftingen_Sobolev}. In particular, it is shown in Theorem \ref{thm:local_sobolev_eq_of_defs} that these approaches all yield the same first order locally Sobolev functions whenever applicable.

Before moving on to the quasiregular theory, we spend Section \ref{sect:Sobolev_forms} discussing Sobolev differential forms, and present proofs for some of their basic properties. Afterwards, we move on to the quasiregular theory in Section \ref{sect:definitions_of_QR}, where we state several definitions of quasiregular maps between manifolds and show their equivalence. The main result of that section, Theorem \ref{thm:equivalence_of_definitions}, is as follows.
\begin{thm}
	The following four approaches to the definition of a quasiregular map $f$ between oriented Riemannian $n$-manifolds $M, N$ are all equivalent.
	\begin{itemize}
		\item $f$ is a continuous $W^{1,n}_\loc(M, N)$-map with $\abs{Df(x)}^n \leq K J_f(x)$ for almost every $x$, where $W^{1,n}_\loc(M, N)$ is defined by any of the following methods:
		\begin{itemize}
			\item by using bilipschitz charts and the continuity of $f$ to check whether $f \in W^{1,n}_\loc(M, N)$;
			\item by defining $W^{1,n}_\loc(M, N)$ using a Nash embedding;
			\item by defining $W^{1,n}_\loc(M, N)$ using the theory in \cite{Convent-VanSchaftingen_Sobolev}.
		\end{itemize}
		\item For every $x \in M$, $K' > K$, and $L > 1$, we find orientation preserving $L$-bilipschitz charts $\phi$ and $\psi$ at neighborhoods of $x$ and $f(x)$, respectively, such that the composition $\psi \circ f \circ \phi^{-1}$ is well defined and $K'$-quasiregular in the Euclidean sense.
	\end{itemize}
\end{thm}

Following this statement of definitions, we then begin transferring the basic analytic properties of quasiregular maps from the Euclidean setting to the manifold setting. The main topic of Section \ref{sect:reshetnyaks_theorem} is to transfer Reshetnyak's theorem to manifolds.

\begin{thm}[Reshetnyak]
	Let $f \colon M \to N$ be a non-constant quasiregular map between connected, oriented Riemannian $n$-manifolds. Then $f$ is discrete, open, and sense-preserving.
\end{thm}

Deriving the manifold version of Reshetnyak's theorem from the Euclidean version is straightforward and easy. Instead, most of the work comes from discussing the foundations of local topological orientations and local index theory, which is the language in which the sense-preserving part of Reshetnyak's theorem is usually given. Moreover, Section \ref{sect:reshetnyaks_theorem} also contains introductions to several standard topological tools of quasiregular analysis, such as normal domains and global degree theory.

For the following Section \ref{sect:change_of_variables}, the overarching goal is to transfer the change of variables formula for quasiregular maps to the manifold setting. For this purpose, we also end up discussing multiplicity, branching and the Lusin properties.

\begin{thm}[Change of variables]
	Let $M$ and $N$ be connected, oriented Riemannian $n$-manifolds, and let $f \colon M \to N$ be a non-constant quasiregular map. Suppose that $\omega$ is a locally integrable differential $n$-form on $N$. Then $f^* \omega$ is locally integrable, and we have a local integral transformation formula
	\[
		\int_U f^* \omega = \int_N N(f, \cdot, U) \omega,
	\]
	where $N(f, \cdot, U)$ is the multiplicity function.
\end{thm}

Finally, in Section \ref{sect:quasiregular_pullback}, we discuss pull-backs of differential forms by quasiregular maps. This is perhaps the most characteristic part of the quasiregular theory on higher dimensional Riemannian manifolds, as similar concepts are more seldom seen in the Euclidean setting. Our final goal is a complete proof of the following basic result.

\begin{thm}
	Let $M$ and $N$ be connected, oriented Riemannian $n$-manifolds, and let $f \colon M \to N$ be a non-constant quasiregular map. Suppose that $\omega \in L^{n/k}_\loc(\wedge^k N)$ and that $\omega$ has a weak exterior derivative $d\omega \in L^{n/(k+1)}_\loc(\wedge^{k+1} N)$, where we intepret $n/0 = \infty$ in the case $k = 0$. Then $f^* \omega \in L^{n/k}_\loc(\wedge^k M)$, and $f^* d\omega \in L^{n/(k+1)}_\loc(\wedge^{k+1} M)$ is the weak exterior derivative of $f^* \omega$.
\end{thm}

\subsection*{Thanks}

The author extends his thanks to Pekka Pankka for his continued guidance on the topic. Moreover, the author thanks everyone who he has discussed the topic of these notes with. A special mention goes to Ville Tengvall, who helped by pointing out some existing texts on the topic.

During the time of writing these notes, the author has been employed first by the University of Helsinki in Finland, and later by Syracuse University in the USA.

%\vfill\pagebreak

\tableofcontents

%%%%%%%%%%%%%%%%%%%%%%%%%%%%%%%%%%%%%%%%%%%%%%%%%%%%%%%%%%%%%%%%%%%%%%%%%%%%%%%%%%%%%%%%%%%%

\section{Smooth differential geometry}\label{sect:diff_geo_prelims}

In this section we state the necessary prerequisites of smooth Riemannian geometry, including our notation of choice. We assume that the reader is familiar with the theory. For details, we refer to e.g.\ \cite{Lee_RiemGeo} and \cite{Klingenberg_RiemannGeoBook}.

\subsection{Smooth manifolds}

A topological space $M$ is a \emph{topological $n$-manifold} if it is Hausdorff, second countable, and if every $x \in M$ has a neighborhood homeomorphic to $\R^n$. A \emph{chart} on a topological manifold $M$ is a topological embedding $\phi \colon U \to \R^n$, where $U \subset M$ and $\phi U \subset \R^n$ are open and nonempty.

A \emph{smooth atlas} $\cA$ on a topological manifold $M$ is a collection of charts on $M$, such that the domains of the charts in $\cA$ cover $M$, and for every pair of charts $\phi \colon U \to \R^n$ and $\psi \colon V \to \R^n$ in $\cA$ with $U \cap V \neq \emptyset$, the \emph{transition function} $\phi \circ \psi^{-1} \colon \psi(U \cap V) \to \phi(U \cap V)$ is a smooth diffeomorphism. A \emph{smooth differential $n$-manifold} is then a topological $n$-manifold equipped with a maximal smooth atlas.

Given a smooth manifold $M$, we denote its \emph{tangent space} at a point $x \in M$ by $T_x M$, and the resulting \emph{tangent bundle} by $TM$. Similarly, we denote the \emph{cotangent space} at $x \in M$ by $T^*_x M$, and the \emph{cotangent bundle} by $T^* M$. The $k$-th \emph{exterior power} of the tangent bundle is denoted $\wedge^k TM$, and similarly $\wedge^k T^*M$ for the cotangent bundle.

We denote by $C(M, N)$ the set of continuous functions $f \colon M \to N$ between two smooth manifolds $M$ and $N$. We also abbreviate $C(M, \R) = C(M)$. Moreover, a function $f \in C(M, N)$ is \emph{$k$ times differentiable} for $k \in \{1, 2, \dots, \infty\}$ if, for every $x \in M$, we can find charts $\phi \colon U \to \R^n$ and $\psi \colon V \to \R^m$ on $M$ and $N$, respectively, such that $x \in U$, $fU \subset V$, and $\psi \circ f \circ \phi^{-1}$ is a $k$ times differentiable map between Euclidean domains. The space of $k$ times differentiable maps from $M$ to $N$ is denoted $C^k(M, N)$, with again $C^k(M) = C^k(M, \R)$. We also refer to $\infty$ times differentiable maps as \emph{smooth} maps.

If $M$ and $N$ are smooth Riemannian manifolds and $f \in C^1(M, N)$, then we denote the \emph{differential map} of $f$ between tangent bundles by $Df \colon TM \to TN$. The restriction of this map to an individual tangent space at $x \in M$ is denoted $Df(x) \colon T_x M \to T_{f(x)} N$. The resulting induced map between the $k$-th exterior powers of the tangent spaces is denoted by $\wedge^k Df \colon \wedge^k TM \to \wedge^k TN$.

A \emph{differential $k$-form} $\omega$ on a smooth manifold $M$ is a section $\omega \colon M \to \wedge^k T^*M$, and the value of $\omega$ at a point $x \in M$ is denoted $\omega_x \in \wedge^k T_x^*M$. Since $\wedge^k T^*M$ also has a smooth manifold structure, we obtain classes of continuous, and $l$-times differentiable differential forms. We denote the space of continuous $k$-forms on $M$ by $C(\wedge^k M)$, and the space of $l$ times differentiable $k$-forms on $M$ by $C^l(\wedge^k M)$.

The \emph{exterior derivative} of differential forms is denoted by $d \colon C^\infty(\wedge^k M) \to C^\infty(\wedge^{k+1} M)$. If $f \in C^1(M, N)$ where $M$ and $N$ are smooth manifolds, then the \emph{pull-back} of a differential form $\omega \colon N \to \wedge^k T^*N$ is the differential form $f^*\omega \colon N \to \wedge^k T^*N$ given by
\[
	(f^* \omega)_x = \omega_{f(x)} \circ (\wedge^k Df(x)).
\]
The \emph{support} of a differential $k$-form $\omega$ on $M$ is given by
\[
	\spt(\omega) = \overline{\{x \in M : \omega_x \neq 0\}},
\]
and we use the sub-index 0 to denote spaces of compactly supported differential forms; for instance, $C^\infty_0(\wedge^k M)$ denotes the space of compactly supported smooth differential $k$-forms on $M$.

\subsection{Riemannian manifolds}

A \emph{Riemannian metric} $\ip{\cdot}{\cdot}$ is a smooth choice of inner product $\ip{\cdot}{\cdot}_x$ at every tangent space $T_xM$. A \emph{Riemannian manifold} $M$ is then a smooth manifold equipped with a Riemannian metric.

Suppose then that $M$ is a Riemannian manifold and $x \in M$. Then there exists a standard duality isomorphism $\iota_x \colon T_x M \to T_x^* M$ defined by
\[
(\iota_x(v))(w) = \ip{v}{w}_x
\]
for $v, w \in T_x M$. The isomorphism $\iota_x$ induces an inner product on $T^*_x M$ by 
\[
	\ip{\alpha}{\beta}_x = \ip{\iota_x^{-1}(\alpha)}{\iota_x^{-1}(\beta)}_x
\]
for $\alpha, \beta \in T^*_x M$.

Moreover, the inner products $\ip{\cdot}{\cdot}_x$ on $T_x M$ and $T^*_x M$, $x \in M$, induce a \emph{Grassmann inner product} $\ip{\cdot}{\cdot}_x$ on the exterior powers $\wedge^k T_xM$ and $\wedge^k T^*_x M$. The Grassmann inner product on $\wedge^k T_xM$ (or alternatively on $\wedge^k T^*_x M$) is the unique inner product which satisfies
\[
	\ip{v_1\wedge \cdots \wedge v_k}{w_1\wedge \cdots \wedge w_k}_x 
	= \det \left[\ip{v_i}{w_j}_x\right]_{ij}
\]
for all $v_i, w_i \in T_xM$ (or alternatively for all $v_i, w_i \in T_x^* M$), $i \in \{1, \dots, k\}$.

For all of the spaces $T_x M$, $T_x^*M$, $\wedge^k T_x M$ and $\wedge^k T_x^* M$, we denote the norm induced by the inner product $\ip{\cdot}{\cdot}_x$ by $\abs{\cdot}_x$. We also usually omit the sub-index $x$ from $\ip{\cdot}{\cdot}_x$ and $\abs{\cdot}_x$ if it can be understood from the rest of the expression.

For a continuously differentiable path $\gamma \in C^1([0,1], M)$, we denote by $\len \gamma$ the \emph{length} of $\gamma$, defined by
\[
\len \gamma = \int_0^1 \abs{\gamma'(t)} \dd t.
\]
We denote by $d_M$ the \emph{Riemannian distance} on $M$, where $d_M(x,y)$ is the infimal length of a $C^1$-smooth path connecting the points $x, y \in M$. Moreover, for every $x \in M$, we denote the \emph{exponential map} by $\exp_x \colon B \to M$, where $B = B_{T_x M}(0, r) \subset T_x M$. We recall in particular that the exponential map is smooth if $r$ is sufficiently small, and that $D \exp_x(0) \colon T_{0} T_x M \to T_x M$ is an isometry when $T T_x M$ is equipped with its natural induced Riemannian metric.

\subsection{Volume and orientation}

A smooth $n$-manifold $M$ is \emph{orientable} if there exists a maximal smooth atlas $\cA$ on $M$ whose transition functions have nonnegative Jacobian determinants. Such an atlas is an \emph{orientation} of $M$, and $M$ is called \emph{oriented} if such an atlas has been fixed. Charts contained in the orientation of $M$ are called \emph{positively oriented}.

For an oriented Riemannian $n$-manifold $M$, we denote the \emph{volume form} induced by the orientation and the Riemannian metric by $\vol_M \in C^\infty(\wedge^n M)$. We denote by $J_f \colon M \to \R$ the \emph{Jacobian determinant}, or \emph{Jacobian} for short, of a function $f \in C^1(M, N)$ between two oriented $n$-manifolds. In the case where $M$ and $N$ are not Euclidean, we define $J_f$ by the relation $f^* \vol_N = J_f \vol_M$. Equivalently, for $x \in M$, $J_f(x)$ is the determinant of the matrix of $Df(x)$ when written in positively oriented orthonormal bases of $T_xM$ and $T_{f(x)} N$, where an ordered basis $e_1, \dots, e_n$ of $T_x M$ is \emph{positively oriented} if $\vol_M(e_1 \wedge \dots \wedge e_n) > 0$.

Moreover, for an oriented Riemannian $n$-manifold $M$, we define the \emph{integral} of a continuous $n$-form $\omega = g \vol_M$ in the domain of a positively oriented smooth chart $\phi \colon U \to \R^n$ on $M$ by
\[
	\int_U \omega = \int_{\phi U} (g \circ \phi^{-1})(x) J_{\phi^{-1}}(x) \dd x.
\]
The integral of $\omega$ over the entire manifold $M$ is then defined using a smooth partition of unity.

Due to its importance for these notes, we recall a version of \emph{Stokes' theorem}, as well as a consequence of it. For details, see e.g.\ \cite[Theorem 14.9]{Lee_RiemGeo}.

\begin{thm}[Stokes]\label{thm:stokes}
	Let $M$ be an oriented Riemannian $n$-manifold (without boundary), and let $\omega \in C^\infty_0(\wedge^{n-1} M)$. Then
	\[
		\int_M d\omega = 0.
	\]
\end{thm}
\begin{cor}\label{cor:stokes_corollary}
	Let $M$ be an oriented Riemannian $n$-manifold (without boundary), and let $\omega \in C^\infty(\wedge^{k} M)$, where $k \in \{0, \dots, n-1\}$. Then for every $\eta \in C^\infty_0(\wedge^{n-k-1} M)$, we have
	\[
		\int_M \omega \wedge d\eta = (-1)^{k+1} \int_M d\omega \wedge \eta.
	\]
\end{cor}

%%%%%%%%%%%%%%%%%%%%%%%%%%%%%%%%%%%%%%%%%%%%%%%%%%%%%%%%%%%%%%%%%%%%%%%%%%%%%%%%%%%%%%%
\section{Measurable Riemannian geometry}\label{sect:meas_diff_geo_prelims}

In this section, we discuss measurable functions and differential forms in the setting of Riemannian geometry. Unlike smooth Riemannian geometry, we do not assume familiarity with this part of the theory. Hence, we will provide proofs and explanations whenever appropriate, although some of the details are necessarily left to the reader due to how significant a diversion they would present.

\subsection{Lipschitz maps and Jacobian estimates}

We begin the section with a discussion on $C^1$-smooth Lipschitz maps between Riemannian manifolds. Although this is still in the realm of smooth Riemannian geometry, it seems to be covered in less detail in most textbooks. Hence, we give a somewhat more detailed account.

Suppose that $(X, d_X)$ and $(Y, d_Y)$ are metric spaces. A map $f \colon X \to Y$ is \emph{$L$-Lipschitz} for $L \geq 0$ if
\[
	d_Y(f(x), f(y)) \leq L d_X(x,y)
\]
for all $x, y \in X$. Moreover, $f$ is \emph{$L$-bilipschitz} for $L > 0$ if
\[
	L^{-1}d_X(x, y) \leq d_Y(f(x), f(y)) \leq L d_X(x,y)
\]
for all $x, y \in X$.

Suppose then that $M$ and $N$ are Riemannian manifolds. For a linear map $L \colon T_x M \to T_y N$, we denote by $\abs{L}$ the \emph{operator norm} of $L$, defined by
\[
	\abs{L} = \sup \{ \abs{L(v)} : \abs{v} = 1 \}.
\]
Moreover, we define the \emph{$l$-operator} by
\[
	l(L) = \inf \{ \abs{L(v)} : \abs{v} = 1 \}.
\]
Suppose then that $f \in C^1(M, N)$. For the lack of a better term, we call $f$ \emph{infinitesimally $L$-Lipschitz} if
\begin{equation}\label{eq:lipschitz_D}
	\abs{Df(x)} \leq L
\end{equation}
for all $x \in M$. Moreover, we call $f$ \emph{infinitesimally $L$-bilipschitz} if it is injective, infinitesimally $L$-lipschitz, and satisfies
\begin{equation}\label{eq:bilipschitz_D}
	l(Df(x)) \geq L^{-1}.
\end{equation}
for all $x \in M$.

The metric space definition is the traditional definition of Lipschitz and bilipschitz maps. However, in practice we use the infinitesimal definition for essentially all of this text. Hence, we need to link these two definitions, which we accomplish by sketching the proof of the following theorem.

\begin{thm}\label{thm:C1_lip_characterization}
	Let $M$ and $N$ be connected Riemannian manifolds, let $f \in C^1(M, N)$, and let $L > 0$. Then $f$ is $L$-Lipschitz with respect to the Riemannian distance if and only if $f$ is infinitesimally $L$-Lipschitz. Moreover, if $M$ and $N$ are of the same dimension, then $f$ is $L$-bilipschitz with respect to the Riemannian distance if and only if $f$ is infinitesimally $L$-bilipschitz.
\end{thm}

Before starting the proof proper, we sketch the proof of a lemma on the existence of infinitesimally bilipschitz charts with a small constant $L$.

\begin{lemma}\label{lem:infi_bilipschitz_charts}
	Let $M$ be a smooth Riemannian $n$-manifold, and let $x \in M$. Then, for every $L > 1$, there exists a domain $U \subset M$ containing $x$ and a $C^\infty$-smooth infinitesimally $L$-bilipschitz chart $\phi \colon U \to \R^n$, where $\phi U$ is an Euclidean ball centered around $\phi(x)$.
\end{lemma}
\begin{proof}[Sketch of proof]
	For the exponential map $\exp_x \colon B_{T_x M}(0, r) \to M$, we have that $D \exp_x(0)$ is isometric. Therefore, $\abs{D \exp_x(0)} = l(D \exp_x(0)) = 1$. Since $\exp_x$ is smooth in a neighborhood of $0$, we may by selecting a sufficiently small $r$ assume that $\exp_x$ is infinitesimally $L$-bilipschitz in $B_{T_x M}(0, r)$. We may thus precompose $\exp_x$ with a linear isometry $I \colon \R^n \to T_x M$ and take the inverse of the resulting map, yielding the desired chart $\phi$.
\end{proof}

\begin{proof}[Sketch of proof of Theorem \ref{thm:C1_lip_characterization}]
	There are four implications to check.
	
	\emph{Infinitesimally $L$-Lipschitz implies $L$-Lipschitz:} If $f$ satisfies \eqref{eq:lipschitz_D}, then for any $C^1$-smooth path $\gamma$, we may estimate $\abs{(f \circ \gamma)'(t)} = \abs{Df(\gamma(t)) \gamma'(t)} \leq \abs{Df(\gamma(t))}\abs{\gamma'(t)} \leq L \abs{\gamma'(t)}$. It then follows that $f$ is $L$-Lipschitz in the Riemannian distance.
	
	\emph{Infinitesimally $L$-bilipschitz implies $L$-bilipschitz:} By the previous part, $f$ is $L$-Lipschitz. Since $l(Df(x)) > 0$ for every $x \in M$, and since $M$ and $N$ are of the same dimension, we have by the inverse function theorem that $f$ has a local $C^1$ inverse. Since $f$ is assumed injective, we therefore have a global $C^1$ inverse $f^{-1} \colon fM \to M$ for which $D(f^{-1})(f(x)) = (Df(x))^{-1}$. Now we have $\abs{D(f^{-1})(f(x))} = l(Df(x))^{-1} \leq L$ for every $x \in M$. By the previous part, $f^{-1}$ is $L$-Lipschitz, and consequently $f$ is $L$-bilipschitz.
	
	\emph{$L$-Lipschitz implies infinitesimally $L$-Lipschitz:} Let $x \in M$. We use Lemma \ref{lem:infi_bilipschitz_charts} to select infinitesimally $(1+\eps)$-bilipschitz smooth charts $\phi \colon U \to B_n(0, r)$ and $\psi \colon V \to B_m(0, s)$ such that $x \in U$ and $fU \subset V$. By the previous two parts, the charts are $L$-bilipschitz when the Euclidean balls are equipped with the Riemannian distance. However, since the balls are convex, their Riemannian distance coincides with the Euclidean metric. Hence, $g = \psi \circ f \circ \phi^{-1}$ is a $C^1$-smooth $(1+\eps)^2 L$-Lipschitz map between Euclidean balls.
	
	We now obtain for every unit vector $v$ of $\R^n$ the directional derivative estimate $\abs{\partial_v g(0)} = \lim_{t \to 0^+} \abs{g(tv) - g(0)}/t \leq (1+\eps)^2 L$. Consequently, $\abs{Dg(0)} \leq (1+\eps)^2 L$. We then further conclude by the infinitesimal $(1+\eps)$-bilipschitzness of $\phi$ and $\psi$ that $\abs{Df(x)} \leq l(D\psi(f(x)))^{-1} \abs{Dg(0)} \abs{D\phi(x)} \leq (1+\eps)^4 L$. Letting $\eps \to 0$ yields the result.
	
	\emph{$L$-bilipschitz implies infinitesimally $L$-bilipschitz:} As a bilipschitz map $f$ is injective, and we've already shown that $f$ is infinitesimally $L$-Lipschitz. Since $M$ and $N$ have the same dimension, $f$ is open by invariance of domain. Hence, $f^{-1} \colon fM \to M$ is an $L$-Lipschitz map between connected manifolds, and therefore by the previous part, $f^{-1}$ is infinitesimally $L$-lipschitz. Now, $l(Df(x)) = \abs{D(f^{-1}) (f(x))}^{-1} \geq L^{-1}$ for every $x$, concluding the proof.
\end{proof}

From now on, we will stick to the terms Lipschitz and bilipschitz, even when we're using the infinitesimal properties provided by Theorem \ref{thm:C1_lip_characterization}. We also record the following version of Lemma \ref{lem:infi_bilipschitz_charts}, which will be used repeatedly throughout the text.

\begin{thm}\label{thm:bilipschitz_charts}
	Let $M$ be a smooth Riemannian $n$-manifold, and let $x \in M$. Then, for every $L > 1$, there exists a domain $U \subset M$ containing $x$ and a $C^\infty$-smooth $L$-bilipschitz chart $\phi \colon U \to \R^n$, where $\phi U$ is an Euclidean ball centered around $\phi(x)$. Moreover, if $M$ is oriented, we may select the chart $\phi$ to be positively oriented.
\end{thm}
\begin{proof}
	The bilipschitz chart part follows from Lemma \ref{lem:infi_bilipschitz_charts} and Theorem \ref{thm:C1_lip_characterization}, along with the fact that the Euclidean and Riemannian distances coincide on Euclidean balls. If $M$ is oriented, the resulting chart is either orientation preserving or orientation reversing; in the latter case we may post-compose the chart with a reflection to obtain an orientation preserving one.
\end{proof}

Finally, we discuss how $\abs{\cdot}$ and $l(\cdot)$ give estimates for the Jacobian of a map. This is particularly useful when dealing with bilipschitz maps, but is also used later on in connection with quasiregular maps.

Namely, let $M$ and $N$ be Riemannian $n$-mani\-folds. Let $x \in M$, let $y \in N$, and let $L \colon T_x M \to T_y N$ be a linear map. If we fix orthonormal ordered bases $(e_i) \subset T_x M$ and $(e_i') \subset T_y N$, we may compute $\det L$ with respect to $(e_i)$ and $(e_i')$. However, since no orientation is taken into account, the result is dependent on the choices of $(e_i)$ and $(e_i')$. Nevertheless, the choice of basis only affects the sign of the Jacobian. Hence, the quantity $\abs{\det L}$ obtained is independent on the choice of $(e_i)$ and $(e_i')$.

Recall that, by the singular value decomposition theorem of matrices, there exist orthonormal bases $(e_i)$ and $(e_i')$ of $T_x M$ and $T_{f(x)} N$, respectively, such that for every index $i \in \{1, \dots, n\}$ we have $L e_i = \sigma_i e_i'$ for some $\sigma_i \in [0, \infty)$. We may also assume that $\sigma_1 \geq \dots \geq \sigma_n$. It follows that $\abs{L} = \sigma_1$, $l(L) = \sigma_n$, and $\abs{\det L} = \sigma_1 \cdots \sigma_n$. 

Hence, we obtain the estimate
\[
	l(L)^n \leq \abs{\det L} \leq \abs{L}^n.
\]
In particular, if $f \in C^1(M, N)$, then
\begin{equation}\label{eq:l_jacob_opnorm_ineq}
	l(Df(x))^n \leq \abs{J_f(x)} \leq \abs{Df(x)}^n
\end{equation}
for every $x \in M$. If $M$ and $N$ are connected and $f$ is Lipschitz, we therefore have that 
\begin{equation}\label{eq:lipschitz_J}
\abs{J_f(x)} \leq L^n.
\end{equation}
for every $x \in M$. Moreover, if $f$ is also $L$-bilipschitz, then also
\begin{equation}\label{eq:bilipschitz_J}
\abs{J_f(x)} \geq L^{-n}.
\end{equation}
for every $x \in M$.

\subsection{The Lusin conditions}

Let $f \colon U \to V$ be any function between domains $U, V \subset \R^n$. We say that $f$ satisfies the \emph{Lusin $(N)$ condition}, if $fE \subset V$ is a set of (Lebesgue) measure zero whenever $E \subset U$ is a set of measure zero. Similarly, we say that $f$ satisfies the \emph{Lusin $(N^{-1})$ condition} if $f^{-1}F \subset U$ is a set of measure zero whenever $F \subset V$ is a set of measure zero.

Next, we show a major property of the Lusin conditions: that they imply the preservation of measurable sets. 

\begin{lemma}\label{lem:lusin_preserves_measurable}
	Suppose that $f \colon U \to V$ is continuous and satisfies the Lusin $(N)$ condition. Then $fE$ is measurable for every measurable $E \subset U$.
	
	Similarly, suppose that $f \colon U \to V$ is continuous and satisfies the Lusin $(N^{-1})$ condition. Then $f^{-1}F$ is measurable for every measurable $F \subset V$.
\end{lemma}
\begin{proof}
	For the first claim, let $E \subset U$ be measurable. Then $E = E_0 \cup E_1$, where $E_1$ is a $F_\sigma$ set, that is, a countable union of closed sets, and $E_0$ has measure zero. Since $f$ is continuous, it maps compact sets to compact sets. Therefore, since every closed subset of $U$ is a countable union of compact sets, $fE_1$ is $F_\sigma$, and therefore measurable. Since $f$ satisfies the Lusin $(N)$ condition, $fE_0$ is measurable. Consequently, $fE = fE_1 \cup fE_0$ is measurable.
	
	For the second claim, we proceed similarly. Let $F \subset V$ be measurable, in which case $F = F_0 \cup F_1$ where $F_1$ is a $F_\sigma$-set and $F_0$ is of measure zero. Since $f$ is continuous, $f^{-1} C$ is closed for every closed $C \subset V$, and hence $f^{-1} F_1$ is $F_\sigma$. Moreover, by the Lusin $(N^{-1})$-condition, $f^{-1} F_0$ is measurable, and the claim therefore follows.
\end{proof}

We observe that Lipschitz maps between Euclidean domains satisfy the Lusin $(N)$ condition. Indeed, if $f \colon U \to V$ is $L$-lipschitz, then for every $p \in (0, \infty)$ and every set $E \subset U$, we have $\hausd^p(fE) \leq L^p \hausd^p(E)$, where $\hausd^p$ is the $p$-dimensional Hausdorff outer measure. Since $\hausd^n$ is equivalent with the Lebesgue (outer) measure $m_n$ on Euclidean domains, the Lusin $(N)$ condition for Lipschitz maps follows.

From this observation, it also follows that locally Lipschitz maps satisfy the Lusin $(N)$ condition. Consequently, we obtain that $C^1$-maps satisfy the Lusin $(N)$ condition, as they are locally Lipschitz. In particular, if $f \colon U \to V$ is a $C^1$ diffeomorphism, then it satisfies both the Lusin $(N)$ and $(N^{-1})$ conditions, and consequently by Lemma \ref{lem:lusin_preserves_measurable} preserves measurable sets.

\subsection{Measurable sets on Riemannian manifolds}

By the previous discussion, we may define Lebesgue measurable sets on a smooth manifold as follows.

\begin{defn}\label{def:manifold_meas_set}
	Let $M$ be a smooth differential $n$-manifold without boundary, and let $E \subset M$ be a set. Since Lebesgue measurable sets are preserved in diffeomorphisms, the following two conditions are equivalent.
	\begin{itemize}
		\item For every chart $\phi \colon U \to \R^n$ on $M$, the set $\phi (E \cap U) \subset \R^n$ is Lebesgue measurable.
		\item For every $x \in M$, there exists a chart $\phi \colon U \to \R^n$ on $M$ for which $x \in U$ and $\phi (E \cap U) \subset \R^n$ is Lebesgue measurable.
	\end{itemize}
	Hence, if either of these conditions hold, we say that $E \subset M$ is (Lebesgue) \emph{measurable}.
\end{defn}

With the definition of measurable sets, we also obtain a definition for measurable functions with a manifold domain. Recall that, if $U \subset \R^n$ is a domain and $X$ is a topological space, then a function $f \colon U \to X$ is called (Lebesgue) measurable if $f^{-1} U$ is (Lebesgue) measurable for every open $U \subset X$.

\begin{defn}\label{def:manifold_meas_function}
	Let $M$ be a smooth differential $n$-manifold without boundary, let $X$ be a topological space, and let $f \colon M \to X$ be a function. It follows from Definition \ref{def:manifold_meas_set} that the following conditions are equivalent.
	\begin{itemize}
		\item For every open set $U \subset X$, the set $f^{-1}U \subset M$ is measurable.
		\item For every chart $\phi \colon U \to \R^n$ on $M$, the map $f \circ \phi^{-1} \colon \phi U \to X$ is measurable.
		\item For every $x \in M$, there exists a chart $\phi \colon U \to \R^n$ on $M$ for which $x \in U$ and $f \circ \phi^{-1} \colon \phi U \to X$ is measurable.
	\end{itemize}
	If any of these conditions hold, we say that $f$ is (Lebesgue) \emph{measurable}.
\end{defn}

Moreover, we also obtain a several equivalent definitions for measurable differential forms. We leave the verification that these are equivalent to the reader: it essentially follows from the fact that $\wedge^k T^* M$ has an atlas consisting of pull-back maps $(\phi^{-1})^* \colon \wedge^k T^* U \to \wedge^k T^* \phi(U) \cong \phi(U) \times \R^{\binom{n}{k}}$ where $\phi \colon U \to \R^n$ is a chart of $M$.

\begin{defn}\label{def:manifold_meas_form}
	Let $M$ be a smooth differential $n$-manifold without boundary, and let $\omega$ be a differential $k$-form on $M$, $k \in \{0, \dots, n\}$. Then $\omega$ is a section $M \to \wedge^k T^* M$. It then follows that the following conditions are equivalent.
	\begin{itemize}
		\item The map $\omega \colon M \to \wedge^k T^* M$ is measurable.
		\item For every chart $\phi \colon U \to \R^n$ on $M$, the pull-back $(\phi^{-1})^* \omega$ has Lebesgue measurable coordinate functions when written in the standard basis of $\wedge^k \R^n$.
		\item For every $x \in M$, there exists a chart $\phi \colon U \to \R^n$ on $M$ for which $x \in U$, and the pull-back $(\phi^{-1})^* \omega$ has Lebesgue measurable coordinate functions when written in the standard basis of $\wedge^k \R^n$.
	\end{itemize}
	If any of these conditions hold, we say that the $k$-form $\omega$ is (Lebesgue) \emph{measurable}. We denote the vector space of all measurable $k$-forms on $M$ by $\Gamma(\wedge^k M)$.
\end{defn}

Before moving on to discuss integration, we point out one more characterization of measurable differential forms which turns out particularly useful. Namely, note that a differential $k$-form $\omega \colon M \to \wedge^k T^* M$ also induces a map $\omega \colon \wedge^k TM \to \R$.

\begin{lemma}\label{lem:two_measurability_defs_for_forms}
	A differential form $\omega \colon M \to \wedge^k T^* M$ is measurable if and only if it is measurable as a map $\omega \colon \wedge^k TM \to \R$.
\end{lemma}
\begin{proof}[Sketch of proof]
	Consider the case where $M$ is a domain in $\R^n$. In this case, $\wedge^k TM \cong M \times \R^{\binom{n}{k}}$ canonically, and $\omega$ can be written as $\omega = \sum_I \omega_I \eps_I$, where $\eps_I$ are standard basis elements of $k$-forms in $\R^n$. The ``only if'' -direction follows since the map $M \times \R^{\binom{n}{k}} \to \R$ induced by $\omega$ is given by $(x, y) \mapsto \sum_I \omega_I(x) y_I$; this is clearly measurable if $\omega_I \colon M \to \R$ are measurable. 
	
	For the ``if'' -direction, we use the fact that for every measurable $g \colon M \times \R^k \to \R$, the slice $g(\cdot, y) \colon M \to \R$ is measurable for a.e.\ $y \in \R^k$. Note that this fact is usually considered a part of the Fubini--Tonelli theorem. Hence, $x \mapsto \sum_I \omega_I(x) y_I$ is measurable for almost every $y \in \R^{\binom{n}{k}}$. Since $\R^{\binom{n}{k}}$ is spanned by its every full measure subset, taking linear combinations of the maps $x \mapsto \sum_I \omega_I(x) y_I$ yields that every $\omega_I$ is measurable, and hence $\omega$ is measurable.
	
	For the general case, the result now holds for all pull-backs $(\phi^{-1})^* \omega$ where $\phi \colon U \to \R^n$ is a chart. Since $\wedge^k D\phi^{1} \colon \wedge^k T\phi U \to \wedge^k TU$ is a $C^1$ diffeomorphism, it satisfies both Lusin conditions, and hence $\omega\vert_U \colon \wedge^k TU \to \R$ is measurable if and only if $(\phi^{-1})^* \omega = \omega\vert_U \circ \wedge^k D\phi^{1}$ is measurable. The claim then follows by considering a cover of $M$ with charts.
\end{proof}

\subsection{Integration and $L^p$ spaces}

Suppose then that $M$ is a Riemannian $n$-manifold, and let $\phi \colon U \to \R^n$ be a smooth chart on $M$. Without the assumption that $M$ is orientable, we cannot determine the Jacobian $J_{\phi^{-1}}$ of $\phi^{-1}$. However, we may still define the absolute value of the Jacobian $\abs{J_{\phi^{-1}}} \in C(\phi U, [0, \infty))$. Indeed, the range $U$ of $\phi^{-1}$ is an orientable manifold, and therefore there exist exactly two volume forms on it corresponding to the Riemannian metric $\ip{\cdot}{\cdot}$, each yielding the same Jacobian with a different sign.

Now, let $f \colon M \to [0, \infty]$ be a non-negative measurable function. By using the chart $\phi$, we may define the Lebesgue integral of $f$ over $U$ by
\[
	\int_U f \dd m_n = \int_{\phi U} (f \circ \phi^{-1}) \abs{J_{\phi^{-1}}} \dd m_n.
\]
By the Euclidean change of variables formula, the expression $\int_U f \dd m_n$ is in fact independent on the choice of chart $\phi$. Now, a unique integral
\[
	\int_M f \dd m_n \in [0, \infty]
\]
is obtained by using a locally finite partition of unity. Hence, we have obtained a Lebesgue measure $m_n$ on our Riemannian manifold $M$; we note that this measure is also known as the \emph{Riemannian measure} of $M$.

Consequently, we obtain Banach spaces $L^p(M)$ of $L^p$-integrable measurable real functions on $M$, where $1 \leq p \leq \infty$. Note that two measurable functions $f, g \colon M \to \R$ are considered to represent the same element in $L^p(M)$ if they differ only in a set of $m_n$-measure zero. We denote the $L^p$-norm of a measurable function $f \colon M \to \R$ by $\norm{f}_p$. We also obtain the spaces $L^p(M, \R^m)$ of functions with $L^p$-integrable coordinate functions.

More specifically to the manifold case, the Riemannian metric also allows us to define $L^p$-spaces of differential forms. We consider two $k$-forms $\omega, \omega' \in \Gamma(\wedge^k M)$ equivalent if, when considered as as maps $M \to \wedge^k T^*M$, they are equal outside a set of $m_n$-measure zero. The space $L^p(\wedge^k M)$ then consists of equivalence classes of $\omega \in \Gamma(\wedge^k M)$ satisfying $\abs{\omega} \in L^p(M)$, where $\abs{\omega}$ denotes the measurable function $M \to \R$ given by $x \mapsto \abs{\omega_x}_x$ for $x \in M$. The $L^p$-norm of a differential $k$-form $\omega \in \Gamma(\wedge^k M)$ is hence given by 
\[
	\norm{\omega}_p = \left( \int_M \abs{\omega}^p \dd m_n \right)^\frac{1}{p}
\]
for $1 \leq p < \infty$, and by
\[
	\norm{\omega}_\infty = \esssup_{x \in M} \abs{\omega_x}_x.
\]

Integration of functions with respect to $m_n$ can be extended to functions $f \in L^1(M)$ by considering positive and negative parts. Moreover, if $M$ is oriented, then every $\omega \in \Gamma(\wedge^n M)$ is of the form $\omega = g \vol_M$ for some measurable real function $g$. In particular, $\abs{\omega} = \abs{g}$, and therefore $\norm{\omega}_1 = \norm{g}_1$. Hence, if $\omega \in L^1(\wedge^n M)$, then we may define
\[
	\int_M \omega = \int_M g \dd m_n.
\]
In conclusion, we may integrate $L^1$-integrable measurable $n$-forms on an oriented $n$-manifold.

%%%%%%%%%%%%%%%%%%%%%%%%%%%%%%%%%%%%%%%%%%%%%%%%%%%%%%%%%%%%%%%%%%%%%%%%%%%%%%%%%%%%%%%%%%%%%

\section{First order Euclidean Sobolev spaces}\label{sect:eucl_sobolev_prelims}

Before starting to discuss first order Sobolev spaces on manifolds, we first review the necessary parts of the theory in Euclidean spaces. For most of the basic theory we only state the results, referring the reader to one of the many introductory texts on the topic for further details; e.g.\ \cite{Adams-Fournier_Sobolev}. However, we will examine chain rules with $C^1$-differentiable Lipschitz maps in detail, due to their importance to us. 

\subsection{Sobolev maps in $\R^n$}

Let $U \subset \R^n$ be open set, and let $f \in L^1_\loc(U)$. Then for $i=1, \ldots, n$, a measurable function $\partial_i f \in L^1_\loc(U)$ is the \emph{$i$-th weak partial derivative of $f$} if, for every $\phi \in C^\infty_0(U)$ we have
\[
	\int_U (\partial_i f)\cdot \phi \dd m_n = -\int_U  f \cdot (\partial_i \phi) \dd m_n.
\]
Suppose then that $f \in L^1_\loc(U, \R^m)$ is such that all weak partial derivatives $\partial_i f_j$ of $f$ exist, where $f_j$ stand for the coordinate functions of $f$. Then $f$ has a \emph{weak derivative} $\nabla f$, which is the measurable matrix-valued function
\[
	\nabla f(x) = \begin{bmatrix}
		\partial_1 f_1(x) & \partial_2 f_1(x) & \dots & \partial_n f_1(x)\\
		\partial_1 f_2(x) & \partial_2 f_2(x) & \dots & \partial_n f_2(x)\\
		\vdots & \vdots & \ddots & \vdots\\
		\partial_1 f_m(x) & \partial_2 f_m(x) & \dots & \partial_n f_m(x)
	\end{bmatrix}.
\]
Weak derivatives are unique outside a set of measure zero.

For the purpose of later generalizing to the setting of manifolds, we define several alternate interpretations of the weak derivative with a manifold flair. First, recall that for every $x \in \R^k$, there exists a canonical linear isomorphism $\iota_{k, x} \colon T_{x} \R^k \to \R^k$ which for every $j \in \{1, \dots, k\}$ maps $\partial_j \in T_x \R^k$ to $e_j \in \R^k$. These maps $\iota_{k, x}$ can then be combined into a single map $\iota_{k} \colon T\R^k \to \R^k$.

Then, given a weakly differentiable $f \colon U \to \R^n$ where $U \subset \R^m$, we define the \emph{weak tangent map} $Df \colon TU \to T\R^m$ of $f$, consisting of maps $Df(x) \colon T_x U \to T_{f(x)}\R^m$, by $Df(x) = \iota_{n, f(x)}^{-1} \circ \nabla f(x) \circ \iota_{m, x}$. Moreover, we define the \emph{weak differential} $df \colon TU \to \R^n$ of $f$, consisting of maps $df_x \colon T_x U \to \R^n$, by $df_x = \nabla f(x) \circ \iota_{m, x}$. The following commutative diagram sums up our definitions.
\[\begin{tikzcd}[sep=large]
	T_x\R^m \ar[r, "Df(x)"] \ar[dr, "df_x"]  \ar[d, "\cong"]
	& T_{f(x)}\R^n \ar[d, "\cong"]\\
	\R^m \ar[r, "\nabla f(x)"]
	& \R^n 
\end{tikzcd}\]

The maps $Df(x)$ are characterized by the formula
\[
	Df(x)(\partial_i)_x = \sum_{j=1}^m \left(\partial_i f_j(x) \right) (\partial_j)_{f(x)}.
\]
for almost every $x \in U$. We use the map $df$ mainly in the case $n = 1$, in which case $df$ is a measurable differential 1-form given by
\[
	df = \sum_{i=1}^n (\partial_i f) dx^i.
\]
We also state alternate definitions for $\nabla f$ and $df$ in the following lemma, the proof of which is a straightforward verification left to the reader.

\begin{lemma}\label{lem:weak_derivative_characterizations}
	Suppose that $U \subset \R^n$ and $f \in L^1_\loc(U, \R^m)$. Then $\nabla f$ is a weak derivative of $f$ if and only if
	\[
		\int_U \phi \nabla f \dd m_n = -\int_U f \nabla \phi \dd m_n
	\]
	for all $\phi \in C^\infty_0(U, \R^m)$.
	
	Furthermore, suppose that $m = 1$. Then $df$ is a weak differential of $f$ if and only if
	\[
		\int_U df \wedge \omega = - \int_U fd\omega
	\]
	for all $\omega \in C^\infty_0(\wedge^{n-1} U)$.
\end{lemma}

Suppose $1 \leq p < \infty$. A function $f \colon U \to \R^m$ with a weak derivative is in the \emph{Sobolev space} $W^{1, p}(U, \R^m)$ if
\begin{equation}\label{eq:sobolev_definition_1}
	\int_U \left(\abs{f(x)}^p + \abs{\nabla f(x)}^p\right) \dd m_n(x) < \infty.
\end{equation}
Here, $\abs{\nabla f(x)}$ stands for the operator norm of $\nabla f(x)$, which we note to be equal to $\abs{Df(x)}$. We then define a norm on $W^{1,p}(U, \R^m)$ given by
\[
	\norm{f}_{1,p} = 
	\left( \int_U \left(\abs{f(x)}^p + \abs{\nabla f(x)}^p\right) \dd m_n(x) \right)^\frac{1}{p}
\]
for $f \in W^{1,p}(U, \R^m)$. For the case $p = \infty$, we say that $f \in W^{1, \infty}(U, \R^m)$ if $f$ and all of its weak partial derivatives are essentially bounded, and the norm $\norm{\cdot}_{1,\infty}$ is then defined by 
\[
	\norm{f}_{1,\infty} = \esssup_{x \in U} \left( \abs{f(x)} + \abs{\nabla f(x)} \right).
\] 

We use the shorthand $W^{1,p}(U) = W^{1,p}(U, \R)$. It is easily seen that a function $f$ is in $W^{1,p}(U, \R^m)$ if and only if all its coordinate functions are in $W^{1,p}(U)$. If $V \subset \R^m$, we define the space $W^{1,p}(U, V)$ as the space of all $f \in W^{1,p}(U, \R^n)$ for which $f(x) \in V$ for almost every $x \in U$.

We also define the local versions of Sobolev spaces: if $f \colon U \to V \subset \R^m$ has all first order weak partial derivatives, then $f \in W^{1,p}_\loc(U, V)$ if $\abs{f}$ and $\abs{\nabla f}$ are locally $p$-integrable, or locally essentially bounded in the case $p = \infty$. We note that since weak partial derivatives are by assumption locally $L^1$, the space $W^{1,1}_\loc(U, V)$ is in fact the space of all weakly differentiable $f \in L^1_\loc(U, V)$.

\subsection{Classical properties of Euclidean Sobolev spaces}

We then recall several classical properties of Sobolev spaces, mainly referring to the literature for the proofs. First, we recall that Sobolev spaces are Banach spaces; see e.g.\ \cite[Theorem 3.3]{Adams-Fournier_Sobolev}.

\begin{thm}\label{thm:sobolev_eucl_banach}
	Let $1 \leq p \leq \infty$ and let $U \subset \R^n$ be a domain. Then the space $W^{1,p}(U, \R^m)$ with the norm $\norm{\cdot}_{1,p}$ is complete.
\end{thm}

Second is the \emph{Meyers--Serrin theorem} which states that, with the exception of the case $p = \infty$, smooth functions are dense in $W^{1,p}(U, \R^m)$; see e.g.\ \cite[Theorem 3.17]{Adams-Fournier_Sobolev}. The standard proof is by convolution approximation.

\begin{thm}\label{thm:sobolev_eucl_smooth_density}
	Let $1 \leq p < \infty$ and let $U \subset \R^n$ be a domain. Then the space $C^{\infty}(U, \R^m) \cap W^{1,p}(U, \R^m)$ is dense in $W^{1,p}(U, \R^m)$.
\end{thm}

Third, the weak derivative and Sobolev spaces behave well with respect to taking maximums. For a proof, see e.g.\ \cite[Theorem 1.20]{Heinonen-Kilpelainen-Martio_book}

\begin{thm}\label{thm:sobolev_eucl_maximum}
	Let $1 \leq p \leq \infty$, let $U \subset \R^n$ be a domain, and let $f, g \in W^{1,p}(U)$ (or $W^{1,p}_\loc(U)$). Then $\max\{f,g\} \in W^{1,p}(U)$ (or $W^{1,p}_\loc(U)$, respectively). Moreover, we have
	\[
		\nabla \max\{f,g\} = \begin{cases}
			\nabla f & \text{a.e.\ in }\{x \in U : f(x) \geq g(x) \},\\
			\nabla g & \text{a.e.\ in }\{x \in U : f(x) \leq g(x) \}.
		\end{cases}
	\]
	 In particular, $\nabla f = \nabla g$ almost everywhere in the set $\{x \in U : f(x) = g(x) \}$.
\end{thm}

Next we state several versions of Sobolev inequalities and embedding theorems which we need. We begin with a version of the \emph{Gagliardo--Nirenberg--Sobolev inequality} for compactly supported Sobolev functions. This inequality is usually stated in the form $\norm{u}_{s^*} \leq C(n, p) \norm{\nabla u}_s$, where $u \in C^\infty_0(\R^n)$, $s \in [1, n)$ and $s^* = (s^{-1} - n^{-1})^{-1}$; see e.g.\ \cite[Theorem 4.31]{Adams-Fournier_Sobolev}. The version we state here follows from this by smooth approximation, and by estimating $\norm{\nabla u}_s \leq [m_n(\spt u)]^{s^{-1} - p^{-1}} \norm{\nabla u}_p$ where $s$ is such that $s^* = q$.

\begin{thm}\label{thm:sobolev_eucl_inequality}
	Let $U \subset \R^n$ be a bounded domain, and let $p, q \in [1, \infty)$ be such that $1/q \geq 1/p - 1/n$.  Suppose that $u \in W^{1, p}(U, \R^m)$ and that $\spt u$ is compactly contained in $U$. Then
	\[
	\norm{u}_q \leq C(n, q) m_n(U)^{\frac{1}{q} - \frac{1}{p} + \frac{1}{n}} \norm{\nabla u}_p.
	\]
\end{thm}

We then give a local embedding theorem, which follows by first multiplying $u \in W^{1,p}_\loc(U, \R^m)$ with $\eta \in C^\infty_0(U)$ which is identically 1 on a neighborhood of a point of $U$, and by then using Theorem \ref{thm:sobolev_eucl_inequality}.

\begin{thm}\label{thm:sobolev_eucl_embedding_local}
	Let $U \subset \R^n$ be a domain, and let $p, q \in [1, \infty)$ be such that $1/q \geq 1/p - 1/n$. Then $W^{1,p}_\loc(U, \R^m) \subset L^{q}_\loc(U, \R^m)$.
\end{thm}

We point out a consequence of the embedding theorem \ref{thm:sobolev_eucl_embedding_local} which is of some importance to us later on.

\begin{cor}\label{cor:sobolev_base_exponent}
	Let $U \subset \R^n$ be a domain, and let $p \in [1, \infty)$. Suppose that $f \in W^{1,1}_\loc(U, \R^m)$ and $\abs{Df} \in L^p_\loc(U)$. Then $f \in W^{1,p}_\loc(U, \R^m)$.
\end{cor}
\begin{proof}
	We see by Theorem \ref{thm:sobolev_eucl_embedding_local} that $f \in L^{n/(n-1)}_\loc(U, \R^m)$. Since $\abs{Df} \in L^p_\loc(U)$, if $p \leq n/(n-1)$ then $f \in W^{1,p}_\loc(U, \R^m)$ and we're done. Otherwise $\abs{Df} \in L^p_\loc(U)$ implies $f \in W^{1, n/(n-1)}_\loc(U, \R^m)$. Now we repeat the use of Theorem \ref{thm:sobolev_eucl_embedding_local}, obtaining either that $f \in W^{1,p}_\loc(U, \R^m)$ or $f \in W^{1,n/(n-2)}_\loc(U, \R^m)$. Continuing this eventually gives us $f \in W^{1,p}_\loc(U, \R^m)$, at the very least after the step $f \in W^{1, n/1}_\loc(U, \R^m)$.
\end{proof}

Finally, we recall a version of the \emph{Rellich--Kondrachov theorem}, which gives for sufficiently nice domains a compact embedding of Sobolev spaces. For the proof, see e.g.\ \cite[Theorems 4.12 and 6.3]{Adams-Fournier_Sobolev}.

\begin{thm}\label{thm:sobolev_eucl_embedding_compact}
	Let $U \subset \R^n$ be a bounded domain with smooth boundary, and let $p, q \in [1, \infty)$ be such that $1/q \geq 1/p - 1/n$. Then $W^{1,p}(U, \R^m) \subset L^{q}(U, \R^m)$, where the embedding is continuous. 
	
	Moreover, if $1/q > 1/p - 1/n$, the embedding is compact: if $S$ is a bounded subset of $W^{1,p}(U, \R^m)$, then $S$ is precompact in $L^q(U, \R^m)$.
\end{thm}

\subsection{Pre- and postcomposition of Sobolev maps}

In order to define Sobolev spaces on manifolds, we require several chain rules for Sobolev maps. Due to their importance for the discussion, we present relatively detailed proofs for them.

Before the main result, we recall a lemma used in the proof of the chain rule for postcomposition. The lemma is a variant on the dominated convergence theorem.

\begin{lemma}\label{lem:dom_convergence_variant}
	Let $U \subset \R^n$ be open and $1 \leq p < \infty$. Suppose that the functions $u_j, v_j, u, v \in L^p(U, \R^m)$ are such that $u_j \to u$ almost everywhere, $v_j \to v$ in $L^p(U, \R^m)$, and $\abs{u_j} \leq \abs{v_j}$ almost everywhere. Then $u_j \to u$ in $L^p(U, \R^m)$.
\end{lemma}
\begin{proof}
	By Fatou's lemma,
	\[
		\int_U \liminf_{j \to \infty} (\abs{v_j}^p - \abs{u_j}^p) \dd m_n 
			\leq \liminf_{j \to \infty} \int_U (\abs{v_j}^p - \abs{u_j}^p) \dd m_n.
	\]
	Rearranging this estimate yields
	\begin{align*}
		&\limsup_{j \to \infty} \int_U \abs{u_j}^p \dd m_n\\
		&\qquad\leq \int_U \limsup_{j \to \infty} \abs{u_j}^p \dd m_n  
			+ \limsup_{j \to \infty} \int_U \abs{v_j}^p \dd m_n 
			- \int_U \liminf_{j \to \infty} \abs{v_j}^p \dd m_n\\
		&\qquad= \int_U \abs{u}^p \dd m_n
			+ \int_U \abs{v}^p \dd m_n
			- \int_U \abs{v}^p \dd m_n\\
		&\qquad= \int_U \abs{u}^p \dd m_n.
	\end{align*}
	Another use of Fatou's lemma also yields
	\[
		\int_U \abs{u}^p \dd m_n = \int_U \liminf_{j \to \infty} \abs{u_j}^p \dd m_n
		\leq \liminf_{j \to \infty} \int_U \abs{u_j}^p \dd m_n.
	\]
	Hence, we now have that $u_j \to u$ almost everywhere and $\norm{u_j}_p \to \norm{u}_p$. It is a classical theorem of Riesz that this implies $u_j \to u$ in $L^p$ (the proof of this is in fact yet another application of Fatou's lemma, this time to the nonnegative function $2^p(\abs{u}^p + \abs{u_j}^p) - \abs{u-u_j}^p$).
\end{proof}

We begin by giving the Sobolev chain rule for precomposition.

\begin{lemma}\label{lem:sobolev_precomposition}
	Let $U, U' \subset \R^n$ be open domains, and let $1 \leq p < \infty$. Suppose that $f \in W^{1,p}(U, \R^m)$ and that $h \colon U' \to U$ is a $C^1$-smooth $L$-bilipschitz map.
	
	Then $f \circ h \in W^{1,p}(U', \R^m)$ with
	\[
		\norm{f \circ h}_{1,p} \leq L^{1 + \frac{n}{p}} \norm{f}_{1,p},
	\]
	and 
	\[
	 	D(f \circ h) = Df \circ Dh
	\]
	almost everywhere on $U'$.
\end{lemma}
\begin{proof}
	Let $f_j \in C^\infty(U,\R^m)$ be a sequence of smooth functions for which $\norm{f_j - f}_{1,p} \to 0$. Note that $\nabla f \circ h$ is measurable since $h$ satisfies the Lusin conditions. We first obtain by \eqref{eq:bilipschitz_J} that
	\begin{align*}
		\int_{U'} \abs{f_j \circ h}^p \dd m_n
		&\leq L^n \int_{U'} \left(\abs{f_j}^p \circ h\right) J_h \dd m_n
		\\&= L^n \int_{hU'} \abs{f_j}^p \dd m_n,
	\end{align*}
	from which we conclude that $f_j \circ h \in L^{1,p}(U', \R^m)$. Similarly, we have by \eqref{eq:lipschitz_D} and \eqref{eq:bilipschitz_J} that
	\begin{align*}
		\int_{U'} \abs{(\nabla f_j \circ h) \nabla h}^p \dd m_n
		&\leq \int_{U'} \left(\abs{\nabla f_j}^p \circ h\right) \abs{\nabla h}^p \dd m_n
		\\&\leq L^{p+n} \int_{U'} \left(\abs{\nabla f_j}^p \circ h\right) J_h \dd m_n
		\\&= L^{p+n} \int_{hU'} \abs{\nabla f_j}^p \dd m_n.
	\end{align*}
	Hence, $f_j \circ h \in W^{1,p}(U', \R^m)$.
	
	By the same computation as previously, we also obtain
	\begin{align*}
		\int_{U'} \abs{f \circ h - f_j \circ h}^p \dd m_n
		&\leq L^n \int_{hU'} \abs{f-f_j}^p \dd m_n
	\end{align*}
	and
	\begin{align*}
		\int_{U'} \abs{(\nabla f \circ h) \nabla h - (\nabla f_j \circ h) \nabla h}^p \dd m_n
		&\leq L^{p+n} \int_{hU'} \abs{\nabla f- \nabla f_j}^p \dd m_n.
	\end{align*}
	Hence, $(f_j \circ h)$ is Cauchy in $W^{1,p}(U', \R^m)$, $f_j \circ h \to f \circ h$ in $L^p$, and moreover $D(f_j \circ h) \to Df \circ Dh$ in $L^p$. Since $W^{1,p}(V, \R^m)$ is complete, we get that $f \circ h \in W^{1,p}(V, \R^m)$ and $D(f \circ h) = Df \circ Dh$. Also, we see from the above that the comparison constant between norms is $(\max\{L^n, L^{p+n} \})^{1/p} = L^{1+n/p}$.
\end{proof}

The postcomposition version is as follows.

\begin{lemma}\label{lem:sobolev_postcomposition}
	Let $U \subset \R^n$ and $V \subset \R^m$ be open domains, and let $1 \leq p < \infty$. Suppose that $f \in W^{1,p}(U, V)$ and that $h \colon V \to \R^k$ is a $C^1$-smooth $L$-Lipschitz map. Suppose furthermore that one of the following holds:
	\begin{itemize}
		\item $U$ has finite volume;
		\item $0 \in V$ and $h(0) = 0$.
	\end{itemize}
	
	Then $h \circ f \in W^{1,p}(U, \R^k)$, and 
	\[
		D(h \circ f) = Dh \circ Df
	\]
	almost everywhere on $U$.
\end{lemma}
\begin{proof}
	Let again $f_j \in C^\infty(U,\R^m)$ be a sequence of smooth functions for which $\norm{f_j - f}_{1,p} \to 0$. Consider first the case $h(0) = 0$. We note that $\abs{h \circ f(x)} = \abs{h \circ f(x) - h(0)} \leq L \abs{f(x)}$, from which it follows that $h \circ f \in L^p(U, \R^m)$. We also get that
	\begin{align*}
		\int_U \abs{h \circ f - h \circ f_j}^p \dd m_n
		&\leq L^p \int_U \abs{f-f_j}^p \dd m_n.
	\end{align*}
	For the weak differentials, we first note using \eqref{eq:lipschitz_D} that
	\begin{align*}
		\abs{(\nabla h \circ f_j) \nabla f_j}
		& \leq (\abs{\nabla h} \circ f_j) \abs{\nabla f_j}\\
		& \leq L \abs{\nabla f_j}
	\end{align*}
	Hence, $\nabla(h \circ f_j) = (\nabla h \circ f_j) \nabla f_j$ is $L^p$, and we therefore have that $h \circ f_j \in W^{1,p}(U, \R^m)$. It remains to show that $(\nabla h \circ f_j) \nabla f_j \to (\nabla h \circ f) \nabla f$ in $L^p$, after which we obtain the result as before.
	
	For this, note first that $(\nabla h \circ f_j) \nabla f_j \to (\nabla h \circ f) \nabla f$ almost everywhere. Indeed, for almost every $x$, the $W^{1,p}$-convergence $f_j \to f$ yields that $f_j(x) \to f(x)$ and $\nabla f_j(x) \to \nabla f(x)$. The desired convergence then holds for such $x$ by continuity of $\nabla h$. Moreover, since $L \nabla f_j \to L \nabla f$ in $L^p$, we may use Lemma \ref{lem:dom_convergence_variant} with $u_j = (\nabla h \circ f_j) \nabla f_j$, $u = (\nabla h \circ f) \nabla f$, $v_j = L \nabla f_j$ and $v = L \nabla f$. With that, the desired $L^p$-convergence of differentials follows.
	
	The remaining case is that $U$ has finite volume. For this, we fix a $x \in U$, and define $\tilde{f} \colon U \to \tilde{V}$ by $\tilde{f} = f - f(x)$. Then we define $\tilde{h} \colon \tilde{V} \to \tilde{V}'$ by $\tilde{h}(y) = h(y+f(x)) - h(f(x))$. We may now apply the previous case to $\tilde{f}$ and $\tilde{h}$. Clearly $(\nabla h \circ f) \nabla f = (\nabla \tilde{h} \circ \tilde{f}) \nabla \tilde{f}$. Therefore, it remains to check $h \circ f \in L^p(U, \R^m)$, after which the claim follows. This follows from the fact that $h \circ f = \tilde{h} \circ \tilde{f} - h(f(x))$, where on the right side $\tilde{h} \circ \tilde{f}$ is in $L^p$, and the constant function $h(f(x))$ is in $L^p$ due to the finite volume of the domain of integration.
\end{proof}

A local version of the chain rules follows immediately as a corollary. Note that in the local case we can automatically assume the boundedness of $U$ when applying Lemma \ref{lem:sobolev_postcomposition}, and therefore dividing to cases is unnecessary.

\begin{cor}\label{cor:sobolev_composition_local}
	Let $U, U' \subset \R^n$ and $V \subset \R^m$, $V' \subset \R^k$ be open domains, and let $1 \leq p < \infty$. Suppose that $f \in W^{1,p}_\loc(U, V)$, $h_1 \colon U' \to U$ is a $C^1$-smooth locally bilipschitz map, and $h_2 \colon V \to V'$ is a $C^1$-smooth Lipschitz map.
	
	Then $h_2 \circ f \circ h_1 \in W^{1,p}_\loc(U', V')$ and 
	\[
		D(h_2 \circ f \circ h_1) = Dh_2 \circ Df \circ Dh_1
	\]
	almost everywhere on $U'$.
\end{cor}

%%%%%%%%%%%%%%%%%%%%%%%%%%%%%%%%%%%%%%%%%%%%%%%%%%%%%%%%%%%%%%%%%%%%%%%%%%%%%%%%%%%%%%%%%%%
\section{Sobolev spaces with manifold domain}\label{sect:sobolev_manifold_domain}

In this section, we generalize the discussion from the last section to the case where the domains of Sobolev maps are Riemannian manifolds, instead of being contained in $\R^n$.

\subsection{Sobolev maps from manifolds into $\R^k$}

Suppose then that $M$ is a Riemannian $n$-manifold without boundary. We wish to define weak derivatives on $M$, that is, to define the space $W^{1,1}_\loc(M, \R^k)$. The definition is as follows:
\begin{defn}
	Let $f \in L^1_\loc(M, \R^k)$. Then $f \in W^{1,1}_\loc(M, \R^k)$ if, for every $x \in M$, there exists a neighborhood $U \subset M$ of $x$ and a $C^1$-chart $\phi \colon U \to \R^n$ for which $f \circ \phi^{-1} \in W^{1,1}(\phi(U), \R^k)$.
\end{defn}

If $M$ is a domain in $\R^n$, we need to verify that our definition is equivalent with the Euclidean one. For that, let $x \in M$. If $f$ satisfies the Euclidean definition, we may pick a neighborhood $U$ of $x$ where $f\vert_U \in W^{1,1}(U, \R^k)$, and now the manifold definition holds by the identity chart on $U$. On the other hand, if the manifold definition holds, we find a neighborhood $U$ of $x$ and a chart $\phi \colon U \to \R^n$ for which $f \circ \phi^{-1} \in W^{1,1}(\phi(U), \R^k)$. Since $\phi$ is a diffeomorphism to its image, we may assume it is bilipschitz by further restricting $U$. Then $f \in W^{1,1}(U, \R^k)$ by Lemma \ref{lem:sobolev_precomposition}, and the Euclidean definition is easily seen to follow.

We now wish to define a weak derivative $Df \colon TM \to T\R^k$ for the function $f$. Suppose $x \in M$, and select a chart $\phi \colon U \to \R^n$ at $x$ such that $f \circ \phi^{-1} \in W^{1,1}(\phi(U))$. By further restricting $U$, we may also assume that $\phi$ is bilipschitz. Since $D\phi \colon T\phi(U) \to TU$ is a diffeomorphism, it satisfies the Lusin conditions, and we hence have a measurable map $Df\vert_U \colon TU \to T\R^n$ defined by $Df\vert_U = D(f \circ \phi^{-1}) \circ D\phi$.

We may then define $Df$ by taking a countable locally finite cover of $M$ by bilipschitz charts $(U_i, \phi_i)_{i\in \N}$, and setting $Df(x) = Df\vert_{U_0}(x)$ for $x \in U_0$, $Df(x) = Df\vert_{U_1}(x)$ for $x \in U_1 \setminus U_0$, $Df(x) = Df\vert_{U_2}(x)$ for $x \in U_2 \setminus (U_1 \cup U_0)$, etc. The resulting $Df$ is clearly measurable. To see the uniqueness of $Df$, suppose that $\phi \colon U \to \R^n$ and $\psi \colon V \to \R^n$ are both bilipschitz charts at $x$. Then by Lemma \ref{lem:sobolev_precomposition}, we conclude that
\begin{multline*}
Df\vert_U
= D(f \circ \phi^{-1}) \circ D\phi
= D(f \circ \phi^{-1}) \circ D\phi \circ D\psi^{-1} \circ D\psi\\
= D(f \circ \phi^{-1}) \circ D(\phi \circ \psi^{-1}) \circ D\psi
= D(f \circ \psi^{-1}) \circ D\psi
= Df\vert_V
\end{multline*}
almost everywhere on $U \cap V$. It follows that $Df$ is unique up to a set of measure zero.

Recall the continuous map $\iota_k \colon T\R^k \to \R^k$ given by the canonical identifications $T_x \R^n \cong \R^n$. We may then for any $f \in W^{1,1}_\loc(M, \R^k)$ define a weak differential $df \colon TM \to \R^n$ by $df = \iota_k \circ df$. Since $\iota_k$ is continuous, $df$ is a measurable differential 1-form by Lemma \ref{lem:two_measurability_defs_for_forms}. Hence, the following part of the diagram of different weak derivatives survives to the manifold domain setting.
\[\begin{tikzcd}[sep=large]
	T_x M \ar[r, "Df(x)"] \ar[dr, "df_x"]
	& T_{f(x)}\R^k \ar[d, "\cong"]\\
	& \R^k 
\end{tikzcd}\]

Note that $\abs{Df}$ is in $L^1_\loc(M)$. Indeed, this follows from the fact that, given an $L$-bilipschitz chart $\phi \colon U \to \R^n$, we have $\abs{Df\vert_U} = \abs{D(f \circ \phi^{-1}) \circ D\phi} \leq L^{n+1}(\abs{D(f \circ \phi^{-1})} \circ \phi) J_\phi$. If $k = 1$, this also holds for $\abs{df}$, where $\abs{\cdot}$ is the pointwise norm induced by the Grassmann inner product; we discuss this detail later on when we consider Sobolev differential forms (see Corollary \ref{cor:bilip_pullback_Lp_estimate}).

The Sobolev space is then defined as follows.

\begin{defn}\label{def:sobolev_mf-dom}
	Let $M$ be a Riemannian $n$-manifold, and let $1 \leq p \leq \infty$. The space $W^{1,p}(M, \R^m)$ consists of all $f \in W^{1,1}_\loc(M, \R^m)$ for which $f \in L^p(M, \R^m)$ and $\abs{Df} \in L^p(M, \R^m)$. Similarly, the space $W^{1,p}_\loc(M, \R^m)$ consists of all $f \in W^{1,1}_\loc(M, \R^m)$ for which $f \in L^p_\loc(M, \R^m)$ and $\abs{Df} \in L^p_\loc(M, \R^m)$.
\end{defn}

It is clear that if $M$ is a domain of $\R^n$, then the above definition agrees with the Euclidean one. As in the real case, we have the norm
\[
	\norm{f}_{1,p} = \left( \int_U \left(\abs{f(x)}^p + \abs{Df(x)}^p\right) \dd m_n(x) \right)^\frac{1}{p}
\]
for $f \in W^{1,p}(M, \R^m)$.

\subsection{Manifold domain versions of some classical results}

We then point out some manifold versions of the classical Euclidean results for Sobolev spaces, obtained through simple use of charts and the corresponding Euclidean results. They are not necessarily in full generality, but are sufficient for our purposes. We begin by considering the pre- and postcomposition lemmas in the manifold context.

\begin{lemma}\label{lem:sobolev_precomposition_mfld_domain}
	Let $M, M'$ be $n$-dimensional Riemannian manifolds, and let $1 \leq p < \infty$. Suppose that $f \in W^{1,p}(M, \R^m)$ and that $h \colon M' \to M$ is a $C^1$-smooth $L$-bilipschitz map.
	
	Then $f \circ h \in W^{1,p}(M', \R^m)$ with
	\[
		\norm{f \circ h}_{1,p} \leq L^{1 + \frac{n}{p}} \norm{f}_{1,p},
	\]
	and 
	\[
		D(f \circ h) = Df \circ Dh
	\]
	almost everywhere on $M'$.
\end{lemma}
\begin{proof}
	Let $\phi \colon U \to \R^n$ be a $C^1$-chart on $U \subset M$ for which $f \circ \phi^{-1}$ is weakly differentiable. Then $\phi \circ h \colon h^{-1} U \to \R^n$ is also a $C^1$-chart, and we have that $(f \circ h) \circ (\phi \circ h)^{-1} = f \circ \phi^{-1}$ is weakly differentiable. 
	
	By definition, $f \circ h \in W^{1,1}_\loc(M)$. Moreover, $D(f \circ h) \circ Dh^{-1} \circ D\phi^{-1} = D((f \circ h) \circ (\phi \circ h)^{-1}) = D(f \circ \phi^{-1}) = Df \circ D\phi^{-1}$, from which we conclude $D(f \circ h) = Df \circ Dh$.
	
	Finally, the norm estimates are done similarly as in the proof of Lemma \ref{lem:sobolev_precomposition} using \eqref{eq:lipschitz_D} and \eqref{eq:bilipschitz_J}.
\end{proof}

We note that there was in fact no need to directly apply Lemma \ref{lem:sobolev_precomposition}, since it had already been used to make sure the definition of weak differentiability on manifolds was well-posed.

\begin{lemma}\label{lem:sobolev_postcomposition_mfld_domain}
	Let $M$ be a Riemannian $n$-manifold, let $V \subset \R^m$, and let $1 \leq p < \infty$. Suppose that $f \in W^{1,p}(M, V)$ and that $h \colon V \to \R^k$ is a $C^1$-smooth $L$-Lipschitz map. Suppose furthermore that one of the following holds:
	\begin{itemize}
		\item $M$ has finite volume;
		\item $0 \in V$ and $h_2(0) = 0$.
	\end{itemize}
	
	Then $h \circ f \in W^{1,p}(U, \R^k)$, and 
	\[
		D(h \circ f) = Dh \circ Df
	\]
	almost everywhere on $U$.
\end{lemma}
\begin{proof}
	Let $\phi \colon U \to M$ be a chart on a neighborhood of a point $x \in M$ for which $f \circ \phi^{-1}$ is weakly differentiable. Then by Corollary \ref{cor:sobolev_composition_local}, $h \circ f \circ \phi^{-1}$ is weakly differentiable and $D(h \circ f \circ \phi^{-1}) = Dh \circ D(f \circ \phi^{-1}) = Dh \circ Df \circ D\phi^{-1}$. Now by definition of weak differentiability on manifolds, we have $h \circ f \in W^{1,1}_\loc(M, \R^k)$ and $D(h \circ f) = Dh \circ Df$.
	
	The integrabilities are then handled by the exact same way as in the proof of Lemma \ref{lem:sobolev_postcomposition}
\end{proof}

\begin{cor}\label{cor:sobolev_composition_mfld_dom_local}
	Let $M, M'$ be Riemannian $n$-manifolds, let $V \subset \R^m$ and $V' \subset \R^k$ be open subsets, and let $1 \leq p < \infty$. Suppose that $f \in W^{1,p}_\loc(M, V)$, $h_1 \colon M' \to M$ is a $C^1$-smooth locally bilipschitz map, and $h_2 \colon V \to V'$ is a $C^1$-smooth Lipschitz map.
	
	Then $h_2 \circ f \circ h_1 \in W^{1,p}_\loc(M', V')$ and 
	\[
		D(h_2 \circ f \circ h_1) = Dh_2 \circ Df \circ Dh_1
	\]
	almost everywhere on $M'$.
\end{cor}

Next, we give a local Sobolev embedding theorem on manifold domains.

\begin{thm}\label{thm:sobolev_mfld_dom_embedding_local}
	Let $M$ be a Riemannian $n$-manifold, and let $p, q \in [1, \infty)$ be such that $1/q \geq 1/p - 1/n$. Then $W^{1,p}_\loc(M, \R^m) \subset L^{q}_\loc(M, \R^m)$.
\end{thm}
\begin{proof}
	Let $f \in W^{1,p}_\loc(M, \R^m)$, and let $h \colon U \to \R^n$ be a bilipschitz chart on $U \subset M$. By Corollary \ref{cor:sobolev_composition_mfld_dom_local}, $f \circ h^{-1} \in W^{1,p}_\loc(hU, \R^m)$. Hence by the Euclidean version of the result, Theorem \ref{thm:sobolev_eucl_embedding_local}, we have that $f \circ h^{-1} \in L^q_\loc(hU, \R^m)$. The claim therefore follows by a similar estimate as in the proof of Lemma \ref{lem:sobolev_precomposition}.
\end{proof}

Finally, we again have by the same argument as in the Euclidean case the following corollary of Theorem \ref{thm:sobolev_mfld_dom_embedding_local}.

\begin{cor}\label{cor:sobolev_mfld_dom_base_exponent}
	Let $M$ be a Riemannian $n$-manifold, and let $p \in [1, \infty)$. Suppose that $f \in W^{1,1}_\loc(M, \R^m)$ and $\abs{Df} \in L^p_\loc(M)$. Then $f \in W^{1,p}_\loc(M, \R^m)$.
\end{cor}

%%%%%%%%%%%%%%%%%%%%%%%%%%%%%%%%%%%%%%%%%%%%%%%%%%%%%%%%%%%%%%%%%%%%%%%%%%%%%%%%%%%%%%%%%%%%%
\section{Local Sobolev spaces for continuous functions between manifolds}\label{sect:sobolev_cont_functions}

In this section, we discuss an approach to Sobolev maps between manifolds which avoids defining a full Sobolev space. Namely, suppose that $f \colon M \to N$ is continuous. Then for every $x \in M$, there is a neighborhood $U$ of $x$ which is mapped by $f$ inside a chart neighborhood $V$ of $f(x)$; note that the continuity of $f$ is crucial for this. Hence, we're able to define the space $C(M, N) \cap W^{1,p}_\loc(M, N)$ using charts, similarly to our definition of the space $W^{1,1}_\loc(M, \R^n)$.

\begin{defn}\label{def:local_Sobolev_for_continuous_maps}
	Let $M$ and $N$ be a Riemannian $n$- and $m$-manifold, respectively. Suppose $f \colon M \to N$ is a continuous map and $1 \leq p < \infty$. Then $f \in W^{1, p}_\loc(M, N)$ if, for every $x \in M$, there exists a neighborhood $U$ and a bilipschitz $C^1$-chart $\psi \colon V \to \R^m$ for which $fU \subset V$ and $\psi \circ f\vert_U \in W^{1,p}_\loc(U, \R^m)$.
\end{defn}

In order to work with this definition, one key result is required

\begin{prop}\label{prop:cont_sobolev_weak_D_existence}
	Suppose that $f \colon M \to N$ is continuous and that $f \in W^{1,p}_\loc(M, N)$ as defined above. Then there exists a measurable bundle map $Df \colon TM \to TN$ for which 
	\[
		D\psi \circ Df = D(\psi \circ f\vert_U)
	\]
	whenever $U \subset M$ is an open domain and $\psi \colon V \to \R^m$ is a $C^1$-smooth bilipschitz chart for which $fU \subset V$.
\end{prop}
\begin{proof}
	Given such $\psi$, we define $Df$ on $U$ by
	\[
		Df = D(\psi^{-1}) \circ D(\psi \circ f\vert_U).
	\]
	The resulting map is measurable. It then remains to verify that this definition is independent on the $\psi$ used.
	
	Suppose then that $\psi_1, \psi_2 \colon V \to \R^m$ are bilipschitz $C^1$-charts with $fU \subset V$. By Corollary \ref{cor:sobolev_composition_mfld_dom_local}, we get
	\begin{align*}
		&D(\psi_1^{-1}) \circ D(\psi_1 \circ f\vert_U)\\
		&\qquad= D(\psi_2^{-1} \circ \psi_2 \circ \psi_1^{-1})
		\circ D(\psi_1 \circ f\vert_U) \\
		&\qquad= D(\psi_2^{-1}) \circ D(\psi_2 \circ \psi_1^{-1} \circ \psi_1 
		\circ f\vert_U)\\
		&\qquad= D(\psi_2^{-1}) \circ D(\psi_2 \circ f\vert_U)
	\end{align*}
	almost everywhere on $U$.
\end{proof}

We also point out a composition result in the similar vein as Corollary \ref{cor:sobolev_composition_mfld_dom_local}

\begin{lemma}\label{lem:sobolev_composition_mfld_targ_cont}
	Let $M$, $M'$, $N, N'$ be connected Riemannian manifolds of dimensions $m$, $m$, $n$ and $n'$, respectively, and let $1 \leq p < \infty$. Suppose that $f \colon M \to N$ is a continuous map with $f \in W^{1,p}_\loc(M, N)$, that $h_1 \colon M' \to M$ is a $C^1$-smooth $L$-bilipschitz map, and that $h_2 \colon N \to N'$ is a $C^1$-smooth $L$-Lipschitz map. Then $h_2 \circ f \circ h_1 \in W^{1,p}_\loc(M', N')$ and 
	\[
		D(h_2 \circ f \circ h_1) = Dh_2 \circ Df \circ Dh_1
	\]
	almost everywhere on $M'$.
\end{lemma}
\begin{proof}
	Let $x \in M'$. By using continuity and Theorem \ref{thm:bilipschitz_charts}, we find a chain of smooth bilipschitz charts $\psi' \colon V' \to \R^{n'}$, $\psi \colon V \to \R^n$, $\phi \colon U \to \R^m$ and $\phi' \colon U' \to \R^m$ such that $x \in U'$, $h_1(U') \subset U$, $f(U) \subset V$, $h_2(V) \subset V'$, and $\psi \circ f\vert_U \in W^{1,p}_\loc(U, \R^n)$.
	
	By Corollary \ref{cor:sobolev_composition_mfld_dom_local}, we obtain that $\psi' \circ h_2 \circ f \circ h_1\vert_{U'} = (\psi' \circ h_2 \circ \psi^{-1}) \circ (\psi \circ f) \circ h_1\vert_{U'} \in W^{1,p}_\loc(U', \R^n)$, and also $D(\psi' \circ h_2 \circ f \circ h_1\vert_{U'}) = D(\psi' \circ h_2 \circ \psi^{-1}) \circ D(\psi \circ f) \circ D(h_1\vert_{U'})$. By definition, we conclude that $h_2 \circ f \circ h_1 \in W^{1,p}_\loc(M', N')$. Moreover, since $D(\psi \circ f\vert_U) = D\psi \circ Df\vert_U$ by Proposition \ref{prop:cont_sobolev_weak_D_existence}, and since $D(\psi' \circ h_2 \circ \psi^{-1}) = D\psi' \circ Dh_2 \circ D\psi^{-1}$ by the classical chain rule, we conclude that the stated chain rule holds.
\end{proof}

With just this little effort, our tools are in fact already essentially sufficient to give a definition for quasiregular maps between Riemannian manifolds and to work with said definition. Indeed, only a slight detour to define the Jacobian determinant would be required before stating the definition. However, since our secondary goal is to show the equivalence of various definitions of quasiregular maps between Riemannian manifolds, we instead proceed to discuss several other means of defining Sobolev spaces of maps with a manifold target, this time without the assumption of continuity.

%%%%%%%%%%%%%%%%%%%%%%%%%%%%%%%%%%%%%%%%%%%%%%%%%%%%%%%%%%%%%%%%%%%%%%%%%%%%%%%%%%%%%%%%%%%%%
\section{The Sobolev space via the Nash embedding}\label{sect:sobolev_nash_embedding}

We then state the first definition of a true Sobolev space for measurable maps between manifolds, which is by using the Nash embedding theorem. The Nash embedding theorem states that every manifold is isometrically embeddable into some higher dimensional Euclidean space $\R^l$. More precisely, the statement is as follows; see \cite{Nash_embedding} for details.

\begin{thm}[$C^\infty$ Nash embedding]\label{thm:nash_embedding}
	Let $M$ be a smooth Riemannian $n$-manifold. Then there exists $m \geq n$ and a smooth injective map $\iota \colon M \to \R^m$ satisfying
	\[
		\ip{v}{w}_x = \ip{D\iota(x)v}{D\iota(x)w}_{\iota(x)}
	\]
	for all $x \in M$ and $v,w \in T_x M$; here, the right hand metric is the standard Euclidean one. The map $\iota$ is called a \emph{Nash embedding}.
	
	Moreover, $m$ can be chosen to be $m \leq n(n+1)(3n+11)/2$. If $M$ is closed, this further improves to $m \leq n(3n+11)/2$.
\end{thm}

We then call any embedding provided by Theorem \ref{thm:nash_embedding} a \emph{Nash embedding}. Using a Nash embedding, we may define a Sobolev space between manifolds as follows.

\begin{defn}\label{def:sobolev_via_nash_embedding}
	Let $M$ and $N$ be smooth Riemannian manifolds of dimensions $m$ and $n$, respectively, and let $\iota \colon N \to \R^{n'}$ be a Nash embedding. Then we define
	\[
		W^{1,p}_\iota(M, N) = \left\{
			f \colon M \to N \text{ measurable} : 
			\iota \circ f \in W^{1,p}\bigl(M, \R^{n'}\bigr)
		\right\}.
	\]
	Similarly,
	\[
		W^{1,p}_{\iota, \loc}(M, N) = \left\{
		f \colon M \to N \text{ measurable} : 
		\iota \circ f \in W^{1,p}_\loc\bigl(M, \R^{n'}\bigr)
		\right\}.
	\]
\end{defn}

There are two primary questions that require verification in order to properly use Definition \ref{def:sobolev_via_nash_embedding}:
\begin{itemize}
	\item Are the spaces $W^{1,p}_\iota(M, N)$ and $W^{1,p}_{\iota, \loc}(M, N)$ independent on the choice of Nash embedding $\iota$?
	\item If $f \in W^{1,p}_\iota(M, N)$, is $Df := (D\iota)^{-1} \circ D(\iota \circ f)$ well defined; that is, do we have $\Im D(\iota \circ f) \subset \Im D\iota$?
\end{itemize}

For the independence on embedding, an obstacle immediately arises for global Sobolev spaces. Namely, consider the embedding $\iota \colon \R \to \R^2$ which maps $\R$ to $\R \times \{1\}$ isometrically. Then $W^{1,p}_{\iota}(\R, \R) = \emptyset$, since no function $f \colon \R \to \R \times \{1\}$ has a finite integral. This issue occurs whenever $M$ has infinite measure, and only for the global space $W^{1,p}_{\iota}$.

Our goal for the Nash embedding Sobolev spaces is to show that the local $W^{1,p}_{\iota, \loc}(M, N)$-spaces are well-behaved. That is, that their elements $f$ are independent on the choice of embedding $\iota$ and admit a well defined measurable map $Df \colon TM \to TN$. We leave this for the next section, as there we prove an equivalence of definitions -result which implies this immediately. We note that this result also trivially implies a global version for compact $M$, as the compactness of $M$ implies that $W^{1,p}_\loc(M, N) = W^{1,p}(M, N)$. Indeed, assuming compactness of $M$ and/or $N$ is relatively common in the study of Nash embedding Sobolev spaces; e.g.\ the paper of Haj\l asz, Iwaniec, Mal\'y and Onninen \cite{Hajlasz-Iwaniec-Maly-Onninen} provides an introduction to Nash embedding Sobolev spaces between compact manifolds.

%%%%%%%%%%%%%%%%%%%%%%%%%%%%%%%%%%%%%%%%%%%%%%%%%%%%%%%%%%%%%%%%%%%%%%%%%%%%%%%%%%%%%%%%%%
\section{The Sobolev space via post-composition}\label{sect:sobolev_intrinsic_derivative}

One more approach to first order Sobolev spaces between manifolds was proposed by Convent and Van Schaftingen in \cite{Convent-VanSchaftingen_Sobolev}. Their definition is based on post-composition with $C^1_0$-functionals. Indeed, Lipschitz post-composition is a known approach to generalizing Sobolev spaces, with another example being e.g.\ a study of Sobolev functions into metric spaces by Reshetnyak \cite{Reshetnyak_metric_Sobolev}. Convent and Van Schaftingen show that the post-composition definition results in a unique map $Df \colon TM \to TN$, which they call the \emph{colocal weak derivative}. This weak $Df$ then allows for an intrinsic definition of a Sobolev space.

A notable point is that the Sobolev space in \cite{Convent-VanSchaftingen_Sobolev} does not assume any integrability on the map $f$ itself, only on the weak derivative $Df$. The lack of this assumption however turns out to be the only ultimate difference between this definition and the global first order Nash embedding definition. In particular, the resulting local Sobolev spaces end up fully equivalent with the Nash embedding definition.

We present here the relevant parts of the theory; for further details, along with some of the proofs of the results we use, the reader is referred to \cite{Convent-VanSchaftingen_Sobolev}.

\subsection{Colocal weak differentiability}

Under the approach of this section, the weak differential of a map $f \colon M \to N$ is defined as follows.

\begin{defn}[{\cite[Definition 1.1]{Convent-VanSchaftingen_Sobolev}}]
	Let $M$, $N$ be smooth Riemannian manifolds of dimension $m$ and $n$, respectively. Suppose that $f \colon M \to N$ is measurable. Then $f$ is \emph{colocally weakly differentiable} if for every $u \in C^1_0(N)$, the map $u \circ f$ is weakly differentiable (ie.\ $u \circ f \in W^{1,1}_\loc(M)$). We denote the space of all colocally weakly differentiable $f \colon M \to N$ by $\cW(M, N)$.
\end{defn}

For every $f \in \cW(M, N)$, there exists an almost everywhere unique measurable \emph{colocal weak derivative} $Df \colon TM \to TN$ which satisfies $Du \circ Df = D(u \circ f)$ for all $u \in C^1_0(N, \R)$. We refer to the proof of \cite[Proposition 1.2]{Convent-VanSchaftingen_Sobolev} for the details, but due to how important this is for us, we briefly comment on the main idea of how this result is obtained. The key trick is to consider a bilipschitz chart on $N$ which has been extended to all of $N$ using a cutoff function: that is, $\phi \in C^\infty_0(N, \R^n)$ and $\phi\vert_U$ is a bilipschitz chart for some $U \subset N$. By definition, $\phi \circ f$ is weakly differentiable, and we may define $Df$ on $f^{-1} U$ by $Df = D(\phi\vert_U^{-1}) \circ D(\phi \circ f)$. Any compactly supported $u \in C^1_0(N)$ can then be decomposed into a sum of finitely many $u_i$ with support contained in the chart-part of such an extended chart $\phi$, and the rule $Du_i \circ Df = D(u_i \circ f)$ can be verified by a Sobolev chain rule (see Lemma \ref{lem:sobolev_postcomposition_mfld_domain}).

We then recall some other main results of \cite[Section 1]{Convent-VanSchaftingen_Sobolev}, referring there for the details. First, if $f \in \cW(M, N)$ and $g \in C^1(N, \R^n)$ are such that $g \circ f$ is weakly differentiable, then $D(g \circ f) = Dg \circ Df$; see \cite[Proposition 1.2]{Convent-VanSchaftingen_Sobolev}. Note that this does not imply that $g \circ f$ is weakly differentiable for $f \in \cW(M, N)$ and $g \in C^1(N, \R^n)$; only that the chain rule holds if this composition is a-priori weakly differentiable. Similarly, if $f \in \cW(M, N)$ and $g \in C^1(N, N')$ are such that $g \circ f \in \cW(M, N')$, then again $Dg \circ Df = D(g \circ f)$; see \cite[Proposition 1.5]{Convent-VanSchaftingen_Sobolev}. Moreover, if $g$ is a $C^1$-embedding, then $g \circ f \in \cW(M, N')$ implies $f \in \cW(M, N)$; see \cite[Proposition 1.6]{Convent-VanSchaftingen_Sobolev}.

It is \emph{not} true that $\cW(M, \R) = W^{1,1}_\loc(M)$; according to \cite[p.4]{Convent-VanSchaftingen_Sobolev} it instead corresponds to a so called \emph{Sobolev space by truncations}, which is a larger space. However, the gap can be bridged in the following way.

\begin{lemma}\label{lem:colocal_euclidean_equivalence}
	Let $M$ be a smooth Riemannian $n$-manifold. Then we have $f \in W^{1,1}_\loc(M, \R^m)$ if and only if $f \in \cW(M, \R^m)$ and $\abs{Df} \in L^1_\loc(M)$. Moreover, in this situation, the two definitions of $Df$ coincide.
\end{lemma}
\begin{proof}
	The "only if" -claim is clear by Corollary \ref{cor:sobolev_composition_mfld_dom_local}, since maps $u \in C^1_0(\R^m)$ are Lipschitz. 
	
	We then prove the other direction, as \cite{Convent-VanSchaftingen_Sobolev} only states such an equivalence result for the Euclidean definition of Sobolev spaces, referring to B\'enilan et al.\ \cite[Lemma 2.1]{Benilan-et-al} for the proof. Indeed, our argument will be similar as in \cite{Benilan-et-al}, with just the slight complications of using charts to work with the manifold domain. For the duration of this proof, we use $\tilde{D}$ and to denote colocal weak derivatives, in order to clearly distinguish them from ordinary weak derivatives. We similarly use $\tilde{d}$ for the weak differentials obtained by post-composing a colocal weak derivative with the standard identification $\iota_1 \colon T \R \to \R$.
	
	Suppose that $f \in \cW(M, \R^m)$ and $\smallabs{\tilde{D}f} \in L^1_\loc(M)$. Let $u_{i,k} \in C_0^1(\R^m)$ be such that $u_{i,k} \vert_{B_m(0, k)} = \proj_i \vert_{B_m(0, k)}$, the estimate $\abs{u_{i,k}} \leq \abs{\proj_i}$ holds everywhere, and $u_{i,k}$ is 1-Lipschitz. Let $f_{i,k} = u_{i,k} \circ f$, in which case $f_{i,k} \in W^{1,1}_\loc(M)$. We then have $f_{i,k} = f_i$ and $Df_{i,k} = \tilde{D}f_i$ almost everywhere on the set $f^{-1} B_m(0, k)$, where $f_i$ is the $i$:th coordinate mapping of $f$. 
	
	Moreover, since $Df_{i,k} = Du_{i,k} \circ \tilde{D}f$ and $u_{i,k}$ is 1-Lipschitz, we have $\abs{Df_{i,k}} \leq \smallabs{\tilde{D}f}$ almost everywhere. Hence, $\smallabs{df_{i,k} - \tilde{d}f_i} \leq 2\smallabs{\tilde{D}f}$ almost everywhere outside $f^{-1} B_m(0, k)$. Since $f^{-1} B_m(0, k)$ tends locally to full measure as $k \to \infty$, we have $df_{i,k} \to \tilde{d}f_i$ locally in $L^1$.
	
	We then show that $f_{i,k} \to f_i$ locally in $L^1$. For this, we wish to show that $f_i \in L^1_\loc(M)$; given that, we may similarly as before estimate $\abs{f_{i,k} - f_i} \leq 2\abs{f_i}$ outside $f^{-1} B_m(0, k)$, and then conclude that $f_{i,k} \to f_i$ locally in $L^1$ by the fact that $f^{-1} B_m(0, k)$ tends locally to full measure.
	
	Suppose then towards contradiction that $f_i \notin L^1_\loc(M)$. Since $\abs{f_{i,k}} \to \abs{f_i}$ monotonely, we therefore obtain by monotone convergence a precompact neighborhood $U \subset M$ on which $\norm{f_{i,k}\vert_U}_1 \to \infty$. We may additionally assume that the boundary of $U$ is smooth by slightly enlarging $U$. If we write $g_k = \norm{f_{i,k}\vert_U}_1^{-1} f_{i,k}\vert_U$, we then have $g_k \to 0$ almost everywhere and $\norm{g_k}_1 = 1$.
	
	But we also have that $g_k \in W^{1,1}(U)$ and that 
	\[
		\norm{Dg_k}_1 = \norm{f_{i,k}\vert_U}_1^{-1} \norm{(Df_{i,k})\vert_U}_1 \leq \norm{f_{i,k}\vert_U}_1^{-1} \smallnorm{(\tilde{D}f)\vert_U}_1 \to 0.
	\]
	Hence, $(g_k)$ is bounded in $W^{1,1}(U)$, and the Rellich-Kondrachov Theorem \ref{thm:sobolev_eucl_embedding_compact} implies that it has a convergent subsequence in $L^1$. Since $g_k \to 0$ almost everywhere, the limit of this subsequence is necessarily zero. But this is a contradiction since $\norm{g_k}_1 = 1$. Hence, we finally have $f_i \in L^1_\loc(M)$, and therefore $f_{i,k} \to f_i$ locally in $L^1$.
	
	Consider now a bounded $U \subset M$ with closure contained in the domain of a bilipschitz chart $\phi \colon U \to \R^n$. In this case, $f_{i,k}\vert_U \to f_i\vert_U$ and $Df_{i,k}\vert_U \to (\tilde{D}f_i)\vert_U$ in $L^1$. By Lemma \ref{lem:sobolev_precomposition_mfld_domain}, \eqref{eq:lipschitz_D}, and \eqref{eq:bilipschitz_J}, we have that $f_{i,k} \circ \phi^{-1} \in W^{1,1}(\phi U, \R^m)$, that $f_{i,k} \circ \phi^{-1} \to f_i \circ \phi^{-1}$ in $L^1$, and that $d(f_{i,k} \circ \phi^{-1}) \to \tilde{d}f_i \circ D\phi^{-1}$ in $L^1$.
	
	By the completeness of $W^{1,1}(\phi U, \R^m)$ given in Theorem \ref{thm:sobolev_eucl_banach}, we therefore have that $f_i \circ \phi^{-1} \in W^{1,1}(\phi U, \R^m)$ and $d(f_i \circ \phi^{-1}) = \tilde{d}f_i \circ D\phi^{-1}$. Now the definition \ref{def:sobolev_mf-dom} of $W^{1,1}_\loc(M, \R^m)$ yields that $f_i \in W^{1,1}_\loc(M, \R^m)$, and its weak differential is $\tilde{d}f_i$. Since this holds for all coordinate maps $f_i$, the proof is therefore complete.
\end{proof}

\subsection{Adding the integrability condition}
Based on this, we may define a Sobolev space as in \cite[Definition 2.1]{Convent-VanSchaftingen_Sobolev}. We also define a local counterpart.

\begin{defn}
	\label{def:sobolev_via_colocal_derivatives}
	Let $M$, $N$ be smooth Riemannian manifolds of dimension $m$ and $n$, respectively. Then $f \in \cW^{1,p}(M, N)$ if $f \in \cW(M, N)$ and $\abs{Df} \in L^p(M)$. Similarly, $f \in \cW^{1,p}_\loc(M, N)$ if $f \in \cW(M, N)$ and $\abs{Df} \in L^p_\loc(M)$.
\end{defn}

\begin{rem}
	Note that the operator norm $\abs{Df}$ used above is different from the norm $\abs{Df}_{T^*M \otimes TN}$ originally used by Convent and Van Schaftingen. The norm $\abs{A}_{T^*M \otimes TN}$ for $A \colon T^*_x M \to T^*_y N$ is obtained by defining an inner product on $T^*_xM \otimes T_yN$, by setting
	\[
		\ip{\omega \otimes v}{\tau \otimes w} = \ip{\omega}{\tau} \ip{v}{w}
	\]
	for $\omega, \tau \in T^*_x M$ and $v, w \in T_y N$ and then extending bilinearly. If $e_i$ is a basis of $T_x M$ and $\eps_i$ is the dual basis $\eps_i(e_j) = \delta_{ij}$, then we may identify $A$ uniquely with $\eps_1 \otimes A(e_1) + \dots + \eps_m \otimes A(e_m)$. 
	
	If moreover the basis $e_i$ is orthonormal, then so is $\eps_i$, and a straightforward computation gives us the formula
	\[
		\abs{A}^2_{T^*M \otimes TN} = \abs{A(e_1)}^2 + \ldots + \abs{A(e_m)}^2.
	\]
	Hence, we obtain the estimate
	\[
		\abs{A}^2_{T^*M \otimes TN} \leq m\abs{A}^2.
	\]
	Conversely, if $v \in T_x M$ with $\abs{v} \leq 1$, we may write $v = v_1 e_1 + \dots + v_m e_m$, with $v_1^2 + \dots + v_m^2 \leq 1$. We can then estimate using triangle inequality and Cauchy--Schwarz as follows:
	\begin{multline*}
		\abs{A(v)} 
		\leq \abs{v_1}\abs{A(e_1)} + \dots + \abs{v_m}\abs{A(e_m)}\\
		\leq \left(v_1^2 + \dots + v_m^2\right)^{\frac{1}{2}} 
				\left(\abs{A(e_1)}^2 + \dots + \abs{A(e_m)}^2 \right)^{\frac{1}{2}}
		\leq \abs{A}_{T^*M \otimes TN}.
	\end{multline*}
	
	Hence, the maps $\abs{Df}_{T^*M \otimes TN}$ and $\abs{Df}$ are uniformly comparable with constant $\sqrt{m}$, and therefore they yield equivalent definitions of $W^{1,p}(M, N)$, leaving the choice of norm up to preference. We use the operator norm due to its closer connections with quasiregular analysis.
\end{rem}

It is worth noting that the above definition does not include any global integrability assumptions on $f$: indeed, the actual relation to the Euclidean situation is $W^{1,p}(M, \R^m) = \cW^{1,p}(M, \R^m) \cap L^p(M, \R^m)$ by Lemma \ref{lem:colocal_euclidean_equivalence}. This way, $\cW^{1,p}(M, N)$ sidesteps the issue where the isometric embedding definition would depend on eg.\ where/whether the embedded space meets zero. The main result of \cite[Section 2]{Convent-VanSchaftingen_Sobolev} in fact links this definition to the one by embeddings.

\begin{prop}[{\cite[Proposition 2.6]{Convent-VanSchaftingen_Sobolev}}]
	\label{prop:colocal_sobolev_embedding_equivalence}
	Let $M$, $N$ and $N'$ be Riemannian manifolds of dimension $m$, $n$ and $n'$, respectively. Let $f \colon M \to N$, and let $\iota \colon N \to N'$ be a $C^1$-embedding satisfying
	\[
		\ip{v}{w}_y = \ip{D\iota(y)v}{D\iota(y)w}_{\iota(y)}
	\]
	for all $y \in N$ and $v, w \in T_y N$. Then $f \in \cW^{1,p}(M, N)$ if and only if $\iota \circ f \in \cW^{1,p}(M, N')$.
\end{prop}

Hence, combined with Lemma \ref{lem:colocal_euclidean_equivalence}, we see that the isometric embedding definition of Sobolev spaces is dependent of the choice of embedding only through the global $L^p$-integrability of $\iota \circ f$.

However, the above result is for global spaces, while our ultimate goal is the equivalence of the corresponding local spaces. Due to this, we give a localization result which will be used to apply the above result in a local context.

\begin{lemma}\label{lem:colocal_sobolev_restricion_characterization}
	Let $M$ and $N$ be Riemannian manifolds of dimension $m$ and $n$, respectively. Let $f \colon M \to N$. Then $f \in \cW^{1,p}_\loc(M, N)$ if and only if for every $x \in M$ there exists a neighborhood $U$ of $x$ for which $f\vert_U \in \cW^{1,p}(U, N)$. Moreover, if $U$ is such a neighborhood, then $D(f\vert_U)(y) = Df(y)$ for almost every $y \in U$.
\end{lemma}
\begin{proof}
	The main idea of the (simple) proof is to reduce these to the corresponding properties for $W^{1,p}_\loc(M, \R)$.
	
	For the "only if" -direction, if $f \in \cW^{1,p}_\loc(M, N)$, then for $u \in C_0^1(N)$ we have $u \circ f\vert_U = (u \circ f)\vert_U \in W^{1,1}_\loc(U)$ since restrictions of $W^{1,1}_\loc$-maps remain $W^{1,1}_\loc$. Hence, if $U$ is a neighborhood of $x$, we have $f \in \cW(U, N)$.
	
	We then wish to show the part that $D(f\vert_U) = Df$ almost everywhere in $U$. This readily follows by computing using Lemma \ref{lem:sobolev_postcomposition_mfld_domain} that $Du \circ Df = D(u \circ f) = D((u \circ f)\vert_U) = D(u \circ (f\vert_U)) = Du \circ D(f\vert_U)$ almost everywhere on $U$ for every $u \in C^1_0(M)$. Hence, we have $D(f\vert_U) = Df$ almost everywhere in $U$. The "only if" -part of the main claim then also follows, since $\abs{Df} \in L^p(U)$ if we select a bounded $U$ compactly contained in $M$.
	
	For the "if" -direction, suppose then that we have a locally finite cover of $M$ by $U_i$ for which $f\vert_{U_i} \in \cW^{1,p}(U_i, N)$. Let $u \in C^1_0(N)$, in which case $(u \circ f)\vert_{U_i} = u \circ f\vert_{U_i} \in W^{1,1}_\loc(U)$. It follows that $u \circ f \in W^{1,1}_\loc(M)$. We conclude that $f \in \cW(M, N)$. For the rest, the previously used restriction property of colocal weak derivatives implies that $(\abs{Df})\vert_{U_i} = \abs{D(f\vert_{U_i})} \in L^p(U_i)$ for every $i$, in which case $\abs{Df} \in L^p_\loc(M)$.
\end{proof}

\subsection{Equivalence of local Sobolev spaces}

We now finally show that the multiple definitions for local Sobolev spaces we've discussed are equivalent.

\begin{thm}\label{thm:local_sobolev_eq_of_defs}
	Let $M$ and $N$ be Riemannian $n$- and $m$-manifolds, respectively. Then $W^{1,p}_{\iota, \loc}(M, N) = \cW^{1,p}_\loc(M, N)$. 
	
	Furthermore, if $f \colon M \to N$ is continuous, then $f \in \cW^{1,p}_\loc(M, N)$ if and only if $f \in W^{1,p}_\loc(M, N)$ in the sense of Definition \ref{def:local_Sobolev_for_continuous_maps}.
	
	Moreover, for all above definitions, the corresponding maps $Df(x)$ coincide for almost every $x \in M$.
\end{thm}
\begin{proof}
	For the first claim, suppose first that $f \in W^{1,p}_{\iota, \loc}(M, N)$. Then $\iota \circ f \in W^{1,p}_\loc(M, \R^k)$. It follows that for every $x \in M$ we have a neighborhood $U \subset M$ in which $\iota \circ f\vert_U \in W^{1,p}(U, \R^k)$, and $D(\iota \circ f\vert_U) = D(\iota \circ f)$ almost everywhere on $U$. 
	
	By Lemma \ref{lem:colocal_euclidean_equivalence}, $\iota \circ f\vert_U \in \cW^{1,p}(U, \R^k)$ with the same weak derivative for both spaces. By Proposition \ref{prop:colocal_sobolev_embedding_equivalence}, $f \vert_U \in \cW^{1,p}(U, \R^k)$. Moreover, since $\iota$ is $C^1$, we obtain that $D(\iota \circ f\vert_U) = D\iota \circ D(f\vert_U)$; recall that this was due to \cite[Proposition 1.2]{Convent-VanSchaftingen_Sobolev}. Since these neighborhoods $U$ cover $M$, we conclude by Lemma \ref{lem:colocal_sobolev_restricion_characterization} that $f \in \cW^{1,p}_\loc(M, N)$, and moreover $D\iota \circ Df = D(\iota \circ f)$.
	
	The opposite direction is nearly the same in reverse, but with one extra detail needed. Indeed, if $f \in \cW^{1,p}_\loc(M, N)$, by doing the argument in reverse we have no issues reaching the point where for every point $x \in M$ we have a neighborhood $U$ on which $\iota \circ f\vert_U \in \cW^{1,p}(U, \R^k)$. However, now Lemma \ref{lem:colocal_euclidean_equivalence} only yields us that $\iota \circ f\vert_U \in W^{1,1}_\loc(U, \R^k)$ and $\abs{D(\iota \circ f\vert_U)} \in L^p(U)$. Hence, we need to use Corollary \ref{cor:sobolev_mfld_dom_base_exponent} to conclude that $\iota \circ f\vert_U \in W^{1,p}_\loc(U, \R^k)$. Now, since we have this in the neighborhood of every point of $M$, we conclude that $\iota \circ f \in W^{1,p}_\loc(M, \R^k)$, and consequently $f \in W^{1,p}_{\iota, \loc}(M, N)$.
	
	For the equivalence of maps $Df$, we already computed in the beginning that the $Df$ given by the space $\cW^{1, p}_\loc(M, N)$ satisfies $D(\iota \circ f) = D\iota \circ Df$, and is therefore valid for $W^{1, p}_{\iota, \loc}(M, N)$. Since the $Df$ of $\cW^{1, p}_\loc(M, N)$ is almost everywhere unique, we only need a similar uniqueness result for any possible maps $Df$ for $W^{1, p}_{\iota, \loc}(M, N)$. However, this immediately follows since $D \iota$ is injective. Hence, we conclude that $\cW^{1, p}_\loc(M, N)$ and $W^{1, p}_{\iota, \loc}(M, N)$ have almost everywhere equivalent maps $Df$.
	
	Next, we prove the second claim regarding the continuous case. Suppose $f$ is continuous and is in $W^{1,p}_\loc(M, N)$ in the sense of Definition \ref{def:local_Sobolev_for_continuous_maps}. By Lemma \ref{lem:sobolev_composition_mfld_targ_cont}, $u \circ f$ is weakly differentiable for every $u \in C^1_0(N)$, and therefore $f \in \cW(M, N)$. Moreover, $Df$ is the colocal weak derivative of $f$, due to Lemma \ref{lem:sobolev_composition_mfld_targ_cont} and the uniqueness of the colocal weak derivative (see again \cite[Proposition 1.2]{Convent-VanSchaftingen_Sobolev}). Moreover, given a smooth $L$-bilipschitz chart $\psi \colon V \to \R^n$ on $N$, we have $\abs{Df\vert_{f^{-1}V}}^p \leq L^p \abs{D(\psi \circ f\vert_{f^{-1}V})}^p$. Hence, $\abs{Df} \in L^p_\loc(M)$, and therefore $f \in \cW^{1, p}_\loc(M, N)$.
	
	It therefore only remains to prove the opposite direction, where we suppose $f$ is continuous and $f \in \cW^{1,p}_\loc(M, N)$. We may select a $\psi \in C^\infty_0(N, \R^n)$ such that $\psi \vert_V$ is a bilipschitz chart, and a $U \subset M$ such that $f(U) \subset V$. It follows from Lemma \ref{lem:colocal_sobolev_restricion_characterization} that $f \in \cW^{1,p}_\loc(U, N)$. We thus have by the definition of colocal weak differentiability that $\psi \circ f \in W^{1,1}_\loc(U, \R^n)$. By \cite[Proposition 1.2]{Convent-VanSchaftingen_Sobolev}, $D(\psi \circ f\vert_U) = D\psi \circ Df\vert_U$, and since $\psi$ is $L$-Lipschitz on the image of $f\vert_U$, we conclude that $\abs{D(\psi \circ f\vert_U)} \in L^p_\loc(U)$. By Corollary \ref{cor:sobolev_mfld_dom_base_exponent}, we hence have $\psi \circ f\vert_U \in W^{1,p}_\loc(U, \R^n)$. Hence, $f \in W^{1,p}_\loc(M, N)$, and the previous case then implies that the corresponding weak derivatives are the same.
\end{proof}

The previous theorem also confirms the independence of $W^{1,p}_{\iota, \loc}(M, N)$ on the embedding $\iota$, and therefore does the same for $W^{1,p}_\iota(M, N)$ when $M$ is compact. Moreover, in the above situations it also shows the existence of a unique $Df$ satisfying $D\iota \circ Df = D(\iota \circ f)$, implying the well-definedness of the $Df$ map for $W^{1, p}_{\iota, \loc}(M, N)$.

%%%%%%%%%%%%%%%%%%%%%%%%%%%%%%%%%%%%%%%%%%%%%%%%%%%%%%%%%%%%%%%%%%%%%%%%%%%

\section{Sobolev differential forms}\label{sect:Sobolev_forms}

As the last topic on Sobolev spaces before moving on to the subject of quasiregular maps, we discuss Sobolev spaces of differential forms. We start by considering Sobolev differential forms on $\R^n$, and then on Riemannian manifolds. We also give several results on how pull-backs and Sobolev differential forms interact.
 
\subsection{Sobolev differential forms and the weak exterior derivative}

Suppose that $U$ is an open domain in $\R^n$, and that $\omega$ is a differential $k$-form on $U$. That is, $\omega$ is a section $U \to \wedge^k T^*U \subset \wedge^k T^*\R^n$. As has already been discussed, the isometries $\iota_{n, x} \colon T_x \R^n \to \R^n$ allow us to canonically identify $\wedge^k T^*\R^n$ with $\R^n \times (\wedge^k \R^n)$. In particular, we have a standard orthonormal frame for $\wedge^k T^*\R^n$: the differential forms
\[
	\eps_I = \dd x_{I_1} \wedge \dots \wedge \dd x_{I_k},
\]
where the index $I$ runs through all $k$-element subsets of $\{1, \dots, n\}$, and $I_j$ is the $j$:th element of $I$ in increasing order. In particular, there are $\binom{n}{k}$ different indices $I$.

Hence, the form $\omega$ can be written in the form $\omega = \sum_I \omega_I \eps_I$, where $\omega_I \colon U \to \R$ is a real function. As discussed earlier, $\omega$ is smooth, continuous, or measurable, exactly if all of its coordinate functions $\omega_I$ are. The pointwise norm induced by the Grassmann inner product becomes
\[
	\abs{\omega} = \left( \sum_I \abs{\omega_I}^2 \right)^{1/2}.
\]
Since all norms on a finite dimensional vector space are equivalent, we have for every $p \geq 1$ a dimensional constant $C(p, n, k)$ for which
\[
	C(p,n,k)^{-1} \left( \sum_I \abs{\omega_I}^p \right)^{1/p}
	\leq \abs{\omega} 
	\leq C(p,n,k)\left( \sum_I \abs{\omega_I}^p \right)^{1/p}.
\]
In conclusion, we also obtain that $\omega \in L^p(\wedge^k U)$ if and only if $\omega_I \in L^p(U)$ for every $I$.

Hence, a space $W^{1, p}(\wedge^k U)$ of Sobolev $k$-forms may be defined as follows: $\omega \in W^{1, p}(\wedge^k U)$ if $\omega_I \in W^{1,p}(U)$ for every $I$. The space is therefore essentially the space $W^{1, p}(U, \R^N)$ for $N = \binom{n}{k}$ with a special notation, and we inherit a norm and all standard properties of Sobolev spaces from this identification. A local version $W^{1,p}_\loc(\wedge^k U)$ can be defined accordingly.

However, in these notes, we do not focus on the space $W^{1,p}(\wedge^k U)$. Instead, we focus on another type of Sobolev space for differential forms, based on a weak exterior deivative. Suppose that $\omega \in C^{\infty}(\wedge^k U)$, and let $\eta \in C^\infty_0(\wedge^{n-k-1} U)$. We have then that
\[
	0 = \int_U d(\omega \wedge \eta) 
		= \int_U \left( d\omega \wedge \eta + (-1)^k \omega \wedge d\eta \right).
\]
We may hence extend the weak exterior derivative $d$ from functions to $k$-forms, following the characterization in Lemma \ref{lem:weak_derivative_characterizations}. That is, suppose that $\omega \in L^1_\loc(\wedge^k U)$. We say that a form $\tau \in L^1_\loc(\wedge^k U)$ is a \emph{weak differential} of $\omega$, or in other words a \emph{weak exterior derivative} of $\omega$, if it satisfies the condition
\begin{equation}\label{eq:weak_differential_for_forms_eucl}
	\int_U \omega \wedge d\eta = (-1)^{k+1} \int_U \tau \wedge \eta
\end{equation}
for all $\eta \in C^\infty_0(\wedge^{n-k-1} U)$. 

If a form $\tau$ as in \eqref{eq:weak_differential_for_forms_eucl} exists, it is unique up to a set of measure zero, and hence we denote $\tau = d\omega$. We then say that $\omega \in W^{d, p}(\wedge^k U)$ if $\omega \in L^p(\wedge^k U)$ and $\omega$ has a weak differential $d\omega \in L^p(\wedge^{k+1} U)$. The local version $W^{d,p}_\loc(\wedge^k U)$ is defined similarly by requiring that $\omega \in L^p_\loc(\wedge^k U)$ and $d\omega \in L^p_\loc(\wedge^{k+1} U)$. Moreover, we define a mixed integrability version $W^{d, p, q}(\wedge^k U)$ by requiring that $\omega \in L^p(\wedge^k U)$ and $d\omega \in L^q(\wedge^{k+1} U)$. A local mixed integrability version $W^{d, p, q}_\loc(\wedge^k U)$ is also defined in the same manner as usual.

The space $W^{d, p}(\wedge^k U)$ is larger than $W^{1,p}(\wedge^k U)$. This is since even if the coordinate functions $\omega_I$ don't have weak partial derivatives, suitable combinations of them may still exist to yield a valid weak $d\omega$. A good example of this is to consider the 1-form $\omega = f(x_1) dx_1$ in $U = (0, 1)^2 \subset \R^2$, where $f \in L^1([0, 1])$. Clearly these assumptions are not enough to guarantee that $x \mapsto f(x_1)$ has a weak $\partial_1$-derivative; yet it has a weak $\partial_2$-derivative of 0, due to which it can be shown that $\omega$ has the weak differential $d\omega = 0$.

However, for $k = 0$, the spaces are the same. Indeed, $0$-forms on $U$ can be interpreted as functions $U \to \R$, and $W^{1, p}(\wedge^0 U) = W^{d,p}(\wedge^0 U) = W^{1, p}(U)$, with the weak differential being essentially the weak gradient as discussed earlier.

\medskip

The Sobolev spaces defined using the weak exterior derivative extend naturally to the manifold setting. Indeed, suppose that $M$ is an oriented Riemannian manifold. Then for a measurable differential form $\omega \in L^1_\loc(M)$, we again say that $d\omega \in L^1_\loc(\wedge^{k+1} M)$ is a weak differential of $\omega$ if it satisfies
\begin{equation}\label{eq:weak_differential_for_forms}
	\int_M \omega \wedge d\eta = (-1)^{k+1} \int_M d\omega \wedge \eta
\end{equation}
for every $\eta \in C^\infty_0(\wedge^{n-k-1} M)$. We may then define the spaces $W^{d, p}(\wedge^k M)$, $W^{d, p}_\loc(\wedge^k M)$, $W^{d, p, q}(\wedge^k M)$ and $W^{d, p, q}_\loc(\wedge^k M)$ in the same manner as in the Euclidean case. It is clear that if $M$ is a domain in $\R^n$, these definitions are equivalent with the previously stated ones. 

We note that in these notes, we do not even define a manifold counterpart for the spaces $W^{1, p}(\wedge^k U)$. However, defining such a manifold counterpart is regardless possible; an interested reader will find an introduction to the spaces $W^{1, p}(\wedge^k M)$, as well as a far more in-depth discussion on the spaces $W^{d, p}(\wedge^k M)$, in the papers \cite{Scott_HodgeTheory} by Scott and \cite{Iwaniec-Scott-Stroffolini} by Iwaniec, Scott and Stroffolini.

We then briefly point out some properties of Sobolev differential forms. First, if $\omega \in W^{d, 1}_\loc(\wedge^k M)$ and $\xi \in C^\infty(\wedge^l M)$, then $\omega \wedge \xi \in W^{d, 1}_\loc(\wedge^{k+l} M)$ with weak differential $d(\omega \wedge \xi) = d\omega \wedge \xi + (-1)^k \omega \wedge d\xi$. The argument is simple: take a test form $\eta \in C^\infty_0(\wedge^{n-k-l-1} M)$, and verify using the definition of weak differential and the Leibniz rule for smooth forms that
\begin{multline*}
\int_M (\omega \wedge \xi) \wedge d\eta
= \int_M \omega \wedge (-1)^l \left( d(\xi \wedge \eta) - d\eta \wedge \xi\right)\\
= (-1)^{k+l+1} \int_M d\omega \wedge \xi \wedge \eta + (-1)^{l+1} \int_U \omega \wedge d\xi \wedge \eta\\
= (-1)^{k+l+1} \int_M (d\omega \wedge \xi + (-1)^k \omega \wedge d\xi) \eta.
\end{multline*}

Second, if $\omega \in L^1_\loc(\wedge^k M)$ has a weak differential $d\omega \in L^1_\loc(\wedge^{k+1} M)$, then $d\omega$ has a weak differential $d d\omega = 0$. We see this using the corresponding result for smooth functions, as
\[
	\int_M d\omega \wedge d\eta = (-1)^{k+1} \int_M \omega \wedge d d \eta = (-1)^{k+1} \int_M \omega \wedge 0 = 0
\]
for every $\eta \in C^\infty_0(\wedge^{n-k-2} M)$.

\subsection{Pull-backs of Sobolev forms with smooth maps}

We then proceed to discuss the interaction of the weak exterior derivative with pull-backs. An important property of smooth differential forms $\omega \in C^\infty(\wedge^k M)$ is that $f^* d\omega = d f^*\omega$ for smooth $f$. It would therefore be desirable to have a counterpart for this in the case of a Sobolev form $\omega$.

We immediately remark that if $\omega \in W^{d,p,q}(\wedge^k N)$, then $f \in C^\infty(M, N)$ isn't enough for $f^* d\omega = d f^* \omega$. Indeed, the problems start way before we can even discuss the weak exterior derivative; a smooth $f$ can have an image of measure zero, in which case changing $\omega$ in a set of measure zero would change $f^* \omega$ everywhere. However, a sufficiently strong condition is that $f$ is smooth and bilipschitz. We formulate this in the following proposition, the proof of which is the main focus of this subsection.

\begin{prop}\label{prop:wdpqloc_bilip_pullback_formula}
	Let $M$ and $N$ be oriented Riemannian $n$-manifolds. Suppose that $\omega \in W^{d, p, q}_\loc(\wedge^k N)$, and that $f \in C^\infty(M, N)$ is $L$-bilipschitz. Then $f^*\omega \in W^{d, p, q}_\loc(\wedge^k M)$, and $f^* d\omega = d f^* \omega$. More precisely, if $\omega \in W^{d, p, q}(N)$, then $f^*\omega \in W^{d, p, q}(M)$, with
	\begin{align*}
		\norm{f^* \omega}_p &\leq \binom{n}{k}^{\frac{1}{2}} L^{k + \frac{n}{p}} \norm{\omega}_p,&
		\norm{f^* d\omega}_q &\leq \binom{n}{k+1}^{\frac{1}{2}} L^{k+1 + \frac{n}{q}} \norm{d\omega}_q.
	\end{align*}
\end{prop}

For the following proof, it is convenient to recall the pointwise comass norm of differential forms. In the author's opinion, this provides the most illuminating approach to pointwise norm estimates of differential forms, which will be used here as well as later on when we develop similar theory for quasiregular maps. Hence, let $M$ be a Riemannian manifold, let $x \in M$, and let $\alpha \in \wedge^k T^*_x M$. Then $\alpha$ is a linear map $\wedge^k T_x M \to \R$. We define the \emph{comass norm} $\abs{\alpha}_{\mass}$ of $\alpha$ by
\[
	\abs{\alpha}_{\mass} = \sup \{ \alpha(v_1 \wedge \dots \wedge v_k) :
	v_i \in T_x M, \abs{v_1 \wedge \dots \wedge v_k} = 1 \}.
\]
The comass norm is therefore similar to the operator norm of $\alpha$, but it is restricted to the values of $\alpha$ on simple $k$-vectors of $\wedge^k T_x M$. Recall that a $k$-vector is \emph{simple} if it can be expressed as a wedge product of elements of $T_x M$.

Simple $k$-vectors have among others the following properties
\begin{lemma}\label{lem:simple_k_vec_props}
	Let $M$ be a Riemannian manifold, let $x \in M$, and let $v \in \wedge^k T_x M$. Suppose that $v$ is simple, and can therefore be written as $v = v_1 \wedge \dots \wedge v_k$, with $v_i \in T_x M$.
	
	\begin{enumerate}
		\item \label{item:simple_k_vec_prop_ortho} The $k$-vector $v$ can also be written as $v = v_1' \wedge \dots \wedge v_k'$, where the $v_i' \in T_x M$ are pairwise orthogonal.
		\item \label{item:simple_k_vec_prop_estimate} We have
		\[
			\abs{v} \leq \abs{v_1} \cdots \abs{v_k},
		\]
		where equality holds if $v_i$ are pairwise orthogonal.
	\end{enumerate}
\end{lemma}
\begin{proof}
	Part \eqref{item:simple_k_vec_prop_ortho} is due to a Gram--Schmidt argument. We first write $v_1' = v_1$. Then we decompose $v_2 = v_2' + a v_1'$, where $v_2'$ is orthogonal to $v_1'$. It follows that $v_1' \wedge v_2 = v_1' \wedge v_2'$, as $v_1' \wedge v_1' = 0$. Then we decompose $v_3 = v_3' + bv_1' + cv_2'$, with $v_3'$ orthogonal to both $v_1'$ and $v_2'$, and similarly obtain $v_1' \wedge v_2' \wedge v_3 = v_1' \wedge v_2' \wedge v_3'$. We continue up to $v_k$, and obtain the desired decomposition.
	
	Part \eqref{item:simple_k_vec_prop_estimate} is then consequence of part \eqref{item:simple_k_vec_prop_ortho}. Namely, 
	\[
		\abs{v} = \sqrt{\ip{v}{v}} = \sqrt{\det{\left(\ip{v_i'}{v_j'}\right)_{ij}}},
	\] 
	where the matrix inside the determinant is diagonal with entries $\abs{v_i'}^2$. We thus obtain 
	\[
		\abs{v} = \abs{v_1'} \abs{v_2'} \cdots \abs{v_k'}.
	\]
	Finally, we note that $\abs{v_i'} \leq \abs{v_i}$ for every $i$, since $v_i - v_i'$ is orthogonal to $v_i'$. Hence, the claim follows.
\end{proof}

We now discuss how the comass norm compares to the usual pointwise norm for differential forms. The following estimate is stated in \cite[Section 1.8]{Federer_book}. However, as the proof is left for the reader there, we provide it for completeness.

\begin{lemma}\label{lem:comass_norm_conversion}
	Let $M$ be a Riemannian manifold, let $x \in M$, and let $\alpha \in \wedge^k T^*_x M$. Then
	\[
		\abs{\alpha}_{\mass} \leq \abs{\alpha} \leq \binom{n}{k}^{\frac{1}{2}} \abs{\alpha}_{\mass}.
	\]
\end{lemma}
\begin{proof}
	Let $\{e_I\}$ be an orthonormal basis of $\wedge^k T_x M$ consisting of simple vectors; such a basis exists by the standard construction $e_I = e_{I_1} \wedge \dots \wedge e_{I_k}$, where $\{e_1, \dots, e_n\}$ is an orthonormal basis of $T_x M$. Let $\{\eps_I\}$ be the corresponding dual basis, in which case $\{\eps_I\}$ is also orthonormal. We write $\alpha = \sum_{I} \alpha_I \eps_I$. Then $\abs{\alpha}^2 = \sum_{I} \alpha_I^2$. Since there are $\binom{n}{k}$ basis elements $\eps_I$, we may estimate
	\[
		\abs{\alpha} \leq \sqrt{\binom{n}{k} \sup_I \alpha_I^2}
		= \sqrt{\binom{n}{k} \sup_I \alpha(e_I)^2}.
	\]
	In particular, since the elements $e_I$ are simple and $\abs{e_I} = 1$, we obtain
	\[
		\abs{\alpha} \leq \binom{n}{k}^{\frac{1}{2}} \abs{\alpha}_{\mass}.
	\]
	
	For the converse direction, let $v \in \wedge^k T_x M$ be simple with $\abs{v} = 1$. By the Gram--Schmidt process, we may create an orthonormal basis $\{e_I'\}$ of $\wedge^k T_x M$ which contains $v$. Let $\{\eps_I'\}$ be the corresponding dual basis of $\wedge^k T_x^* M$, which is also orthonormal. Then by writing $\alpha = \sum_{I} \alpha_I' \eps_I'$, we obtain
	\[
		\abs{\alpha} = \sqrt{\sum_I (\alpha_I')^2} = \sqrt{\sum_I \alpha(e_I')^2}
		\geq \sup_I \alpha(e_I') \geq \alpha(v).
	\]
	By taking the supremum over all such $v$, we obtain that
	\[
		\abs{\alpha}_{\mass} \leq \abs{\alpha}.
	\]
\end{proof}

Using the comass norm, we obtain the following pointwise estimate for bilipschitz pull-backs of measurable differential forms.

\begin{lemma}\label{lem:bilip_pullback_pointwise_estimate}
	Let $M$ and $N$ be oriented Riemannian $n$-manifolds. Suppose that $\omega \in \Gamma(\wedge^k N)$, and $f \colon M \to N$ is $L$-bilipschitz. Then
	\[
		\binom{n}{k}^{-\frac{1}{2}} L^{-k} \abs{\omega_{f(x)}}
		\leq \abs{(f^*\omega)_x}
		\leq \binom{n}{k}^{\frac{1}{2}} L^k \abs{\omega_{f(x)}}
	\]
	for almost every $x \in M$.
\end{lemma}
\begin{proof}
	The case $k = 0$ is trivial, as in that case every term in the chain of inequalities equals $\abs{\omega_{f(x)}}$. For $k \geq 1$, we show that in fact
	\begin{equation}\label{eq:bilip_pointwise_estimate_mass}
		L^{-k} \abs{\omega_{f(x)}}_{\mass}
		\leq \abs{(f^* \omega)_x}_{\mass} \leq L^k \abs{\omega_{f(x)}}_{\mass}
	\end{equation}
	for every $x \in M$. Then the claim follows by Lemma \ref{lem:comass_norm_conversion}.
	
	Let $v \in \wedge^k T_x M$ be simple, in which case we may write $v = v_1 \wedge \dots \wedge v_k$ with $v_i \in T_x M$, and suppose that $\abs{v} = 1$. By Lemma \ref{lem:simple_k_vec_props} part \eqref{item:simple_k_vec_prop_ortho}, we may assume that the vectors $v_i$ are pairwise orthogonal. Moreover, we consequently have $\abs{v} = \abs{v_1} \cdots \abs{v_k}$ by Lemma \ref{lem:simple_k_vec_props} part \eqref{item:simple_k_vec_prop_estimate}.
	
	Now, $(Df(x)v_1) \wedge \dots \wedge (Df(x)v_k)$ is a simple $k$-vector of $\wedge^k T_{f(x)} N$. Moreover, by Lemma \ref{lem:simple_k_vec_props} part \eqref{item:simple_k_vec_prop_estimate}, we have that
	\begin{multline*}
		\abs{(Df(x)v_1) \wedge \dots \wedge (Df(x)v_k)}
		\leq \abs{Df(x)v_1} \cdots \abs{Df(x)v_k}\\
		\leq \abs{Df(x)}^k \abs{v_1} \cdots \abs{v_k}
		= \abs{Df(x)}^k \abs{v}
		= \abs{Df(x)}^k
		\leq L^k.
	\end{multline*}
	Hence, we have that
	\begin{multline*}
		(f^*\omega)_x(v)
		= \omega_{f(x)} \left((Df(x)v_1) \wedge \dots \wedge (Df(x)v_k)\right)\\
		\leq \abs{\omega_{f(x)}}_{\mass} \abs{(Df(x)v_1) \wedge \dots \wedge (Df(x)v_k)}
		\leq \abs{\omega_{f(x)}}_{\mass} \abs{L}^k.
	\end{multline*}
	By taking the supremum over all $v$, we obtain the upper bound of \eqref{eq:bilip_pointwise_estimate_mass}.
	
	For the lower bound, we note that the inverse $f^{-1} \colon fM \to M$ is also an $L$-bilipschitz smooth map. Hence, we have
	\[
		\abs{\omega_{f(x)}}_{\mass} = \abs{((f^{-1})^* f^* \omega)_{f(x)}}_{\mass}
		\leq L^k \abs{(f^* \omega)_{f^{-1}(f(x))}}_{\mass}
		= L^k \abs{(f^* \omega)_{x}}_{\mass}
	\]
	Hence, we obtain the lower bound of \eqref{eq:bilip_pointwise_estimate_mass}.
\end{proof}

The immediate consequence of Lemma \ref{lem:bilip_pullback_pointwise_estimate} is the following estimate.
\begin{cor}\label{cor:bilip_pullback_Lp_estimate}
	Let $M$ and $N$ be oriented Riemannian $n$-manifolds. Suppose that $\omega \in L^p(\wedge^k N)$ for $p \in [1, \infty]$, and that $f \colon M \to N$ is $L$-bilipschitz. Then
	\[
		\norm{f^*\omega}_p
		\leq \binom{n}{k}^{\frac{1}{2}} L^{k + \frac{n}{p}} \norm{\omega}_p,
	\]
	where we interpret $n/\infty = 0$.
\end{cor}
\begin{proof}
	Consider first the case $p = \infty$. Then the estimate reads
	\[
		\norm{f^*\omega}_\infty
		\leq \binom{n}{k}^{\frac{1}{2}} L^{k} \norm{\omega}_\infty.
	\]
	This follows directly from the estimate of \ref{lem:bilip_pullback_pointwise_estimate}, combined with the fact that bilipschitz maps satisfy the Lusin $(N^{-1})$ condition due to being $C^1$ diffeomorphisms onto their image.
	
	Suppose then that $p < \infty$. Now we use the Jacobian estimate \eqref{eq:bilipschitz_J}, diffeomorphic change of variables, as well as Lemma \ref{lem:bilip_pullback_pointwise_estimate} to obtain that
	\begin{align*}
		\int_M \abs{f^* \omega}^p \vol_M 
		&\leq \binom{n}{k}^{\frac{1}{2}} L^{pk} \int_M (\abs{\omega}^p \circ f)\vol_M\\
		&\leq \binom{n}{k}^{\frac{1}{2}} L^{pk + n} 
			\int_M (\abs{\omega}^p \circ f)\abs{J_f}\vol_M\\
		&= \binom{n}{k}^{\frac{1}{2}} L^{pk + n} \int_N \abs{\omega}^p \vol_N.
	\end{align*}
	Raising to a power $1/p$ yields the result.
\end{proof}

We may now begin the proof of Proposition \ref{prop:wdpqloc_bilip_pullback_formula}

\begin{proof}[Proof of Proposition \ref{prop:wdpqloc_bilip_pullback_formula}]
	The correct integrabilities follow from Corollary \ref{cor:bilip_pullback_Lp_estimate}. The only remaining question is whether $f^* d\omega$ is a weak differential of $d f^* \omega$. Hence, suppose that $\eta \in C^\infty_0(\wedge^{n-k-1} M)$, with the goal of showing that
	\[
	\int_M f^*\omega \wedge d\eta = (-1)^{k+1} \int_M f^* d\omega \wedge \eta.
	\]
	 
	Note that $f$ has a smooth inverse $f^{-1} \colon fM \to M$. Let $\mu = \pm 1$ be a coefficient, where $\mu = 1$ if $f$ is orientation preserving, and $\mu = -1$ if $f$ is orientation reversing. Using a diffeomorphic change of variables for $n$-forms, we may compute that
	\begin{multline*}
	\int_M f^* \omega \wedge d\eta 
	= \int_M f^* \omega \wedge f^* (f^{-1})^* d\eta
	= \int_M f^*(\omega \wedge (f^{-1})^* d\eta)\\
	= \mu \int_{fM} \omega \wedge (f^{-1})^* d\eta
	= \mu \int_{fM} \omega \wedge d (f^{-1})^*\eta.
	\end{multline*}
	
	The support of $(f^{-1})^*\eta$ is compact, since it is contained in the compact set $f (\spt \eta)$. Since $fM$ is open in $N$, we may hence extend $(f^{-1})^* \eta$ as zero to the whole manifold $N$, obtaining a compactly supported smooth form. Hence, by definition of the weak differential and another diffeomorphic change of variables,
	\begin{multline*}
	\int_M f^* \omega \wedge d\eta
	= \mu \int_{fM} \omega \wedge d (f^{-1})^*\eta
	= \mu \int_{N} \omega \wedge d (f^{-1})^*\eta\\
	= (-1)^{k+1} \mu \int_{N} d\omega \wedge (f^{-1})^*\eta
	= (-1)^{k+1} \mu \int_{fM} (f^{-1})^* f^* d\omega \wedge (f^{-1})^*\eta\\
	= (-1)^{k+1} \mu \int_{fM} (f^{-1})^* (f^* d\omega \wedge \eta)
	= (-1)^{k+1} \mu^2 \int_{M}  f^* d\omega \wedge \eta.
	\end{multline*}
	Since $\mu^2 = 1$, the claim therefore follows.
\end{proof}

\medskip

With the proof of Proposition \ref{prop:wdpqloc_bilip_pullback_formula} complete, we then tie a loose end from before. Namely, if $f \in W^{1, p}(M, \R)$, then in Section \ref{sect:sobolev_manifold_domain} we defined a map $df \in \Gamma(\wedge^1 M)$ by $df = \iota_1 \circ Df$, where $\iota_1 \colon T\R \to \R$ is again given by the canonical identification of $T_x \R$ with $\R$. However, since $f \in L^1_\loc(M)$, we may also consider a weak differential $df$ by the definitions of this section. In the Euclidean setting, we know these two versions of $df$ to be the same object. We now show that this equivalence holds even for manifold domains.

\begin{lemma}\label{lem:sobolev_0_forms}
	Suppose $M$ is an oriented Riemannian $n$-manifold. If $f \in W^{1,1}_\loc(M)$, then the map $df = \iota_1 \circ Df$ is in $L^{1}_\loc(\wedge^1 M)$ and satisfies
	\begin{equation}\label{eq:weak_differential}
		\int_M df \wedge \omega = -\int_M f d\omega 
	\end{equation}
	for all $\omega \in C^\infty_0(\wedge^{n-1} M)$. Conversely, if $f \in L^1_\loc(M)$ and there exists a $df \in L^1_\loc(M)$ satisfying \eqref{eq:weak_differential}, then $f \in W^{1,1}_\loc(M)$ and $df = \iota_1 \circ Df$.
	
	In particular, we have $W^{1, p}(M, \R) = W^{d, p}(\wedge^0 M)$ and $W^{1, p}_\loc(M, \R) = W^{d, p}_\loc(\wedge^0 M)$.
\end{lemma}
\begin{proof}
	Let $df = \iota_1 \circ Df$, and suppose that $\phi \colon U \to \R^n$ is a bilipschitz chart. Since $\phi$ satisfies the Lusin conditions, we may pull back $df$ by $\phi$ and its inverse. By the definition of $Df$, we see that
	\[
		df\vert_U = \iota_1 \circ Df\vert_U 
		= \iota_1 \circ D(f \circ \phi^{-1}) \circ D\phi
		=\phi^*d(f \circ \phi^{-1}).
	\]
	Since $\phi$ is bilipschitz, Proposition \ref{prop:wdpqloc_bilip_pullback_formula} implies that $df\vert_U$ is the weak differential of $\phi^*(f \circ \phi^{-1}) = f\vert_U$ in the sense of this section. We conclude that $df \in L^{1}_\loc(\wedge^1 U)$, and that $df$ satisfies \eqref{eq:weak_differential} if $\spt \omega \subset U$. By using a smooth partition of unity, it follows that $df \in L^{1}_\loc(\wedge^1 M)$ and $df$ satisfies \eqref{eq:weak_differential} for all $\omega \in C^\infty_0(\wedge^{n-1} M)$.
	
	The converse is by essentially the same calculation in the opposite direction. Hence, we have $W^{1,1}_\loc(M) = W^{d,1}_\loc(\wedge^0 M)$, where the two definitions of $df$ coincide. Finally, since $\iota_{1, y} \colon T_y\R \to \R$ is an isometry for every $y \in \R$, we have that $\abs{Df} \in L^p(M)$ if and only if the map $x \mapsto \sup_{\abs{v} \leq 1} df_x(v)$ is in $L^p(M)$. But since every element of $T_x M$ is simple, we have $\sup_{\abs{v} \leq 1} df_x(v) = \abs{df_x}_{\mass}$, and by Lemma \ref{lem:comass_norm_conversion}, $\abs{df}_{\mass}$ is in $L^p(M)$ if and only if $\abs{df}$ is in $L^p(M)$. The rest of the claim therefore follows.
\end{proof}

Most notably, we have obtained an alternate definition for $W^{1,p}(M, \R^k)$ using \eqref{eq:weak_differential}, since a function is in $W^{1,p}(M, \R^k)$ if and only if all its coordinate functions are in $W^{1,p}(M) = W^{d,p}(\wedge^0 M)$.

\subsection{Approximation and wedge products}

We now wish to prove some more further properties of $W^{d, p, q}$-spaces. For these, a significant tool we require is a counterpart of the smooth approximation result given in Theorem \ref{thm:sobolev_eucl_smooth_density}. We do not aim to show a result of full generality, instead choosing to avoid technicalities for a simpler version which is sufficient for our uses.

We start with a result for the full space $\R^n$. The argument is essentially the same as that of Theorem \ref{thm:sobolev_eucl_smooth_density}, based on convolution with standard mollifiers. However, the fact that we're taking convolutions of differential forms introduces some minor quirks. Hence, we provide a full detailed proof. For another exposition on this fact, see e.g.\ \cite[II.4.2.]{Reshetnyak-book}

\begin{prop}\label{prop:W_dp_smooth_approx}
	Let $\omega \in W^{d, p, q}(\wedge^k \R^n)$, where $0 \leq k \leq n-1$ and $p, q \in [1,\infty)$. Then there exists a sequence $\omega_j \in C^\infty(\wedge^k \R^n)$ for which $\norm{\omega - \omega_j}_p \to 0$ and $\norm{d\omega - d\omega_j}_q \to 0$. Moreover, if $\spt \omega$ is compact, then we may assume that the supports of $\omega_j$ are contained in any fixed neighborhood $U \subset \R^n$ of $\spt \omega$.
\end{prop}

We start by defining a convolution for $k$-forms and $C^\infty_0$-functions. Let $\omega \in L^1_\loc(\wedge^k \R^n)$, and $\eta \in C^\infty_0(\R^n)$. We define a smooth $k$-form $\omega \ast \eta \in C^\infty_0(\wedge^k \R^n)$ by
\[
	\omega \ast \eta = \sum_I (\omega_I \ast \eta) \eps_I.
\]
Recall that the convolution operator $f \ast g$ of functions $f, g \colon \R^n \to \R$ is defined by $(f \ast g)(x) = \int_{\R^n} f(y) g(x-y) \dd y$, whenever the integral is well defined. The convolution is commutative, associative, and bilinear. Moreover, if $f$ is differentiable, then $f \ast g$ is differentiable with $\partial_i (f \ast g) = (\partial_i f) \ast g$, provided the right hand side is well defined for all partial derivatives.

We now record two specific formulae for convolutions of differential forms.
\begin{lemma}\label{lem:form_convolution_rules}
	Let $\eta \in C^\infty(\wedge^{k} \R^n)$ and $f \in C^\infty_0(\R^n)$. Then
	\[
		d(\eta \ast f) = (d\eta) \ast f.
	\]
	Moreover, suppose that also $\alpha \in L^1_\loc(\wedge^{n-k} \R^n)$ and that $\eta$ is compactly supported. Then
	\[
		\int_{\R^n} (\eta \ast f) \wedge \alpha
		= \int_{\R^n} \eta \wedge (\alpha \ast \tilde{f}),
	\]
	where $\tilde{f}(x) = f(-x)$.
\end{lemma}
\begin{proof}
	The first claim is by essentially writing out the formula for $d\eta$ out with respect to the standard basis. In particular,
	\[
		(d\eta)_I = \sum_{j = 1}^{k+1} (-1)^{c(I, j)} \partial_{I_j} \eta_{I \setminus I_j},
	\]
	where $c(I, j)$ is either 0 or 1 depending on the $(k+1)$-element subset $I \subset \{1, \dots, n\}$ and the element $I_j$ we're removing from it. Now, with a simple (but notationally nasty) computation, we obtain
	\begin{align*}
		&((d\eta) \ast f)_I 
		= (d\eta)_I \ast f\
		= \Bigg(\sum_{j = 1}^{k+1} (-1)^{c(I, j)} \partial_{I_j} 
			\eta_{I \setminus I_j}\Bigg) \ast f\\
		&\quad= \sum_{j = 1}^{k+1} (-1)^{c(I, j)} \partial_{I_j} 
			\left(\eta_{I \setminus I_j} \ast f\right)
		= \sum_{j = 1}^{k+1} (-1)^{c(I, j)} \partial_{I_j} 
			(\eta \ast f)_{I \setminus I_j}
		= \left(d(\eta \ast f)\right)_{I}.
	\end{align*}
	
	For the second formula, we again write $\eta \ast f$ and $\alpha$ with respect to the standard basis, and note that in the wedge product only pairs of complementary basis elements don't get annihilated. For an index $I$, we denote the complementary index $\{1, \dots, n\} \setminus I$ by $I^*$. Hence, we have
	\[
		\int_{\R^n} (\eta \ast f) \wedge \alpha
		= \sum_I (-1)^{c(I)} \int_{\R^n} (\eta_I \ast f)(x) \alpha_{I^*}(x) \dd x.
	\]
	By writing the convolutions out, we have
	\[
		\int_{\R^n} (\eta \ast f) \wedge \alpha
		= \sum_I (-1)^{c(I)} \int_{\R^n} \int_{\R^n} 
			\eta_I(y) \alpha_{I^*}(x) f(x-y) \dd y \dd x.
	\]
	Similarly,
	\begin{align*}
		\int_{\R^n} \eta \wedge (\alpha \ast \tilde{f})
		&= \sum_I (-1)^{c(I)} \int_{\R^n} \int_{\R^n} 
			\eta_I(y) \alpha_{I^*}(x) \tilde{f}(y-x) \dd x \dd y\\
		&= \sum_I (-1)^{c(I)} \int_{\R^n} \int_{\R^n} 
		\eta_I(y) \alpha_{I^*}(x) f(x-y) \dd x \dd y.
	\end{align*}
	Hence, the claim follows from using Fubini to change the integration order, which is valid since the integrand is $L^1$.
\end{proof}

The purpose of the specific formulae in Lemma \ref{lem:form_convolution_rules} is to set up the following corollary.

\begin{cor}\label{cor:form_convolution_weak_diff}
	Let $\omega \in W^{d, 1}_\loc(\wedge^k \R^n)$ and $f \in C^\infty_0(\R^n)$. Then
	\[
		d(\omega \ast f) = (d\omega) \ast f.
	\]
\end{cor}
\begin{proof}
	Let $\eta \in C^\infty_0(\wedge^{n-k-1} \R^n)$, and let $\tilde{f} \in C^\infty_0(\R^n)$ again denote the function $x \mapsto f(-x)$. We compute that
	\begin{multline*}
		\int_{\R^n} ((d\omega) \ast f) \wedge \eta
		= \int_{\R^n} d\omega \wedge (\eta \ast \tilde{f})
		= \int_{\R^n} \omega \wedge d(\eta \ast \tilde{f})\\
		= \int_{\R^n} \omega \wedge ((d\eta) \ast \tilde{f})
		= \int_{\R^n} (\omega \ast f) \wedge (d\eta).
	\end{multline*}
	Here, the first, third and last equality are all due to Lemma \ref{lem:form_convolution_rules}, while the second one is by the definition of a weak differential. Hence, $(d\omega) \ast f$ is a weak differential of $\omega \ast f$. By almost everywhere uniqueness of the weak differential and smoothness of both forms, the claim follows.	
\end{proof}

Now we finally have the tool to prove Proposition \ref{prop:W_dp_smooth_approx}.

\begin{proof}[Proof of Proposition \ref{prop:W_dp_smooth_approx}]
	We take a sequence of standard mollifiers $\sigma_i \in C^\infty_0(\R^n)$, which we've selected such that $\spt \sigma_i + \spt \omega \subset U$ if $\spt \omega$ is compact, and we define $\omega_i = \omega \ast \sigma_i$. Since every coordinate function $\omega_I$ is in $L^p(\R^n)$, the standard properties of mollifiers imply that $\norm{\omega_I - (\omega_I \ast \sigma_i)}_p \to 0$ for every $I$; see for example \cite[Theorem 2.29]{Adams-Fournier_Sobolev}. Hence, $\norm{\omega - \omega_i}_p \to 0$.
	
	If we similarly define $\tau_i = (d\omega) \ast \sigma_i$, then we have by the same argument that $\norm{d\omega - \tau_i}_q \to 0$. Now the claim follows, since by Corollary \ref{cor:form_convolution_weak_diff} we have $\tau_i = d\omega_i$.
\end{proof}

We then give a version that works on manifolds. We settle on an easily attainable corollary for subdomains which are compactly contained in a manifold; this tends to be enough for applications, as it can be used to cover the support of a compactly supported test form $\eta$. More general results do however hold; see e.g.\ \cite[Section 3]{Iwaniec-Scott-Stroffolini}.

\begin{cor}\label{cor:W_dp_smooth_approx_mfld}
	Let $M$ be an oriented Riemannian manifold, and let $\omega \in W^{d, p, q}(\wedge^k M)$, where $0 \leq k \leq n-1$ and $p, q \in [1,\infty)$. Suppose that $U$ is a domain compactly contained in $M$. Then there exists a sequence $\omega_j \in C^\infty_0(\wedge^k M)$ for which $\norm{\omega\vert_U - \omega_j\vert_U}_p \to 0$ and $\norm{d\omega\vert_U - d\omega_j\vert_U}_q \to 0$.
\end{cor}
\begin{proof}
	We select an intermediary domain $V$ which contains $\overline{U}$ and is compactly contained in $M$. We can cover $\overline{U}$ with finitely many bilipschitz chart neighborhoods $U_i \subset V$ of charts $\phi_i \colon U_i \to \R^n$. By multiplying $\omega$ with smooth cutoff functions, we find forms $\omega_i \in W^{d, p, q}(\wedge^k M)$ such that $\spt \omega_i \subset U_i$ and $\omega\vert_U = (\sum_i \omega_i)\vert_U$.
	
	In this case each $(\phi_i^{-1})^* \omega_i$ extends to a compactly supported form in $W^{d, p, q}(\wedge^k \R^n)$ by Proposition \ref{prop:wdpqloc_bilip_pullback_formula}. By Proposition \ref{prop:W_dp_smooth_approx}, we may approximate $(\phi_i^{-1})^* \omega_i$ in $W^{d, p, q}(\wedge^k \R^n)$ with $\alpha_{i, j} \in C^\infty_0(\wedge^k \R^n)$ satisfying $\spt \alpha_{i, j} \subset \phi_i U_i$. We can extend $\phi_i^* \alpha_{i, j}$ to the entire manifold $M$ due to their compact supports, and again by Proposition \ref{prop:wdpqloc_bilip_pullback_formula}, we have that $\phi_i^* \alpha_{i, j} \to \omega_i$ in $W^{d, p, q}(\wedge^k M)$. We conclude that $(\sum_i \alpha_{i, j})\vert_U \to \omega\vert_U$ in $W^{d, p, q}(\wedge^k U)$.
\end{proof}

We then use Corollary \ref{cor:W_dp_smooth_approx_mfld} to show that wedge products of weakly differential forms are weakly differentiable, assuming sufficient local integrability.

\begin{prop}\label{prop:wedges_of_sobolev_forms}
	Let $M$ be an oriented Riemannian manifold, let $\omega \in W^{d, p, q}_\loc(\wedge^k M)$, and let $\omega' \in W^{d, p', q'}_\loc(\wedge^{k'} M)$, where $k + k' \leq n - 1$ and $p, p', q, q' \in [1, \infty)$. Let $a, b$ be given by
	\begin{align*}
		\frac{1}{a} &= \frac{1}{p} + \frac{1}{p'},&
		\frac{1}{b} &= \max \left( \frac{1}{p} + \frac{1}{q'}, \frac{1}{q} + \frac{1}{p'} \right). 
	\end{align*}
	If $a, b \geq 1$, then $\omega \wedge \omega' \in W^{d, a, b}_\loc(\wedge^{k+k'} M)$ and
	\[
		d(\omega \wedge \omega') = d\omega \wedge \omega' + (-1)^k \omega \wedge d\omega'.
	\]
\end{prop}
\begin{proof}
	We note that 
	\begin{equation}\label{eq:wedge_estimate}
		\abs{\omega \wedge \omega'} \leq C \abs{\omega} \abs{\omega'},
	\end{equation} 
	where $C$ depends on $n$, $k$ and $k'$. Indeed, if $v = v_1 \wedge \dots \wedge v_{k+k'}$ is simple, then we recall that 
	\[
		(\omega \wedge \omega')(v) = \sum_{\sigma} (-1)^{\operatorname{sgn}(\sigma)} \omega(v_{\sigma(1)} \wedge \dots \wedge v_{\sigma(k)}) \omega'(v_{\sigma(k+1)} \wedge \dots \wedge v_{\sigma(k+k')}),
	\]
	where the sum is over all permutations $\sigma$ of $\{1, \dots, k+k'\}$ which are increasing on $\{1, \dots, k\}$ and $\{k+1, \dots, k+k'\}$. It therefore follows, keeping in mind Lemma \ref{lem:simple_k_vec_props}, that $\abs{\omega \wedge \omega'}_{\mass} \leq C' \abs{\omega}_{\mass} \abs{\omega'}_{\mass}$, where $C'$ is the number of permutations in the sum. The estimate \eqref{eq:wedge_estimate} then follows by Lemma \ref{lem:comass_norm_conversion}.
	
	We show the case $a^{-1} =  p^{-1} + (q')^{-1} \geq  q^{-1} + (p')^{-1}$, as the other case is analogous. Let $U$ be a domain compactly contained in $M$. We observe that $\omega \wedge \omega'$ and the candidate differential have the correct integrabilities on $U$, as H\"older's inequality and \eqref{eq:wedge_estimate} yield
	\[
		\norm{\omega \wedge \omega'}_a \leq C \norm{\omega}_p \norm{\omega'}_{p'}
	\]
	and
	\begin{multline*}
		\norm{d\omega \wedge \omega' + (-1)^k \omega \wedge d\omega'}_b\\
		\leq C_1 \vol_M(U)^{\frac{1}{b} - \frac{1}{q} - \frac{1}{p'}} \norm{d\omega}_q \norm{\omega'}_{p'} + C_2 \norm{\omega}_p \norm{d\omega'}_{q'},
	\end{multline*}
	where the norms are over $U$. Hence, we have the correct local integrabilities over $M$.
	
	Let then $\eta \in C^\infty_0(\wedge^{n-k-k'-1} U)$. By Corollary \ref{cor:W_dp_smooth_approx_mfld}, we may select $\omega_i \in C^\infty_0(\wedge^k M)$ which approximate $\omega$ in $W^{d,p,q}(\wedge^k U)$. It follows that $d(\omega_i \wedge \omega') = d\omega_i \wedge \omega' + (-1)^k \omega_i \wedge d\omega'$ weakly, and therefore
	\[
		\int_U (d\omega_i \wedge \omega' + (-1)^k \omega_i \wedge d\omega') \wedge \eta = (-1)^{k+k'+1} \int_U \omega_i \wedge \omega' \wedge d\eta.
	\]
	However, we then use \eqref{eq:wedge_estimate} and H\"older's inequality to estimate
	\[
		\abs{\int_U \omega_i \wedge \omega' \wedge d\eta - \int_U \omega \wedge \omega' \wedge d\eta}
		\leq C_3 \norm{\omega_i - \omega}_p \norm{\omega'}_{p'} \norm{d\eta}_{\frac{a}{a-1}}
	\]
	and similarly
	\begin{multline*}
		\abs{\int_U (d\omega_i \wedge \omega' + (-1)^k \omega_i \wedge d\omega') \wedge \eta 
			- \int_U (d\omega \wedge \omega' + (-1)^k \omega \wedge d\omega') \wedge \eta}\\
		\leq C_4 \vol_M(U)^{\frac{1}{b} - \frac{1}{q} - \frac{1}{p'}} 	
			\norm{d\omega_i - d\omega}_q \norm{\omega'}_{p'} \norm{\eta}_{\frac{b}{b-1}}\\
		+ C_5 \norm{\omega_i - \omega}_p \norm{d\omega'}_{q'} 	
			\norm{\eta}_{\frac{b}{b-1}},
	\end{multline*}
	where the norms are again over $U$. By letting $i \to \infty$, we therefore have
	\[
		\int_U (d\omega \wedge \omega' + (-1)^k \omega \wedge d\omega') \wedge \eta = (-1)^{k+k'+1} \int_U \omega \wedge \omega' \wedge d\eta.
	\] 
	Since the support of every test form $\eta \in C^\infty_0(\wedge^{n-k-k'-1} M)$ is contained in such a neighborhood $U$, we conclude that $d(\omega \wedge \omega') = d\omega \wedge \omega' + (-1)^k \omega \wedge d\omega'$ weakly. 
\end{proof}

\subsection{Pull-backs of smooth forms with Sobolev maps}

To end this section, we then discuss the other interaction of Sobolev forms and pull-backs. Namely, suppose that $M$ and $N$ are oriented Riemannian manifolds. When do we have $f^* d \omega = d f^* \omega$ weakly if $f \in \cW^{1, p}(M, N)$ and $\omega \in C^\infty(\wedge^k N)$? We focus on the case where $\omega$ is compactly supported.

Given a $f \in \cW^{1,1}_\loc(M, N)$ with weak derivative $Df \colon TM \to TN$ and a continuous $k$-form $\omega$, we may define $f^* \omega$ by $f^* \omega = \omega \circ (\wedge^k Df)$. It is immediately clear that changing $Df$ in a set of measure zero changes $f^* \omega$ only in a set of measure zero. The other foundational question is whether $f^* \omega$ is measurable.

\begin{lemma}\label{lem:measurability_and_wedges}
	Let $M$ and $N$ be smooth orientad Riemannian manifolds, and suppose that $f \in \cW^{1,1}_\loc(M, N)$ and $\omega \in C(\wedge^k N)$. Then $f^*\omega \in \Gamma(\wedge^k M)$.
\end{lemma}
\begin{proof}
	The case $k = 0$ is simple, as then $f^* \omega = \omega \circ f$, which is measurable as a composition of a continuous and a measurable function. The case $k = 1$ similarly follows using Lemma \ref{lem:two_measurability_defs_for_forms}: $f^* \omega = \omega \circ Df$, which again is a composition of a continuous and a measurable function.
	
	Consider then the remaining case $k > 1$. Then every $x \in N$ has a neighborhood $U \subset N$ with the following property: we have $\omega = \sum_I \omega_I$ on $U$, where $\omega_I$ are simple continuous forms $\omega_I = \omega_{I, 1} \wedge \dots \wedge \omega_{I, k}$ with $\omega_{I, j} \in C(\wedge^1 N)$. By the previous case, the forms $f^* \omega_{I, j}$ are measurable, and hence, we see that the form
	\[
		f^*\left( \sum_I \omega_I \right)
		= \sum_I (f^* \omega_{I, 1}) \wedge \dots \wedge (f^* \omega_{I, k}) 
	\]
	is measurable.
	
	Suppose then that $V \subset \R$ is open, with the goal of showing that the pre-image $(f^* \omega)^{-1} V \subset \wedge^k TM$ is measurable. Then if $E$ is a borel set contained in a neighborhood $U$ as above, then the set
	\[
		(f^{-1} E) \cap (f^* \omega)^{-1} V = (f^{-1} E) \cap \biggl(f^*\Bigl( \sum_I \omega_I \Bigr)\biggr)^{-1} V
	\]
	is measurable by Lemma \ref{lem:two_measurability_defs_for_forms}. The measurability of $(f^* \omega)^{-1} V$ then follows by taking a countable partition of $N$ by Borel sets $E$ as above. Hence, $f^* \omega$ is measurable by Lemma \ref{lem:two_measurability_defs_for_forms}.
\end{proof}

We then consider the formula $f^* d\omega = d f^* \omega$ when $f$ is Sobolev and $\omega$ is smooth and compactly supported. We prove first the Euclidean case.

\begin{lemma}\label{lem:smooth_form_sobolev_pullback_eucl}
	Suppose that $\omega \in C^\infty_0(\wedge^k V)$, where $V$ is a domain in $\R^n$. Suppose that $f \in W^{1, p}(U, V)$ where $U$ is a domain in $\R^m$ and $p \in [1, \infty]$. If $p \geq k+1$, then $f^*\omega \in W^{d, p/k, p/(k+1)}(\wedge^k U)$ and $d f^* \omega = f^* d \omega$.
\end{lemma}

\begin{proof}
	We first show that $f^* \omega$ and $f^* d\omega$ have the correct degree of integrability. Indeed, suppose that $v = v_1 \wedge \dots \wedge v_k \in \wedge^k T_xM$ is simple with $\abs{v_i} = 1$ for every $i$. Note that since $\omega$ is compactly supported, $\omega$ and $d\omega$ both have a finite $L^\infty$-norm. By Lemmas \ref{lem:simple_k_vec_props} and \ref{lem:comass_norm_conversion}, we have
	\begin{multline*}
		(f^* \omega)_x (v) 
		= \omega_{f(x)} (Df(x) v_1 \wedge \dots \wedge Df(x) v_k) \\
		\leq \abs{\omega_{f(x)}}_{\mass} \abs{Df(x)v_1} \dots \abs{Df(x)v_k}
		\leq \binom{n}{k}^\frac{1}{2} \norm{\omega}_\infty \abs{Df}^k.
	\end{multline*}
	Since $\abs{Df}^k \in L^{p/k}(U)$, we have $f^* \omega \in L^{p/k}(\wedge^k U)$. Similarly, if also $v_{k+1} \in T_x M$ with $\abs{v_{k+1}} \leq 1$, we may compute
	\[
		(f^* d\omega)_x (v\wedge v_{k+1}) 
		\leq \binom{n}{k+1}^\frac{1}{2} \norm{d\omega}_\infty \abs{Df}^{k+1},
	\]
	and therefore $f^* d\omega \in L^{p/(k+1)}(\wedge^{k+1} U)$. 
	
	The remaining part is therefore confirming that $f^* d\omega$ is a weak differential of $f^* \omega$. It suffices to show this for a standard basis component of $\omega$. Hence, we may assume $\omega = \omega_0 \, dx_{i_1} \wedge \dots \wedge dx_{i_k}$ where $\omega_0$ is smooth and compactly supported.
	
	Since $\omega_0$ is compactly supported, it is Lipschitz. Hence, we have by Corollary \ref{cor:sobolev_composition_local} that $\omega_0 \circ f \in W^{1,p}_\loc(U)$ and $D(\omega_0 \circ f) = D\omega_0 \circ Df$. Lemma \ref{lem:sobolev_0_forms} then implies that $\omega_0 \circ f \in W^{d, p}_\loc(U)$, and $d f^* \omega_0 = \iota_1 \circ D(\omega_0 \circ f) = \iota_1 \circ D\omega_0 \circ Df = f^* d\omega_0$.
	
	Similarly, the functions $x \mapsto x_{i_j}$ are smooth and 1-Lipschitz. Hence, as before, we have by Corollary \ref{cor:sobolev_composition_local} and Lemma \ref{lem:sobolev_0_forms} that $f^* dx_i = d f_i$ weakly. Since the weak differential of a weak differential is zero, we have $d f^* dx_i = 0$, and in particular, $f^* dx_i \in W^{d, p}_\loc(\wedge^1 U)$. Now, by using the Leibniz formula of Proposition \ref{prop:wedges_of_sobolev_forms}, it follows that
	\[
		d f^*(\omega_0 \, dx_{i_1} \wedge \dots \wedge dx_{i_k})
		= f^* d (\omega_0 \, dx_{i_1} \wedge \dots \wedge dx_{i_k}).
	\]
	The proof is therefore complete.
\end{proof}

We then give the manifold version of the previous result. We include the local version in the statement.

\begin{prop}\label{prop:smooth_form_sobolev_pullback}
	Suppose that $M$ and $N$ are oriented Riemannian manifolds. Let $\omega \in C^\infty_0(\wedge^k N)$, and suppose that $f \in \cW^{1, p}(M, N)$, or respectively $f \in \cW_\loc^{1, p}(M, N)$, where $p \in [1, \infty]$. If $p \geq k+1$, then $f^*\omega \in W^{d, p/k, p/(k+1)}(\wedge^k M)$, or respectively $f^* \omega \in W^{d, p/k, p/(k+1)}_\loc(\wedge^k M)$, where the weak differential is given by $d f^* \omega = f^* d \omega$. Moreover, we have the estimates
	\begin{align*}
		\norm{f^* \omega}_\frac{p}{k} 
		&\leq \binom{n}{k}^{\frac{1}{2}} \norm{Df}_{p}^k 
			\norm{\omega}_\infty,&
		\norm{f^* d\omega}_\frac{p}{k+1} 
		&\leq \binom{n}{k+1}^{\frac{1}{2}} \norm{Df}_{p}^{k+1} 	
			\norm{d\omega}_\infty.
	\end{align*}
\end{prop}
\begin{proof}
	The integrability estimates follow exactly as in the beginning of the proof of Lemma \ref{lem:smooth_form_sobolev_pullback_eucl}. The local result also follows from the global one by restricting to subdomains, keeping in mind Lemma \ref{lem:colocal_sobolev_restricion_characterization}. The main question is therefore showing that $d f^* \omega = f^* d \omega$ weakly in the global case.
	
	For every $y \in N$, we may select neighborhoods $V \subset N$ and $V' \subset \R^n$ of $y$ and $0$, respectively, and smooth maps $\psi_V \colon N \to \R^n$ and $\psi'_V \colon \R^m \to N$, for which $\spt \psi_V$ is compact and $\psi_V \vert_V \colon V \to V'$ is a bilipschitz chart with inverse $\psi_V'\vert_{V'} \colon V' \to V$. Indeed, this is again done by taking a standard smooth bilipschitz chart, and then modifying it and its inverse with suitable cutoffs. The argument is gone through in detail in \cite[Lemma 1.3]{Convent-VanSchaftingen_Sobolev}.
	
	Since $\spt \omega$ is compact, we may use smooth cutoff functions to decompose $\omega = \sum_{i=1}^l \omega_i$, where $\spt \omega_i$ is contained in the neighborhood $V_i \subset N$ of an extended bilipschitz chart $(V_i, V_i', \psi_i, \psi_i')$ as described above. Since there are only finitely many $\omega_i$, it suffices to show that $f^* d\omega_i = d f^* \omega_i$.
	
	Let $\phi \colon U \to \R^m$ be a bilipschitz chart for which $U$ is compactly contained in $M$. By the locality of the weak differential, it suffices to show that $f^* d\omega_i = d f^* \omega_i$ on $U$. By the definitions of the colocal weak derivative and the space $\cW^{1,p}(M, N)$, it follows that $\psi \circ f \in W^{1,1}_\loc(M, \R^n)$ and $\abs{D(\psi \circ f)} \leq \norm{D\psi}_\infty \abs{Df} \in L^p(M)$. Corollary \ref{cor:sobolev_mfld_dom_base_exponent} thus yields $\psi \circ f \in W^{1,p}_\loc(M, \R^n)$. Consequently, since $U$ is compactly contained in $M$, we have $(\psi \circ f)\vert_U \in W^{1,p}_\loc(U, \R^n)$. It then follows by Lemma \ref{lem:sobolev_precomposition_mfld_domain} that $\psi \circ f \circ \phi^{-1} \in W^{1,p}(\phi U, \R^n)$. 
	
	We then write
	\[
		(f^* d \omega_i)\vert_U = \phi^* (\psi \circ f \circ \phi^{-1})^* (\psi')^* d\omega_i,
	\]
	and move the differential on the right side of the above equation across all the pull-backs; first by the smoothness of $\psi'$ and $\omega_i$, then by Lemma \ref{lem:smooth_form_sobolev_pullback_eucl}, and then by Proposition \ref{prop:wdpqloc_bilip_pullback_formula}. It follows that $f^* d\omega_i = d f^* \omega_i$ on $U$, which completes the proof.
\end{proof}

We note that if $f$ is continuous, it is easy to rid ourselves of the need to assume a compact support for $\omega$.

\begin{cor}\label{cor:smooth_form_sobolev_pullback_cont}
	Suppose that $M$ and $N$ are oriented Riemannian manifolds. Let $\omega \in C^\infty(\wedge^k N)$, and suppose that $f \in C(M, N) \cap W^{1, p}_\loc(M, N)$. If $p \geq k + 1$, then $f^* \omega \in W^{d, p/k, p/(k+1)}_\loc(\wedge^k M)$, where the weak differential is given by $d f^* \omega = f^* d \omega$.
\end{cor}
\begin{proof}
	Let $V$ be an open set which is compactly contained in $N$. Then there exists a form $\omega' \in C^\infty_0(\wedge^k N)$ such that $\omega' = \omega$ on $V$. Now, since $f$ is continuous, $f^{-1} V$ is open, and thus a submanifold of $M$. We may therefore use Proposition \ref{prop:smooth_form_sobolev_pullback} on $V$ to conclude that $(f^* \omega)\vert_{f^{-1} V} = (f^* \omega')\vert_{f^{-1} V} \in W^{d, p/k, p/(k+1)}(\wedge^k f^{-1} V)$ and $d(f^* \omega)\vert_{f^{-1} V} = (f^* d\omega)\vert_{f^{-1} V}$. Since neighborhoods $f^{-1} V$ as above form a cover of $M$, the claim follows.
\end{proof}

%%%%%%%%%%%%%%%%%%%%%%%%%%%%%%%%%%%%%%%%%%%%%%%%%%%%%%%%%%%%%%%%%%%%%%%%%%%%%%%%%%%%%%%%%%%
\section{The definitions of quasiregularity}\label{sect:definitions_of_QR}

In this section, we define quasiregular maps between manifolds. We state four variants of the analytic definition, and show that they're equivalent. The first three definitions are by the three different local Sobolev spaces discussed in sections \ref{sect:sobolev_cont_functions}, \ref{sect:sobolev_nash_embedding} and \ref{sect:sobolev_intrinsic_derivative}; the work of showing their equivalence is already done in Theorem \ref{thm:local_sobolev_eq_of_defs}. The last one is by directly using charts and Euclidean quasiregular maps.

\subsection{Statement of definitions}

We begin by defining the \emph{Jacobian determinant}, or \emph{Jacobian} in short, of a Sobolev map between spaces of the same dimension. Suppose that $M$ and $N$ are oriented Riemannian $n$-manifolds, and $f \colon M \to N$ is in any of the Sobolev spaces shown equivalent in Theorem \ref{thm:local_sobolev_eq_of_defs}. Then $f^* \vol_N$ is a measurable $n$-form by Lemma \ref{lem:measurability_and_wedges}. We then define $J_f = \ip{f^* \vol N}{\vol M}$. This is measurable, since the Grassmann inner product with $\vol_M$ is continuous. It also follows that $J_f$ satisfies the standard characterization of Jacobians of $C^1$ maps: $f^* \vol_N = J_f \vol_M$.

Following this definition of the Jacobian, we state the first definition of quasiregular maps between manifolds, which is by using the Sobolev space of continuous maps. 

\begin{defn}\label{def:QR_with_continuous_Sobolev_space}
	Let $M$, $N$ be smooth, oriented Riemannian $n$-manifolds, and let $f \colon M \to N$. Then $f$ is \emph{$K$-quasiregular}, $K \geq 1$, if the following conditions hold:
	\begin{itemize}
		\item $f$ is continuous;
		\item $f \in W^{1, n}_\loc(M, N)$ in the sense of Definition \ref{def:local_Sobolev_for_continuous_maps}.
		\item $\abs{Df(x)}^n \leq K J_f(x)$ for $\hausd^n$-almost every $x \in M$.
	\end{itemize}
\end{defn}

The second definition is by using a Nash embedding.

\begin{defn}\label{def:QR_with_nash_embedding}
	Let $M$, $N$ be smooth, oriented Riemannian $n$-manifolds, and let $f \colon M \to N$. Then $f$ is \emph{$K$-quasiregular}, $K \geq 1$, if the following conditions hold:
	\begin{itemize}
		\item $f$ is continuous;
		\item $f \in W^{1, n}_{\iota, \loc}(M, N)$ for a Nash embedding $\iota \colon N \to \R^m$.
		\item $\abs{Df(x)}^n \leq K J_f(x)$ for $\hausd^n$-almost every $x \in M$, where $Df$ is the unique measurable map satisfying $D(\iota \circ f) = D\iota \circ Df$ almost everywhere.
	\end{itemize}
\end{defn}

The third definition is by using colocal weak differentiability.

\begin{defn}\label{def:QR_with_colocal_diff}
	Let $M$, $N$ be smooth, oriented Riemannian $n$-manifolds, and let $f \colon M \to N$. Then $f$ is \emph{$K$-quasiregular}, $K \geq 1$, if the following conditions hold:
	\begin{itemize}
		\item $f$ is continuous;
		\item $f \in \cW^{1, n}_{\loc}(M, N)$.
		\item $\abs{Df(x)}^n \leq K J_f(x)$ for $\hausd^n$-almost every $x \in M$.
	\end{itemize}
\end{defn}

The last definition is based on directly using the definition of quasiregular maps between Euclidean domains.

\begin{defn}\label{def:QR_with_charts}
	Let $M$, $N$ be smooth, oriented Riemannian $n$-manifolds, and let $f \colon M \to N$. Then $f$ is \emph{$K$-quasiregular}, $K \geq 1$, if for every $L > 1$, every $K' > K$ and every $x \in M$, there exist smooth $L$-bilipschitz charts $\phi \colon U \to \R^n$ and $\psi \colon V \to \R^n$, where $x \in U \subset M$, $f(x) \in V \subset N$, $f(U) \subset V$, and $\psi \circ f \circ \phi^{-1} \colon \phi(U) \to \R^n$ is $K'$-quasiregular.
\end{defn}

Note that the quantifier "for every $L > 1$" in the above definition is crucial; consider eg. the map $f \colon \R^2 \to \R^2$ defined by $f(x,y) = (x, 2y)$ and the bilipschitz chart $\phi(x, y) = (2x, y)$ on the entire space $\R^2$.

\begin{rem}
	We briefly comment on the use of various definitions. The Nash embedding approach of Definition \ref{def:QR_with_nash_embedding} is used in e.g.\ \cite{Okuyama-Pankka_measure} and \cite{Goldstein-Hajlasz-Pakzad}. The direct reduction to the Euclidean case given in Definition \ref{def:QR_with_charts} is also relatively common, and is used in e.g.\ \cite{Holopainen-Rickman_Picard} and \cite{Kangaslampi_thesis}.
	
	For the post-composition approach of Definition \ref{def:QR_with_colocal_diff}, we're only aware of it being used in the author's own paper \cite{Kangasniemi-Pankka_PLMS}. We're also only aware of Definition \ref{def:QR_with_continuous_Sobolev_space} appearing as-is in the lecture notes \cite{Pankka_Lectnotes_Fribourg}. However, for instance the definition in \cite[I.5.2]{Reshetnyak-book} is essentially a blend of Definitions \ref{def:QR_with_continuous_Sobolev_space} and \ref{def:QR_with_charts}. Moreover, not all papers on quasiregular maps between manifolds specify how the Sobolev space is defined, and we suspect that in some of these cases, an approach analogous to that of Definition \ref{def:QR_with_continuous_Sobolev_space} is the one intended.
\end{rem}

\subsection{Equivalence of definitions}

We then show that all of the above definitions are equivalent

\begin{thm}\label{thm:equivalence_of_definitions}
	Definitions \ref{def:QR_with_continuous_Sobolev_space}, \ref{def:QR_with_nash_embedding}, \ref{def:QR_with_colocal_diff} and \ref{def:QR_with_charts} are equivalent.
\end{thm}
\begin{proof}
	Theorem \ref{thm:local_sobolev_eq_of_defs} already provides equivalence of Definitions \ref{def:QR_with_continuous_Sobolev_space}, \ref{def:QR_with_nash_embedding} and \ref{def:QR_with_colocal_diff}. For the remaining part, we show that Definition \ref{def:QR_with_charts} is equivalent to Definition \ref{def:QR_with_continuous_Sobolev_space}.
	
	\emph{\ref{def:QR_with_continuous_Sobolev_space} $\implies$ \ref{def:QR_with_charts}.} Suppose $f$ satisfies Definition \ref{def:QR_with_continuous_Sobolev_space}, and let $x \in M$, $K' > K$. Let $L \leq \sqrt[4n]{K'/K}$. Then, since $f$ is continuous, there exist smooth orientation preserving $L$-bilipschitz charts $\phi \colon U \to \R^n$ and $\psi \colon V \to \R^n$, where $x \in U \subset M$, $f(x) \in V \subset N$ and $f(U) \subset V$; this is again due to Theorem \ref{thm:bilipschitz_charts}.
	
	It follows by Lemma \ref{lem:sobolev_composition_mfld_targ_cont} that the map $\psi \circ f \circ \phi^{-1}$ is in $W^{1,n}_\loc(\phi(U), \R^n)$. In addition, $D(\psi \circ f \circ \phi^{-1}) = D\psi \circ Df \circ D\phi^{-1}$, from which it follows that $J_{\psi \circ f \circ \phi^{-1}} = (J_\psi \circ f \circ \phi^{-1}) \cdot (J_f \circ \phi^{-1}) \cdot J_{\psi^{-1}}$. Now, using \eqref{eq:lipschitz_D} and \eqref{eq:bilipschitz_J}, we obtain for almost every $z \in \phi(U)$ the estimate
	\begin{align*}
		&\abs{D(\psi \circ f \circ \phi^{-1})(z)}^n\\
		&\qquad\leq \abs{D\psi(f(\phi^{-1}(z)))}^n 
				\abs{Df(\psi^{-1}(z))}^n \abs{D\psi^{-1}(z)}^n\\
		&\qquad\leq L^n \cdot K J_f(\psi^{-1}(z)) \cdot L^n
		\quad= KL^{2n} J_f(\psi^{-1}(z))\\
		&\qquad\leq KL^{2n} \bigl(L^{n} J_\psi(f(\psi^{-1}(z)))\bigr)
				\cdot J_f(\psi^{-1}(z)) \cdot \bigl(L^n J_{\phi^{-1}}(z)\bigr)\\
		&\qquad= K L^{4n} J_{\psi \circ f \circ \phi^{-1}}(z)
		\quad\leq K' J_{\psi \circ f \circ \phi^{-1}}(z).
	\end{align*}
	Hence, $\psi \circ f \circ \phi^{-1}$ is $K'$-quasiregular, which is what we wanted to show.
	
	\emph{\ref{def:QR_with_charts} $\implies$ \ref{def:QR_with_continuous_Sobolev_space}.} Suppose $f$ satisfies Definition \ref{def:QR_with_charts}. Then, given $K' > K$, $L > 1$, and $x \in M$, we find a neighborhood $U$ of $x$ and smooth bilipschitz charts $\phi \colon U \to \R^n$ and $\psi \colon V \to \R^n$ where $\psi \circ f \circ \phi^{-1}$ is $K'$-quasiregular. In particular, $\psi \circ f \circ \phi^{-1}$ is continuous and $\psi \circ f \circ \phi^{-1} \in W^{1,n}_\loc(\phi U, \psi V)$. Now by applying Lemma \ref{lem:sobolev_composition_mfld_targ_cont}, we have $f\vert_U = \psi^{-1} \circ (\psi \circ f \circ \phi^{-1}) \circ \phi \in W^{1,n}_\loc(U, N)$. We conclude that $f \in W^{1,n}_\loc(M, N)$.
	
	It remains to show that $\abs{Df}^n \leq K J_f$. However, for this, we note that by a similar estimate as in the previous case using the fact that the charts $\psi$ and $\phi$ are $L$-bilipschitz, we have $\abs{Df}^n \leq K' L^{4n} J_f$. Since we may select $K'$ and $L$ arbitrarily close to $K$ and $1$, respectively, we have the desired estimate.
\end{proof}

\subsection{Omitting the continuity assumption}

It is well known that in the real case, the assumption of continuity in the definition of quasiregularity is unnecessary. Lately, Goldstein, Haj\l asz and Pakzad studied the corresponding question on manifolds, giving the following result:

\begin{thm}[{\cite[Theorem 4]{Goldstein-Hajlasz-Pakzad}}]
	Let $M, N$ be smooth oriented Riemannian $n$-manifolds and assume that $N$ is compact. If $f \in W^{1,n}_\iota(M, N)$ for a Nash embedding $\iota$ and at almost every $x \in M$ either $Df(x) = 0$ or $J_f(x) > 0$, then $f$ is continuous.
	
	In particular, the continuity assumption in Definition \ref{def:QR_with_nash_embedding} (and by equivalence in Definition \ref{def:QR_with_colocal_diff}) is redundant for compact $N$.
\end{thm}

Whether the same holds for noncompact $N$ is however still unknown to our knowledge.

%%%%%%%%%%%%%%%%%%%%%%%%%%%%%%%%%%%%%%%%%%%%%%%%%%%%%%%%%%%%%%%%%%%%%%%%%%%%%%%%%%%%%%%%%%

\section{Topological properties of quasiregular maps}\label{sect:reshetnyaks_theorem}

In this section, we state and prove the manifold versions of the basic topological properties of quasiregular maps. In particular, the central result of this section is Reshetnyak's theorem: a non-constant quasiregular map is discrete, open and sense-preserving.

\subsection{Openness and discreteness}

Let $f \colon M \to N$ be a mapping between topological $n$-manifolds. Recall that $f$ is \emph{open} if $fU \subset N$ is open for every open $U \subset M$, and that $f$ is \emph{discrete} if $f^{-1}\{y\}$ has no accumulation points for any $y \in N$. We call $f$ a \emph{branched cover} if it is continuous, open and discrete. The first part of Reshetnyak's theorem states that a non-constant quasiregular map is open and discrete, and therefore a branched cover. The second part, that the map is sense-preserving, is discussed later, as we need to first introduce the topological orientation theory used to define the sense-preserving property.

The standard proof of the first part of Reshetnyak's theorem is through potential theory of the $\cA$-harmonic equation. Fortunately, however, the manifold version of the result can easily be deduced directly from the Euclidean version, so we don't need to discuss the details of the potential theoretic proof here.

\begin{thm}\label{thm:Reshetnyak_on_mflds}
	Let $M$ and $N$ be connected and oriented Riemannian $n$-manifolds, and let $f \colon M \to N$ be a non-constant quasiregular map. Then $f$ is discrete and open, and therefore a branched cover.
\end{thm}
\begin{proof}
	For openness, let $A$ be open, with the goal of showing that $fA$ is open. Let $x \in A$, and select positively oriented bilipschitz charts $\phi \colon U \to \R^n$ and $\psi \colon V \to \R^n$ for which $x \in U$ and $fU \subset V$. Then $\psi \circ f \circ \phi^{-1}$ is non-constant and quasiregular, and therefore open. Then, since $\psi$ and $\phi$ are topological embeddings with open image set, we see that 
	\[
		f(A \cap U) = \psi^{-1} \circ (\psi \circ f \circ \phi^{-1}) \circ \phi(A \cap U)
	\]
	is an open neighborhood of $f(x)$ in $fA$. We conclude that $fA$ is open.
	
	For discreteness, suppose to the contrary $f$ isn't discrete. Then there exists $y \in N$ for which $f^{-1}\{y\}$ has an accumulation point $x \in M$, and therefore  $f^{-1}\{y\}$ also contains a sequence $x_i \to x$, $x_i \neq x$. We again select positively oriented bilipschitz charts $\phi \colon U \to \R^n$ and $\psi \colon V \to \R^n$ for which $x \in U$ and $fU \subset V$. By openness of $U$, we may suppose $(x_n) \subset U$ by moving to a subsequence. Now $\phi(x_i) \to \phi(x)$ by continuity of $\phi$, but all these points are contained in $(\psi \circ f \circ \phi^{-1})^{-1}\{\psi(y)\}$, contradicting discreteness of $\psi \circ f \circ \phi^{-1}$. In conclusion, $f$ is discrete.
\end{proof}

\subsection{Normal domains}

The conclusions of Reshetnyak's theorem imply the existence of normal domains for quasiregular maps. The arguments for this are topological, and do not change significantly from the Euclidean case discussed in \cite[Section I.4]{Rickman_book}. Nevertheless, we provide the arguments for the relevant results.

\begin{defn}
	Let $M$ and $N$ be topological $n$-manifolds, and let $f \colon M \to N$ be a branched cover. A connected open set $U \subset M$ is a \emph{normal domain} of $f$ if $f \partial U = \partial f U$.
	
	Suppose $x \in M$. Then $U \subset M$ is a \emph{normal neighborhood} of $x$ if $U$ is a normal domain and $U \cap f^{-1}\{f(x)\} = \{x\}$.
	
	Furthermore, suppose $r > 0$, and that $M$ and $N$ are each equipped with a metric inducing the topology. Then we denote by $U_f(x, r)$ the $x$-component of $f^{-1} B_N(f(x), r)$.
\end{defn}

\begin{lemma}\label{lem:preimage_normal_neighborhoods}
	Let $M$ and $N$ be connected topological $n$-manifolds, each equipped with a metric inducing the topology, and let $f \colon M \to N$ be a branched cover. Then for every $x \in M$, there exists $r_x > 0$ for which $U_f(x, r)$ is a precompact normal neighborhood of $x$ whenever $r < r_x$.
\end{lemma}
\begin{proof}
	By discreteness of $f^{-1}\{f(x)\}$, we may select a neighborhood $U$ of $x$ for which $\overline{U}$ is compact, $M \setminus \overline{U}$ is nonempty, and $\overline{U} \cap f^{-1}\{f(x)\} = \{x\}$. Then $f \partial U$ is compact, non-empty, and doesn't contain $f(x)$. Therefore, we have $d_N(f(x), f \partial U) > 0$. We select $r_x = d_N(f(x), f \partial U)$.
	
	Suppose now that $r < r_x$. Then $U_f(x, r) \subset U$, since $U$ is open and $U_f(x, r)$ is a connected set which meets $U$ but doesn't meet $\partial U$. We conclude that $U_f(x,r) \cap f^{-1}\{f(x)\} = \{x\}$, and that $U_f(x,r)$ is precompact. It then remains to verify $f \partial U_f(x,r) = \partial f U_f(x,r)$.
	
	Let $y \in f \partial U_f(x,r)$, in which case there exists $z \in \partial U_f(x,r)$ for which $f(z) = y$. If $F$ is a neighborhood of $y$, there is a neighborhood $E$ of $z$ for which $fE \subset F$. Then $E$ meets $U_f(x,r)$, and therefore $F$ meets $f U_f(x,r)$.
	
	Moreover, note that $M \setminus f^{-1} B_N(f(x), r)$ is nonempty, since it contains $\partial U$. In addition, since $U_f(x,r)$ is a connected component of $f^{-1} B_N(f(x), r)$, $\partial U_f(x,r) \subset \partial f^{-1} B_N(f(x), r)$. Since $M$ is connected, we therefore have that $E$ meets $M \setminus f^{-1} B_N(f(x), r)$, and hence $F$ meets $N \setminus B_N(f(x),r) \subset N \setminus f U_f(x,r)$. We conclude that $f \partial U_f(x,r) \subset \partial f U_f(x,r)$.
	
	Suppose then that $y \in \partial f U_f(x,r)$. Select $y_i \in fU_f(x,r)$ for which $y_i \to y$, and $z_i \in U_f(x,r)$ for which $f(z_i) = y_i$. Since $U_f(x,r) \subset \overline{U}$ which is compact, we may by moving to a subsequence assume that $z_i \to z \in \overline{U}_f(x,r)$. Continuity of $f$ then yields $f(z) = y$.
	
	Let then $E$ be a neighborhood of $z$. By openness of $f$, $fE$ is a neighborhood of $y$, and therefore it meets $N \setminus f U_f(x,r)$. Hence, $E$ must meet $M \setminus U_f(x,r)$, from which it follows that $z \in \partial U_f(x,r)$. We conclude that $\partial f U_f(x,r) \subset f \partial U_f(x,r)$, which finishes the proof.
\end{proof}

Some particularly useful standard properties of $U_f(x,r)$ are the following.

\begin{lemma}\label{lem:preimage_normal_neighborhood_properties}
	Let $M$ and $N$ be connected topological $n$-manifolds, each equipped with a metric inducing the topology, and let $f \colon M \to N$ be a branched cover. Let $x \in M$ and $r < r_x$, where $r_x$ is given in Lemma \ref{lem:preimage_normal_neighborhoods}.
	\begin{enumerate}
		\item \label{enum:nn_prop_1} We have $fU_f(x,r) = B_N(f(x), r)$.
		\item \label{enum:nn_prop_2} We have $\diam U_f(x,r) \to 0$ as $r \to 0$.
		\item \label{enum:nn_prop_3} The restriction $f\vert_{U_f(x,r)} \colon U_f(x,r) \to B_N(f(x), r)$ is \emph{proper}; that is $(f\vert_{U_f(x,r)})^{-1} K$ is compact for all compact $K \subset B_N(f(x), r)$.
	\end{enumerate}
\end{lemma}
\begin{proof}
	For \eqref{enum:nn_prop_1}, suppose to the contrary that $B_N(f(x), r) \setminus f U_f(x,r) \neq \emptyset$. We claim that $B_N(f(x), r) \cap (f \partial U_f(x,r)) \neq \emptyset$. Namely, if this is not the case, we have by $f \partial U_f(x,r) = \partial f U_f(x,r)$ that $\{f U_f(x,r), B_N(f(x), r) \setminus  U_f(x,r)\}$ is a partition of $B_N(f(x), r)$ into nonempty open sets. This is only possible if $B_N(f(x), r)$ is disconnected. But then there would exist a sequence $y_j \in f U_f(x,r)$ tending to the boundary of $N$, which is impossible, as this sequence would be contained in the compact set $f \overline{U}_f(x,r)$.
	
	Hence, we have the claimed $B_N(f(x), r) \setminus (f \partial U_f(x,r)) \neq \emptyset$, and we may fix $z \in \partial U_f(x,r)$ for which $f(z) \in B_N(f(x), r)$. But now there exists a connected neighborhood $V$ of $z$ for which $V \subset f^{-1} B_N(f(x), r)$ and $V \cap U_f(x,r) \neq \emptyset$. Since $U_f(x,r)$ is a component of $f^{-1} B_N(f(x), r)$, we must have $z \in U$. This is a contradiction, as $U$ is open. Hence, \eqref{enum:nn_prop_1} follows.
	
	For \eqref{enum:nn_prop_2}, we merely note as in the proof of Lemma \ref{lem:preimage_normal_neighborhoods} that if $F$ is a precompact neighborhood of $x$, then selecting $r < d_N(f(x), f \partial F)$ yields $U_f(x,r) \subset F$. Therefore, $\diam U_f(x,r) < \diam F$ for small enough $r$, where we can make the right hand side arbitrarily small.
	
	For \eqref{enum:nn_prop_3}, suppose to the contrary that $f^{-1} K \cap U_f(x,r)$ isn't compact. Since $f^{-1} K$ is closed and $U_f(x,r)$ is precompact, there must therefore exist a sequence $(x_j) \in f^{-1} K \cap U_f(x,r)$ tending to some $x \partial U_f(x,r)$. By normality  of $U_f$ and condition \eqref{enum:nn_prop_1}, we have $f(x_j) \to f(x) \in \partial f U_f(x,r) = \partial B_N(f(x),r)$. But this is impossible since $f(x_j) \in K$ and $K$ is compactly contained $B_N(f(x),r)$.
\end{proof}

\subsection{Degree and index theory}

The degree and index theory of quasiregular maps is part of a connection between two distinct aspects of quasiregular maps: the change of variables formula, based on analysis, and the local winding properties of branched covers, based on topology.

Due to this connection, it is possible to approach the degree and index theory from two angles. One approach starts from the change of variables formula, using it as the basis for defining the concepts of degree theory. A detailed account on how this is done in the Euclidean setting can be found in the book of Hencl and Koskela \cite[Chapter 3]{Hencl-Koskela_book}. However, since we find the topological definitions of the degree and local index appear more common in writings on the topic, we take the opposite approach. Indeed, we will give the topological definitions of the degree and local index, and then connect them to the change of variables formula.

An unfortunate consequence of this approach, however, is the resulting prerequisite information in algebraic topology. Indeed, the topological definitions of the degree and local index are based on either homology or cohomology, with various equivalent variations available. In particular, the most common general definition of the degree uses compactly supported cohomology, while the local index is typically defined using homology.

In our exposition here, we will use singular homology and cohomology, as they are the most well known such theories. In order to prevent the section from growing too large, we must necessarily assume that the reader knows some of the basics of these theories: homology and cohomology of pairs, excision, the Mayer--Vietoris sequence, induced maps, and the cup and cap product. For anything beyond these basics, we provide references in the text. For the interested reader, we note an alternate treatment in the lecture notes of Pankka \cite{Pankka_Lectnotes_degreetheory}, which give a longer but far more self-contained exposition using Alexander--Spanier cohomology. 

\addtocontents{toc}{\SkipTocEntry}
\subsubsection{Local tolopogical orientation} We begin our discussion by giving the local topological definition of an orientation of a manifold. This will eventually lead to the definition of the local index. 

Suppose that $M$ is a topological $n$-manifold and $x \in M$. Then the relative homology group $H_n(M, M \setminus \{x\})$ is isomorphic to $\Z$. A \emph{local orientation at $x$} is then a selection of generator $c_x$ for the group $H_n(M, M \setminus \{x\})$. By excision, this induces a generator for $H_n(U, U \setminus \{x\}) \cong \Z$ for every neighborhood $U$ of $x$, which we also in an abuse of notation denote $c_x$.

We may then define a set $M_\Z$ as the disjoint union of fibers $H_n(M, M \setminus \{x\})$ over $x$. In case $M = \R^n$, we have a natural identification of this set with $\R^n \times H_n(\R^n, \R^n \setminus \{0\})$, with $c_x \in H_n(\R^n, \R^n \setminus \{x\})$ corresponding to $(x, (T_{-x})_* c_x)$ where $T_{-x}$ is the translation map by $-x$. Hence, we may give $\R^n_{\Z}$ the topology of $\R^n \times \Z$. 

Now, if $\phi \colon U \to \R^n$ is a chart on $M$, we obtain a map $\phi_* \colon U_\Z \to \R^n_\Z$ which is fibre-wise the isomorphism $\phi_* \colon H_n(U, U \setminus \{x\}) \to H_n(\R^n, \R^n \setminus \{\phi(x)\})$. Moreover, $U_\Z$ lies naturally within $M_\Z$ as a result of the aforementioned excision isomorphism $H_n(M, M \setminus \{x\}) \cong H_n(U, U \setminus \{x\})$. Hence, $M_\Z$ becomes a bundle over $M$ by giving it the topology induced by all such maps $\phi_*$, where $\phi$ is a chart. Now we may define that a map $x \mapsto c_x$ is an \emph{orientation} of $M$ if it is continuous as a map $M \to M_\Z$.

For smooth manifolds, this approach to orientation is equivalent with the smooth atlas approach. A proof for this can be found in \cite[Theorem VI.7.15]{Bredon_book_algtopo}. In particular, as is discussed in the proof, given an orientation $x \mapsto c_x$ of $M$ and a smooth oriented atlas $\{\phi_i\}$ on $M$, every $(\phi_i)_* c_x$ will correspond to the same generating element of $H_n(\R^n, \R^n \setminus \{0\})$. 

\addtocontents{toc}{\SkipTocEntry}
\subsubsection{Global topological orientation} We then give another, global definition of topological orientation, by considering \emph{compactly supported cohomology} $H^k_0(M) = H^k_0(M; \Z)$. The group $H^k_0(M)$ has multiple different definitions, which agree when $M$ is a topological manifold. For our approach, we use a direct limit definition to define $H^k_0(M)$ using singular cohomology. Suppose that $M$ is a topological $n$-manifold. Then $H^k_0(M)$ is the direct limit of groups
\[
	H^k_0(M) = \lim\limits_{\substack{\rightarrow \\ K \text{ compact}}} 
		H^k(M, M \setminus K),
\]
where on the right hand side the morphisms in the direct limit are pull-backs by inclusion maps $i^* \colon H^k(M, M \setminus K) \to H^k(M, M \setminus K')$ whenever $K \subset K'$. A more detailed discussion of direct limits and the above definition can be found e.g.\ in a book by Hatcher \cite[pp. 242--245]{Hatcher_book}.

By the following theorem, a selection of local orientations on $M$ yields a specific cohomology class $[M]$ in compactly supported cohomology. We call $[M]$ the \emph{orientation class} of $M$.

\begin{thm}\label{thm:local_to_global_orientation}
	Let $M$ be a connected topological $n$-manifold, and let $c \colon x \mapsto c_x$ be an orientation of $M$. Then $H^n_0(M) \cong \Z$, and moreover, the orientation $c$ corresponds to a specific generator $[M]$ of $H^n_0(M)$.
\end{thm}

We now sketch the proof of Theorem \ref{thm:local_to_global_orientation}, and at the same time explain how $[M]$ is obtained from $c$. We first recall the following result from \cite{Hatcher_book}; we briefly note that while only the first part of the statement is stated there, the second part immediately follows from the uniqueness part of the first statement.

\begin{lemma}[{\cite[Lemma 3.27]{Hatcher_book}}]\label{lem:hatcher_class_system}
	Let $M$ be a topological $n$-manifold, and let $a \colon x \mapsto a_x$ be a continuous section $M \to M_\Z$. Then for every compact $K \subset M$, there exists a unique $a_K \in H_n(M, M \setminus K)$ which satisfies the following: if $x \in K$, and $i_{x, K}$ is the inclusion of pairs $i_{x, K} \colon (M, M \setminus K) \hookrightarrow (M, M \setminus \{x\})$, then
	\begin{equation}\label{eq:class_system_incl_points}
		(i_{x, K})_* a_K = a_x.
	\end{equation}
	
	Moreover, if $K, K' \subset M$ are compact, $K \subset K'$, and $i_{K, K'}$ is the inclusion of pairs $i_{K, K'} \colon (M, M \setminus K') \hookrightarrow (M, M \setminus K)$, then 
	\begin{equation}\label{eq:class_system_incl_compacts}
		(i_{K, K'})_* a_{K'} = a_K. 
	\end{equation}
\end{lemma}

Now, let $\beta \in H^k_0(M)$, and let $a \colon M \to M_\Z$ be a continuous section. We may then define a cap product $a \cap \beta \in H_{n-k}(M)$. Namely, for every compact $K \subset M$, we have by Lemma \ref{lem:hatcher_class_system} an element $a_K \in H_n(M, M \setminus K)$, and the relative version of the cap product yields a map $(a_K \cap \cdot) \colon H^k(M, M \setminus K) \to H_{n-k}(M)$. Using the naturality properties of the cap product along with \eqref{eq:class_system_incl_compacts}, we obtain a map $(a \cap \cdot) \colon H^k_0(M) \to H_{n-k}(M)$ from the direct limit. A more detailed discussion of this can be found in \cite[pp.\ 239--245]{Hatcher_book}.

This cap product map can then be used to state the celebrated \emph{Poincar\'e duality}, which we recall in the following theorem.

\begin{thm}[{\cite[Theorem 3.35]{Hatcher_book}}]\label{thm:poincare_duality}
	Let $M$ be a topological manifold, and let $c \colon M \to M_\Z$ be an orientation. Then the map $(c \cap \cdot) \colon H^k_0(M) \to H_{n-k}(M)$ is an isomorphism for every $k$.
\end{thm}

Now, Theorem \ref{thm:local_to_global_orientation} immediately follows from Poincar\'e duality.

\begin{proof}[Proof of Theorem \ref{thm:local_to_global_orientation}]
 	Since $M$ is connected, we have $H_0(M) \cong \Z$. Therefore, by the Poincar\'e duality given in Theorem \ref{thm:poincare_duality}, $H_0^n(M) \cong \Z$. Moreover we have a standard positive generator $1 \in H_0(M)$, which is the homology class containing every singleton $0$-simplex in $M$ with coefficient 1. Therefore, we may select $[M]$ to be the element mapped to $1$ in the Poincar\'e duality isomorphism: $[M]$ is the unique element of $H^n_0(M)$ for which $c \cap [M] = 1$.
\end{proof}

\addtocontents{toc}{\SkipTocEntry}
\subsubsection{The sense-preserving -part of Reshetnyak's theorem} Having finally discussed the prerequisite background in algebraic topology, we may begin the discussion of applying this to quasiregular maps. Suppose that $f \colon M \to N$ is a branched cover between oriented topological $n$-manifolds. Then for every $x \in M$, we may select a neighborhood $U$ for which $f^{-1}\{f(x)\} \cap U = \{x\}$. We say $f$ is \emph{sense preserving} if, for all such $x$ and $U$, the induced map in homology $f_* \colon H_n(U, U \setminus \{x\}) \to  H_n(fU, fU \setminus \{f(x)\})$ takes $c_{x}$ to a strictly positive multiple of $c_{f(x)}$.

Since the isomorphisms $H_n(U, U \setminus \{x\}) \to H_n(U', U' \setminus \{x\})$ given by excision for $x \in U \subset U'$ are natural, we obtain that it is sufficient to check the sense preserving property in only one suitable neighborhood $U$ of every point $x$. In particular, we may define the \emph{local index} $i(f, x)$ of $f$ at $x$ to be the unique integer for which $f_* c_x = i(f, x) c_{f(x)}$; this is independent on the neighborhood $U$ as long as $f^{-1}\{f(x)\} \cap U = \{x\}$. Hence, the sense preserving property can also be characterized in the manner that a branched cover $f \colon M \to N$ is sense preserving if and only if $i(f, x) > 0$ for every $x \in M$.

We now have the requisite terminology to state the last component of Reshetnyak's theorem.

\begin{prop}\label{prop:Reshetnyak_sense_preserving_on_manifolds}
	Let $M$ and $N$ be connected, oriented Riemannian $n$-manifolds, and let $f \colon M \to N$ be a non-constant quasiregular map. Then $f$ is sense-preserving.
\end{prop}
\begin{proof}
	We again deduce the manifold version from the Euclidean version, which we assume to be known; see e.g.\ \cite[Theorem I.4.5]{Rickman_book}.
	
	As in the other parts of Reshetnyak's theorem \ref{thm:Reshetnyak_on_mflds}, we again fix $x \in M$ and take chart neighborhoods $U, V$ and positively oriented bilipschitz charts $\phi, \psi$; this time we make the extra assumption that $f^{-1}\{f(x)\} \cap U = \{x\}$, which we may assume due to discreteness of $f$. Consequently, $\psi \circ f \circ \phi^{-1}$ is quasiregular, and therefore maps $c_{\phi(x)} \in H_n(\phi U, \phi U \setminus \{\phi(x)\})$ to a positive multiple of $c_{\psi(f(x))} \in H_n(\psi (f U), \psi (f U) \setminus \{\psi(f(x))\})$. Since $\phi_*$ and $\psi_*$ are isomorphisms that preserve local orientation classes, the claim follows.
\end{proof}

\addtocontents{toc}{\SkipTocEntry}
\subsubsection{Degree of a map} Now, suppose that $M$ and $N$ are connected, oriented $n$-manifolds, and that $f \colon M \to N$ is a proper continuous map. Then for nonnegative integers $k$ we have pull-back maps $f^* \colon H^k(N, N \setminus K) \to H^k(M, M \setminus f^{-1} K)$ for all compact $K \subset N$ which commute with the inclusions in the direct limit definition of compactly supported cohomology. We thus obtain a pull-back map $f^* \colon H^k_0(N) \to H^k_0(M)$. The \emph{degree} of $f$ is the unique integer $(\deg f)$ for which $f^* [N] = (\deg f) [M]$, where $[M]$ and $[N]$ are the corresponding orientation classes.

The degree and local indices of a proper branched cover are closely tied together by the following result.

\begin{thm}\label{thm:deg_i_sum_result}
	Suppose that $M$ and $N$ are connected, oriented Riemannian $n$-manifolds, and let $f \colon M \to N$ be a proper branched cover. Then for every $y \in N$, we have
	\[
		\deg f = \sum_{x \in f^{-1}\{y\}} i(f, x).
	\]
\end{thm}

We sketch an essentially complete proof of this using the definitions given here. However, since the proof is long and more in the realm of general algebraic topology, and moreover since the topic already has one treatise in \cite{Pankka_Lectnotes_degreetheory}, we will leave out some details regarding e.g.\ commutation of some diagrams and induced maps in the direct limit.

We begin by noting that, under the assumptions of Theorem $\ref{thm:deg_i_sum_result}$, the set $f^{-1}\{y\}$ is always non-empty and finite. Indeed, since $f$ is a proper discrete map, $f^{-1}\{y\}$ is a discrete compact set, and therefore finite. Moreover, since $f$ is a proper continuous map between locally compact spaces, $f$ is closed. Since $f$ also is open, $fM$ is a non-empty closed-open subset of $N$. By connectedness of $N$, we get that $fM = N$, that is, $f$ is surjective. Hence, the set $f^{-1}\{y\}$ is non-empty.

Next, we proceed to show that $y \mapsto \sum_{x \in f^{-1}\{y\}} i(f, x)$ is constant. This is done in two lemmas.

\begin{lemma}\label{lem:index_formula_points}
	Suppose that $M$ and $N$ are connected, oriented Riemannian $n$-manifolds, with corresponding orientations $c$ and $c'$. Let $f \colon M \to N$ be a proper branched cover. Then for every $y \in N$, we have
	\[
		f_* c_{f^{-1} \{y\}} = \left( \sum_{x \in f^{-1}\{y\}} i(f, x)\right) c_{y}'.
	\]
\end{lemma}
\begin{proof}
	Let $B = B_N(y, r)$, and for every $x \in f^{-1}\{y\}$, let $U_x = U_f(x,r)$, where $r$ is such that $U_x$ are normal neighborhoods. Consider then the following diagram.
	\[\begin{tikzcd}[sep=small]
		H_n(M, M \setminus f^{-1}\{y\})
			\ar[rr, "f_*"]
			\ar[dr, "\cong"] \ar[d, "\cong"]
		&& H_n(N, N \setminus \{y\})
			\ar[d, "\cong"] 
		\\ \bigoplus\limits_{f(x) = y} H_n(M, M \setminus \{x\})
			\ar[d, "\cong"]
		&H_n(f^{-1} B, f^{-1} B\setminus f^{-1}\{y\})
			\ar[dl, "\cong"] \ar[r, "f_*"]
		&H_n(B, B \setminus \{y\})
			\ar[d, "="]
		\\ \bigoplus\limits_{f(x) = y} H_n(U_x, U_x \setminus \{x\})
			\ar[rr, "\oplus f_*"]
		&&H_n(B, B \setminus \{y\})
	\end{tikzcd}\]
	The vertical maps in the diagram are given by either excision or by repeated use of Mayer-Vietoris. By the naturality properties of the vertical maps, the diagram commutes (details left to the reader). 
	
	The isomorphism $H^n(M, M \setminus f^{-1}\{y\}) \to \bigoplus_{f(x) = y} H^n(M, M \setminus \{x\})$ is a sum of pushforwards by inclusion maps. Hence, the image of $c_{f^{-1}\{y\}} \in H^n(M, M \setminus f^{-1}\{y\})$ under the left column of the diagram is $\sum_{x \in f^{-1}\{y\}} c_x$. The bottom map of the diagram then maps $\sum_{x \in f^{-1}\{y\}} c_x$ to the element $\sum_{x \in f^{-1}\{y\}} i(f, x)c_{y}'$. The claim hence follows.
\end{proof}
\begin{lemma}\label{lem:index_formula_constant}
	Suppose that $M$ and $N$ are connected, oriented Riemannian $n$-manifolds, with corresponding orientations $c$ and $c'$. Let $f \colon M \to N$ be a proper branched cover. Then the map $N \to \Z$ given for $y \in N$ by
	\[
		y \mapsto \sum_{x \in f^{-1}\{y\}} i(f, x)
	\]
	is constant.
\end{lemma}
\begin{proof}
	Let $K \subset N$ be such that $K$ is compact, and $K$ is contained in the domain of a chart $\phi \colon U \to \R^n$ such that $B = \phi(K)$ and $B' = \phi(U)$ are a closed and an open Euclidean ball, respectively. It follows that for every $y \in K$, we have $H_n(N, N \setminus K) \cong H_n(N, N \setminus \{y\})$ by the inclusion map. Indeed, this can be deduced from the diagram
	\[\begin{tikzcd}[sep=small]
		H_n(N, N \setminus K)
			\ar[r] \ar[d, "\cong"]
		& H_n(N, N \setminus \{y'\})
			\ar[d, "\cong"] 
		\\ H_n(U, U \setminus K)
			\ar[d, "\phi_*"] \ar[r]
		&H_n(U, U \setminus \{y'\})
			\ar[d, "\phi_*"]
		\\ H_n(B', B' \setminus B)
			\ar[r]
		&H_n(B', B' \setminus \{\phi(y')\}),
	\end{tikzcd}\]
	where the horizontal maps are induced by inclusions of pairs, and the bottom horizontal map is easily seen to be isomorphic.
	
	Now, given a set $K$ as above and $y \in K$, since $H_n(N, N \setminus K)$ is isomorphic to $\Z$, we have $f_* c_{f^{-1}K} = l c'_K$ for some $l \in \Z$. By mapping both sides of this equality into $H_n(N, N \setminus \{y\})$ by an inclusion of pairs, we obtain $f_* c_{f^{-1} \{y\}} = l c'_y$. But by Lemma \ref{lem:index_formula_points}, we have $f_* c_{f^{-1} \{y\}} = \big(\sum_{x \in f^{-1}\{y\}} i(f, x)\big)c_y'$. Hence, $l = \sum_{x \in f^{-1}\{y\}} i(f, x)$.
	
	However, $l$ is independent on the selection of $y \in K$. Hence, $y \mapsto \sum_{x \in f^{-1}\{y\}} i(f, x)$ is constant in $K$, and therefore locally constant. Since $N$ is connected, the claim follows.
\end{proof}

Next, we consider a cap product map $H^k_0(N) \to H_{n-k}(N)$ using the map $f$ and the orientation of $M$.

\begin{lemma}\label{lem:cap_product_formula}
	Suppose that $M$ and $N$ are connected, oriented Riemannian $n$-manifolds, and let $c$ be the orientation of $M$. Let $f \colon M \to N$ be a proper branched cover. Then we have a cap product map
	\[
		((f_* c) \cap \cdot) \colon H^k_0(N) \to H_{n-k}(N)
	\]
	given as a direct limit of the cap product maps
	\[
		((f_* c_{f^{-1} K}) \cap \cdot) \colon H^k(N, N \setminus K) \to H_{n-k}(N).
	\]
	Moreover, this cap product map satisfies
	\begin{equation}\label{eq:cap_push_pull_formula}
		f_*(c \cap f^* \alpha) = (f_* c) \cap \alpha
	\end{equation}
	for every $\alpha \in H^k_0(N)$.
\end{lemma}
\begin{proof}
	There are essentially two things to verify. First, to obtain the map $(f_* c) \cap \cdot$, it has to be verified that that the maps $(f_* c_{f^{-1} K}) \cap \cdot$ commute with the maps induced by inclusions of pairs $(N, N \setminus K) \hookrightarrow (N, N \setminus K')$. Second, the formula \eqref{eq:cap_push_pull_formula} is obtained by showing that a similar formula of the form $f_*(c_{f^{-1}K} \cap f^* \alpha_K) = (f_* c_{f^{-1} K}) \cap \alpha_K$ holds when $K \subset N$ is compact. Both of these follow from the standard naturality properties of the cap product, see e.g.\ \cite[p.\ 240--241]{Hatcher_book}. We leave the details of the proof to the reader.
\end{proof}

Now, we finally may complete the proof of Theorem \ref{thm:deg_i_sum_result}.

\begin{proof}[Proof of Theorem \ref{thm:deg_i_sum_result}]
	As usual, we denote by $c$ and $c'$ the local orientations of $M$ and $N$, respectively. By Lemma \ref{lem:index_formula_constant}, the right hand side of the desired equality assumes a constant value independent of $y$, which we will denote $\deg' f$. It remains to therefore show that $\deg' f = \deg f$.
	
	Now, we consider the cap product formula \eqref{eq:cap_push_pull_formula} of Lemma \ref{lem:cap_product_formula} for $\alpha = [N]$. The left hand side of the formula becomes
	\[
		f_*(c \cap f^* [N]) = f_*(c \cap (\deg f)[M]) = f_* ((\deg f) 1_M) = (\deg f) 1_N,
	\]
	where $1_M$ and $1_N$ are the positive generators of $H_0(M)$ and $H_0(N)$, respectively.
	
	On the other hand, let $K \subset N$ be compact. We claim that $f_* c_{f^{-1} K} = (\deg' f) c_K$. Indeed, for every $y \in K$, we see that $f_* c_{f^{-1} K}$ maps to $f_* c_{f^{-1}\{y\}} = (\deg f') c'_y$ in the map induced by the corresponding inclusion of pairs. However, there is a unique element in $H_n(M, M \setminus K)$ with this property: namely, $(\deg' f) c_K$. Hence, we conclude that $f_* c_{f^{-1} K} = (\deg' f) c_K$.
	
	Now, it follows in the direct limit that $((f_* c) \cap \cdot) = (\deg' f) (c \cap \cdot)$. Hence, the right hand side of \eqref{eq:cap_push_pull_formula} for $\alpha = [N]$ becomes
	\[
		(f_* c) \cap [N] = (\deg' f) (c \cap [N]) = (\deg' f) 1_N.
	\] 
	Since the left hand side was seen to equal $(\deg f) 1_N$, we conclude that $\deg f = \deg' f$, completing the proof.	
\end{proof}

%%%%%%%%%%%%%%%%%%%%%%%%%%%%%%%%%%%%%%%%%%%%%%%%%%%%%%%%%%%%%%%%%%%%%%%%%%%%%%%%%%%%%%%%%%%%
\section{The quasiregular change of variables formula}\label{sect:change_of_variables}

\subsection{The Lusin property}

The absolute first consideration in obtaining a quasiregular change of variables formula is showing that quasiregular maps satisfy the Lusin properties discussed in Section \ref{sect:meas_diff_geo_prelims}.

\begin{thm}\label{thm:QR_lusin_N_conditions}
	Let $M$ and $N$ be connected and oriented Riemannian $n$-manifolds, and let $f \colon M \to N$ be a quasiregular map. Then $f$ satisfies the Lusin $(N)$ -condition. Moreover, if $f$ is not constant, then it also satisfies the Lusin $(N^{-1})$-condition. 
	
	Consequently, if $f$ is quasiregular, then $f(E)$ is measurable for every measurable $E \subset M$, and $f^{-1} F$ is measurable for every measurable $F \subset N$.
\end{thm}
\begin{proof}
	We again reduce to the Euclidean case via charts. For the Lusin $(N)$ -condition, see e.g.\ \cite[Proposition I.4.14]{Rickman_book} for the proof in the real case. Similarly, for the $(N^{-1})$-condition, see e.g.\ \cite[Theorem 8.1]{Bojarski-Iwaniec}.
	
	Suppose that then that $m_n(E) = 0$. For every $x \in M$, select again orientation preserving bilipschitz charts $\phi \colon U \to \R^n$ and $\psi \colon V \to \R^n$ for which $x \in U$ and $fU \subset V$. Since $\phi$ and $\psi$ are bilipschitz, they preserve sets of measure zero. Hence, using the quasiregularity of $\psi \circ f \circ \phi^{-1}$ and the Euclidean Lusin property of quasiregular maps, we obtain that $f(E \cap U)$ has measure zero. Since $M$ is second-countable, we have a countable cover of $E$ with such charts, using which we conclude that $m_n(f(E)) = 0$. This completes the proof of the Lusin $(N)$ -condition.
	
	The generalization of the $(N^{-1})$ -condition to manifolds is virtually identical to the previously shown generalization of the $(N)$ -condition, with the addition of using the openness of $f$ provided by Reshetnyak's theorem to conclude that the restriction of $f$ to any chart isn't constant. The consequent preservation of measurable sets is then due to Lemma \ref{lem:lusin_preserves_measurable}, as quasiregular maps are continuous. Note that in the special case where $f$ is constant, the measurability of $f^{-1} F$ for measurable sets $F \subset N$ is trivial, as the pre-image of any subset of $N$ under a constant function is either $M$ or the empty set.
\end{proof}

\subsection{Multiplicity and branching}

Next, we discuss multiplicity, which is a key component of the change of variables formula. We go over the basic definition, and how it relates to the previous discussion on degree and index theory.

Let $X$ and $Y$ be sets, and let $f \colon X \to Y$ be a function. For every $y \in Y$, and every $A \subset X$, we define the \emph{multiplicity $N(f, y, A)$ of $f$ in $A$ at $y$} by
\[
N(f, y, A) = \# (f^{-1}\{y\} \cap A).
\]
That is, $N(f, y, A)$ is the number of points in $A$ which $f$ maps to $y$. The value of $N(f, y, A)$ is therefore either a non-negative integer, or $+\infty$.

We also define the following variations with less parameters: if $y \in Y$ and $A \subset X$, then
\begin{align*}
N(f, y) &= N(f, X, y),\\
N(f, A) &= \sup_{y \in Y} N(f, y, A), \qquad \text{and}\\
N(f) &= N(f, X) = \sup_{y \in Y} N(f, y).
\end{align*}
We note that the letter $N$ may end up being used both for multiplicity and to denote a manifold. Unfortunately the standard notation in the field results in this slight overloading of the letter $N$, but thankfully there should be no risk of confusion, since the multiplicity version is always followed by parameters while the manifold version never is.

For quasiregular maps, multiplicity has a close connection with index theory, and a key concept illuminating this connection is the \emph{branch set} of a quasiregular map. Namely, suppose that $f \colon X \to Y$ is a branched cover. Then the branch set $\cB_f$ of $f$ is the set of points $x \in X$ at which $f$ is not locally a homeomorphism. We note that $\cB_f$ is clearly a closed subset of $X$; if $x \notin \cB_f$, then it has a neighborhood $U$ for which $f\vert_U \colon U \to fU$ is a homeomorphism, and consequently no point in $U$ belongs to $\cB_f$.

A crucial property of the branch set is that it is small. For the integration theory of quasiregular maps, it is usually enough to know that it is of measure zero; we again quickly show the manifold version of this result from the Euclidean version.

\begin{lemma}\label{lem:branch_set_measure_zero}
	Suppose that $M$ and $N$ are connected, oriented Riemannian $n$-manifolds, and let $f \colon M \to N$ be a non-constant quasiregular map. Then $\cB_f$ is a set of measure zero in the Riemannian volume. Consequently, $f \cB_f$ and $f^{-1} f \cB_f$ also have measure zero.
\end{lemma}
\begin{proof}
	As usual, let $\phi \colon U \to \R^n$ and $\psi \colon V \to \R^n$ be positively oriented bilipschitz charts such that $fU \subset V$. Then $\psi \circ f \circ \phi^{-1}$ is non-constant and quasiregular. Since $\phi$ and $\psi$ are embeddings, it follows that $\cB_{\psi \circ f \circ \phi^{-1}} = \phi^{-1} \cB_f$, and therefore $\phi^{-1} \cB_f$ is a set of Lebesgue measure zero by the Euclidean result. Since $\phi$ is bilipschitz, it preserves sets of measure zero, and it follows that $\cB_f \cap U$ has Riemannian volume zero. Now, since we may cover $M$ with countably many $U$ as above, $\cB_f$ has measure zero.
	
	The claim for $f \cB_f$ and $f^{-1} f \cB_f$ follows from the Lusin properties discussed in Theorem \ref{thm:QR_lusin_N_conditions}.
\end{proof}

We discuss here briefly the background of the Euclidean result. In particular, it is consequence of a stronger result known as the Chernavskii-V\"ais\"al\"a theorem: the branch set of a branched cover has topological dimension at most $n-2$. This result is mentioned briefly in Rickman's book \cite[Theorem 4.6]{Rickman_book}; 
however, a far more in-depth look into it can be found in the previously mentioned notes of Pankka \cite{Pankka_Lectnotes_degreetheory}. For this reason the reader is referred there for details, and we contend ourselves with the above statement which is enough for the integration theory.

Now the concepts of degree theory, multiplicity, and the branch set are linked in the following way for quasiregular maps.

\begin{lemma}\label{lem:multiplicity_degree_branching}
	Suppose that $M$ and $N$ are connected, oriented Riemannian $n$-manifolds, and let $f \colon M \to N$ be a non-constant quasiregular map. 
	
	\begin{itemize}
		\item If $x \in M$ and $U$ is a normal neighborhood of $x$, then
		\[
		i(f, x) = N(f, U).
		\]
		\item For every $x \in M$, we have $x \in \cB_f$ if and only if $i(f, x) > 1$.
		\item Suppose that $f$ is proper. Then $N(f) = \deg f$. In particular, we have $N(f, y) \neq \deg f$ if and only if $y \in f \cB_f$, and therefore $N(f, y) = \deg f$ almost everywhere in $N$.
	\end{itemize}
\end{lemma}
\begin{proof}
	Suppose first that $x \in M$ and $f$ is a local homeomorphism at $x$. Consequently by Lemma \ref{lem:preimage_normal_neighborhoods} and Proposition \ref{prop:Reshetnyak_sense_preserving_on_manifolds}, we find a connected normal neighborhood $U$ of $x$ for which $f\vert_U \colon U \to fU$ is a sense-preserving homeomorphism. Since homeomorphisms induce isomorphisms in compactly supported cohomology, it follows that $\deg f\vert_U = 1$. Hence, $i(f, x) = 1$ by Theorem \ref{thm:deg_i_sum_result}. This shows half of the second claim.
	
	We now show the first claim. Suppose $x \in M$ and $U$ is a normal neighborhood of $x$. Since $f$ is sense-preserving, $i(f, x') \geq 1$ for every $x' \in M$. Hence, for every $y \in N$, we obtain by Theorem \ref{thm:deg_i_sum_result} that
	\[
		N(f, y, U) = \sum_{x' \in f^{-1}\{y\} \cap U} 1
		\leq \sum_{x' \in f^{-1}\{y\} \cap U} i(f, x')
		= \deg f\vert_U
		= i(f, x).
	\]
	Hence, $N(f, U) \leq i(f, x)$.
	
	However, since $f^{-1} f \cB_f$ has measure zero by Lemma \ref{lem:branch_set_measure_zero}, we find a $z \in U$ for which $z \notin f^{-1} f \cB_f$. Hence, for every $x' \in f^{-1} \{f(z)\}$, we have that $f$ is a local homeomorphism at $x'$. By the part of the second claim of the lemma we already proved, it follows that $i(f, x') = 1$ for every $x' \in f^{-1} \{f(z)\}$. Hence, we have
	\[
		N(f, f(z), U) = \sum_{x' \in f^{-1}\{y\} \cap U} 1
		= \sum_{x' \in f^{-1}\{y\} \cap U} i(f, x')
		= \deg f\vert_U
		= i(f, x).
	\]
	We conclude that $N(f, U) = i(x, f)$.
	
	We now finish the proof of the second claim of the lemma. Suppose therefore that $f$ is not a local homeomorphism at $x$. By Lemma \ref{lem:preimage_normal_neighborhoods}, we find a normal neighborhood $U$ of $x$. If $f\vert_U$ is injective, then $f\vert_U \colon U \to fU$ is a bijective, continuous, and open map, and therefore a homeomorphism. This is a contradiction with $f$ not being a local homeomorphism at $x$. Therefore there exist $x_1, x_2 \in U$ with $f(x_1) = f(x_2)$. Using the first claim of the lemma, the proof of which we previously completed, we conclude that $i(f, x) = N(f, U) \geq N(f, f(x_1), U) \geq 2$, which proves that $i(f, x) > 1$.
	
	For the third claim, if $f \colon M \to N$ is proper non-constant quasiregular, we note that for every $y \in N$ we have $\deg f = \sum_{x \in f^{-1}\{y\}} i(x, f)$ by Theorem \ref{thm:deg_i_sum_result}. By the previously shown part of the lemma, this sum equals $N(f, y)$ if and only if $f^{-1}\{y\} \cap \cB_f = \emptyset$, and otherwise it is larger than $N(f, y)$. The third claim immediately follows, with the almost everywhere part due to Lemma \ref{lem:branch_set_measure_zero}.
\end{proof}

Finally, we record a tiny technical detail that will regardless be important for our following discussion on the change of variables formula to make sense.

\begin{lemma}\label{lem:N_measurability}
	Suppose that $M$ and $N$ are connected, oriented Riemannian $n$-manifolds, and let $f \colon M \to N$ be a non-constant quasiregular map. Let $U \subset M$ be open. Then the function $N(f, \cdot, U) \colon N \to [0, \infty]$ mapping $y \in N$ to $N(f, y, U)$ is measurable.
\end{lemma}
\begin{proof}
	We define another function $\mu(f, \cdot U) \colon N \to [0, \infty]$ by
	\[
		\mu(f, y, U) = \sum_{x \in f^{-1}\{y\} \cap U} i(f, x)
	\]
	for $y \in N$. By the second item of Lemma \ref{lem:multiplicity_degree_branching}, we have $\mu(f, y, U) = N(f, y, U)$ for almost every $y \in N$. Hence, it suffices to show that $\mu(f, \cdot U)$ is measurable.
	
	The argument for that presented here is essentially from \cite[14.30]{Heinonen-Kilpelainen-Martio_book}. Let $(U_j)_{j=1}^\infty$ be an increasing sequence of open sets compactly contained in $U$, for which $\bigcup_j U_j = U$. For every $j$, let $\psi_j \colon M \to [0, 1]$ be a continuous function for which $\psi_j \vert_{U_j} \equiv 1$ and $\spt \psi_j \subset U_{j+1}$. And then for every $j$, define
	\[
		\mu_j(f, y, U) = \sum_{x \in f^{-1}\{y\} \cap U} i(f, x) \psi_j(x).
	\]
	for $y \in N$.
	
	Since the supports of $\psi_j$ are compact and $f$ is discrete, the values of every function $\mu_j(f, \cdot, U)$ are finite. Moreover, $\mu_{j+1}(f, \cdot, U) \geq \mu_j(f, \cdot, U)$ everywhere. Moreover, $\mu_j(f, \cdot, U) \to \mu(f, \cdot, U)$ pointwise as $j \to \infty$. We show that the functions $\mu_j(f, \cdot, U)$ are continuous. Then it follows that $\mu(f, \cdot, U)$ is measurable as a monotone limit of measurable functions, and the claim follows.
	
	For the continuity of $\mu_j(f, \cdot, U)$, fix $y \in N$. Then $f^{-1}\{y\} \cap \spt \psi_j$ contains at most finitely many points $\{x_1, \dots, x_k\}$, since $\psi_j$ is compactly supported and the set $f^{-1}\{y\}$ is discrete and closed. Fix $\eps > 0$, and let $A = \sum_{l=1}^k i(f, x_l)$. By Lemma \ref{lem:preimage_normal_neighborhood_properties}, we may select $r > 0$ for which ,for every $l \in \{1, \dots, k\}$, the set $U_f(x_l, r)$ is a normal neighborhood of $x_l$ and $\psi(U_f(x_l, r)) \subset B_\R(\psi(x_l), \eps/A)$. It follows from Theorem \ref{thm:deg_i_sum_result} and a standard application of triangle inequality that $\abs{\mu_j(f, y, U) - \mu_j(f, y', U)} < \eps$ for every $y' \in B_N(y, r)$. The continuity of $\mu_j(f, \cdot, U)$ follows.
\end{proof}

\begin{rem}
	The proof of the above lemma reveals that in some sense, $\mu(f, \cdot, U)$ is in fact the more natural multiplicity function for a quasiregular $f$. However, we find the use of $N(f, \cdot, U)$ more common in the measurable context as it is far easier to define.
\end{rem}

\subsection{Change of variables}

Finally, we arrive at the main goal of this section, the change of variables formula for quasiregular maps. The Euclidean case is discussed eg.\ in Rickman \cite[Proposition 4.1.14]{Rickman_book}. We proceed to first derive a version for non-negative measurable functions, where we take advantage of both the Euclidean result in \cite{Rickman_book}, as well as the methods described in the proof of the Euclidean result in \cite{Rickman_book}.

\begin{thm}\label{thm:change_of_variables_non-neg}
	Let $M$ and $N$ be connected, oriented Riemannian $n$-manifolds, and let $f \colon M \to N$ be a non-constant quasiregular map. Suppose that $g \colon N \to [0, \infty]$ is a non-negative measurable function, and $U \subset M$ is a measurable set. Then $g \circ f \colon M \to [0, \infty]$ is a measurable function, and 
	\[
		\int_U (g \circ f) J_f \vol_M = \int_N N(f, \cdot, U)\, g \vol_N.
	\]
\end{thm}
\begin{proof}
	We note that the above change of variables formula can be written as
	\[
		\int_U f^*(g\vol_M) = \int_N N(f, \cdot, U) \, g\vol_N.
	\]
	The measurability of $g \circ f$ follows by Theorem \ref{thm:QR_lusin_N_conditions}. Indeed, if $A \subset [0, \infty]$ is open, then $g^{-1} A$ is measurable, and consequently $(g \circ f)^{-1} A = f^{-1}(g^{-1} A)$ is measurable. The measurability of $N(f, \cdot, U)$ was shown in Lemma \ref{lem:N_measurability}.
	
	We deal with the transformation formula in two cases. First, suppose that $U$ and $fU$ are both contained in positively oriented bilipschitz charts $\phi \colon U \to \R^n$ and $\psi \colon fU \to \R^n$, respectively. Hence, we may use the definition of integration of $n$-forms to obtain
	\[
	\int_{N} N(f, \cdot, U) \, g \vol_N.
	= \int_{\psi f U} (\psi^{-1})^*(N(f, \cdot, U)\, g \vol_N)
	\]
	We may further write this as
	\begin{multline*}
	\int_{\psi f U} (\psi^{-1})^*(N(f, \cdot, U)\, g \vol_N)
	= \int_{\psi f U} N(f, \psi^{-1}(\cdot), U) (\psi^{-1})^* (g \vol_N)\\
	= \int_{\psi f U} N(\psi \circ f \circ \phi^{-1}, \cdot, \phi U) 
	(\psi^{-1})^* (g \vol_N).
	\end{multline*}
	Note that $(\psi^{-1})^* (g \vol_N) = (g \circ \psi^{-1}) J_{\psi^{-1}} \vol_{\R^n}$, and since $\psi$ is positively oriented, the function $(g \circ \psi^{-1}) J_{\psi^{-1}}$ is positive. Hence, we may use the Euclidean change of variables formula \cite[Proposition 4.1.14]{Rickman_book} for $\psi \circ f \circ \phi^{-1}$, which yields us
	\begin{multline*}
	\int_{\psi f U} N(\psi \circ f \circ \phi^{-1}, \cdot, \phi U) (\psi^{-1})^* (g \vol_N)
	= \int_{\phi U}  (\psi \circ f \circ \phi^{-1})^* (\psi^{-1})^* (g \vol_N)\\
	= \int_{\phi U} (\phi^{-1})^* f^* (\psi^* (\psi^{-1})^* (g \vol_N))
	= \int_{\phi U} (\phi^{-1})^* f^* (g \vol_N)
	\end{multline*}
	Finally, we again refer to the definition of integration of $n$-forms, and conclude that
	\[
		\int_U f^*(g \vol_N) = \int_{\phi U} (\phi^{-1})^* f^* (g \vol_N) 
		= \int_{N} N(f, \cdot, U) \, g \vol_N.
	\]
	
	Next, we consider the general case. Recall that $\cB_f$ is a closed subset of $M$. Using the continuity of $f$, we select for every $x \in M \setminus \cB_f$ a connected neighborhood $U_x$ such that $U_x$ and $fU_x$ are contained in a positively oriented bilipschitz chart and $f\vert_{U_x} \colon U_x \to fU_x$ is a homeomorphism. The family of sets $\{U_x : x \notin \cB_f\}$ is therefore an open cover of $M \setminus \cB_f$. Since manifolds are second countable and therefore hereditarily Lindel\"of, we may select a countable subcover $\{U_i\}$ of $\{U_x : x \notin \cB_f\}$. Finally, we define
	\[
		E_i = (M \setminus \cB_f) \cap (U_i \setminus \bigcup_{j < i} U_j),
	\]
	and note that $\{\cB_f, E_1, E_2, \dots\}$ is a disjoint partition of $M$ into measurable sets, with $E_i \subset U_i$ for each $i$.
	
	Since $(g \circ f) J_f$ is a positive measurable function on $M$, we may use monotone convergence to split the integral in the change of variables formula into
	\[
		\int_U (g \circ f) J_f \vol_M
		= \int_{\cB_f \cap U} (g \circ f) J_f \vol_M 
		+ \sum_i \int_{E_i \cap U} (g \circ f) J_f \vol_M.
	\]
	Since $\cB_f$ is of zero measure by Lemma \ref{lem:branch_set_measure_zero}, the first integral in the split vanishes. For the other integrals, by $E_i \subset U_i$ and the conditions we used to select $U_i$, we can use the previous case and monotone convergence to conclude that
	\begin{align*}
		\int_U (g \circ f) J_f \vol_M
		&= \sum_i \int_{N} N(f, \cdot, E_i \cap U) \, g \vol_M\\
		&= \int_{N} \left(\sum_i N(f, \cdot, E_i \cap U)\right) g \vol_M.	
	\end{align*}
	
	It remains therefore to show that $\sum_i N(f, y, E_i \cap U) = N(f, y, U)$ for almost every $y \in N$. Note that every element of $f^{-1}\{y\}$ belongs in exactly one set of the partition $\{\cB_f, E_1, E_2, \dots\}$. Hence, if $y \notin f\cB_f$, then $f^{-1}\{y\} \cap \cB_f = \emptyset$, and we may therefore conclude that $\sum_i N(f, y, E_i \cap U) = N(f, y, U)$. Since $f\cB_f$ has measure zero by Lemma \ref{lem:branch_set_measure_zero}, the claim follows.
\end{proof}

Next, we present a version of the change of variables formula with alternating signs. Since we no longer have to use non-negativity of a function in the formulation, we directly formulate this in terms of differential $n$-forms, as that is the more natural language for Riemannian manifolds.

\begin{thm}\label{thm:change_of_variables}
	Let $M$ and $N$ be connected, oriented Riemannian $n$-manifolds, and let $f \colon M \to N$ be a non-constant quasiregular map. Suppose that $\omega$ is a measurable differential $n$-form, $U \subset M$ is measurable, and $N(f, \cdot, U)\, \omega \in L^1(\wedge^n N)$. Then $f^* \omega \in L^1(\wedge^n U)$, and we have the integral transformation formula
	\[
		\int_{U} f^* \omega = \int_{N} N(f, \cdot, U)\, \omega.
	\]
	In particular, if $f \colon M \to N$ is a proper and $\omega \in L^1(\wedge^n N)$, then $f^* \omega \in L^1(\wedge^n M)$ and
	\[
		\int_{M} f^* \omega = (\deg f) \int_N \omega.
	\]
\end{thm}
\begin{proof}
	As usual, $\omega$ can be written in the form $\omega = g \vol_N$, where $g \colon N \to \R$ is measurable. We briefly note that the measurability of $f^*\omega$ is by a similar argument as in Theorem \ref{thm:change_of_variables_non-neg}, as $f^* \omega = (g \circ f) J_f \vol_M$. 
	
	By our integrability assumption, we have that $N(f, \cdot, U)\, \abs{\omega}$ is a non-negative integrable function on $N$. Hence, by Theorem \ref{thm:change_of_variables_non-neg} and the non-negativity of $J_f$, we obtain that
	\begin{multline*}
		\int_U \abs{f^* \omega} \vol_M
		= \int_U \abs{(g \circ f) J_f \vol_M} \vol_M		
		= \int_U (\abs{g} \circ f) J_f \vol_M\\
		= \int_U (\abs{\omega} \circ f) J_f \vol_M
		= \int_{N} N(f, \cdot, U)\, \abs{\omega} \vol_N < \infty.
	\end{multline*}
	Hence, $f^* \omega \in L^1(\wedge^n M)$. Finally, the change of variables formula is derived by splitting the $n$-form $\omega$ into positive and negative parts $\omega = \omega_{+} - \omega_{-}$, where $\omega_{\pm} = g_{\pm} \vol_N$ with $g_{\pm} \colon N \to [0,\infty)$. Theorem \ref{thm:change_of_variables_non-neg} can then be applied to the parts $\omega_{\pm}$, and the claim follows.
	
	The version for proper quasiregular maps is a direct consequence from the previous part due to the third part of Lemma \ref{lem:multiplicity_degree_branching}.
\end{proof}

%%%%%%%%%%%%%%%%%%%%%%%%%%%%%%%%%%%%%%%%%%%%%%%%%%%%%%%%%%%%%%%%%%%%%%%%%%%%%%%%%%%%%%%%%%%%
\section{Quasiregular maps and the pull-back operation}\label{sect:quasiregular_pullback}

With the completion of the proof of Theorem \ref{thm:change_of_variables}, we now have transferred the basic analytic tools from Euclidean spaces into manifolds. Using this machinery, we now discuss the basics of the interaction between quasiregular maps and differential forms. This is probably the most characteristic part of the study of quasiregular maps on Riemannian manifolds; differential forms are rarely encountered when studying quasiregular maps in $\R^n$, but several significant results on manifolds make heavy use of them. In this section, we discuss the most basic and fundamental tool with respect to the interaction of quasiregular maps and differential forms: the pull-back operation.

\subsection{Some linear algebraic properties of $Df$}

Before we begin our main discussion on the pull-back operation, we briefly record some of the properties of $Df$ with roots in linear algebra. These will be used later in the section, when we start studying how to apply change of variables to differential forms.

We recall the $l$-operator discussed in Section \ref{sect:diff_geo_prelims}. In particular, the estimate \eqref{eq:l_jacob_opnorm_ineq} shown there was not in any way dependent on the fact that the map $Df$ was a standard $C^1$ derivative. Hence, this same estimate also holds for weak derivatives $Df$. Therefore, if $f \colon M \to N$ is quasiregular, we have by non-negativity of the Jacobian that
\[
l(Df(x))^n \leq J_f(x) \leq \abs{Df(x)}^n
\]
for almost every $x \in M$.

However, note that the condition $\abs{Df(x)}^n \leq K J_f(x)$ of a quasiregular map can be seen as a converse to the right hand side of this estimate. It turns out that, for purely linear algebraic reasons, this in fact also implies a converse to the left hand side. We record all these estimates to the following Lemma.

\begin{lemma}\label{lem:qr_l_estimates}
	Suppose that $M$ and $N$ are smooth, oriented Riemannian $n$-manifolds, and suppose that $f \colon M \to N$ is $K$-quasiregular. Then for almost every $x \in M$, we have
	\[
		K^{-1} \abs{Df(x)}^n \leq J_f(x) \leq \abs{Df(x)}^n
	\]
	and
	\[
		l(Df(x))^n \leq J_f(x) \leq K^{n-1} l(Df(x))^n.
	\]
\end{lemma}
\begin{proof}
	As discussed before, the only estimate that requires proving is $J_f(x) \leq K^{n-1} l(Df(x))^n$, as the others follow either from \eqref{eq:l_jacob_opnorm_ineq} or the definition of a quasiregular map. 
	
	Suppose then that $x$ is such that the other estimates hold at $x$. Let $\sigma_1 \geq \dots \geq \sigma_n$ be the singular values of $Df(x)$, which are well defined since $T_x M$ and $T_{f(x)} N$ are both $n$-dimensional inner product spaces. Hence, we have $\abs{Df(x)} = \sigma_1$, $l(Df(x)) = \sigma_n$, and $\abs{J_f(x)} = \sigma_1 \cdots \sigma_n$. Moreover, since $J_f(x) \geq K^{-1} \abs{Df(x)}^n \geq 0$, we have $\abs{J_f(x)} = J_f(x)$.
	
	We may then estimate that
	\begin{multline*}
		J_f(x) = \sigma_1 \cdots \sigma_n
		\leq \sigma_n \sigma_1^{n-1}
		= l(Df(x)) \abs{Df(x)}^{n-1}\\
		\leq l(Df(x)) K^{\frac{n-1}{n}} (J_f(x))^{\frac{n-1}{n}}.
	\end{multline*}
	By raising to the power $n$ and dividing the $J_f$-term out from the right, we obtain the desired estimate
	\[
		J_f(x) \leq K^{n-1} l(Df(x))^n.
	\]
\end{proof}

Finally, we record that $Df$ is almost everywhere invertible. The requisite work for this fact has essentially been already done, as we can deduce it from the change of variables formula and the preceding discussion on linear algebra.

\begin{lemma}\label{lem:qr_diff_almost_everywhere_bij}
	Let $M$ and $N$ be connected, oriented Riemannian $n$-manifolds, and let $f \colon M \to N$ be a non-constant quasiregular map. Then for almost every $x \in M$, we have $J_f(x) > 0$. Consequently, for almost every $x \in M$, the map $Df(x) \colon T_x M \to T_{f(x)} N$ is a bijection, and therefore also invertible.
\end{lemma}
\begin{proof}
	Recall that it is a direct consequence of the definition of a quasiregular map that $J_f(x) \geq 0$ for almost every $x \in M$. Let then $E \subset M$ be the set of points of $M$ where $J_f(x) = 0$. By measurability of $J_f$, we obtain that $E$ is measurable. The change of variables formula yields
	\[
		0 = \int_E 0 \vol_M = \int_E J_f \vol_M 
		= \int_N N(f, \cdot, E) \vol_N \geq \int_{fE} \vol_N.
	\]
	It follows that $fE$ is of measure zero, and consequently by the Lusin $(N^{-1})$ property of Theorem \ref{thm:QR_lusin_N_conditions}, that $E$ is of measure zero.
	
	Now, it follows from Lemma \ref{lem:qr_l_estimates} that $l(Df(x)) > 0$ for almost every $x \in M$. Hence, $Df(x) \colon T_x M \to T_{f(x)} N$ is injective for almost every $x \in M$. Since $T_x M$ and $T_{f(x)} N$ are both $n$-dimensional, we conclude that $Df(x)$ is bijective for almost every $x \in M$.
\end{proof}

\subsection{Pullback of measurable differential forms}

We first discuss the measurability of a pull-back $f^* \omega$ by a quasiregular $f$. Since quasiregular maps are locally Sobolev, we already know by Lemma \ref{lem:measurability_and_wedges} that $f^* \omega$ is measurable when $\omega \in C(\wedge^k M)$. 

However, it turns out that the quasiregularity of $f$ in fact gives us a measurable $f^* \omega$ even if we only have $\omega \in \Gamma(\wedge^k M)$. This is something that does not even hold for all smooth maps; the key property which allows this is that quasiregular maps satisfy the Lusin $(N^{-1})$ condition.

\begin{prop}\label{prop:pullback_forms_meas}
	Let $M$ and $N$ be connected, oriented Riemannian $n$-manifolds, and let $f \colon M \to N$ be a non-constant quasiregular map. Suppose that $\omega \in \Gamma(\wedge^k N)$, where $k \in \{0, \dots, n\}$. Then the pull-back $f^* \omega$ is in $\Gamma(\wedge^k M)$. Moreover, $f^* \omega$ changes only in a set of measure zero if $Df$ and $\omega$ are changed in a set of measure zero.
\end{prop}
\begin{proof}
	We first show the latter claim. Namely, if we change $Df$ in a set $E \subset M$ with measure zero, then $f^* \omega$ only changes in $E$, which by assumption has measure zero. And if we change $\omega$ in a set $F \subset N$ of measure zero, then $f^* \omega$ changes only in $f^{-1} E$, which by the Lusin $(N^{-1})$ property of $f$ has measure zero; see Theorem \ref{thm:QR_lusin_N_conditions}.
	
	It remains to show that $f^* \omega$ is measurable. For this, let $x \in M$. Select a smooth positively oriented bilipschitz chart $\psi \colon V \to \R^n$ on $N$, where $V$ is a neighborhood of $f(x)$. Using continuity of $f$, select a smooth positively oriented bilipschitz chart $\phi \colon U \to \R^n$, where $U$ is a neighborhood of $x$ and $fU \subset V$.
	
	It suffices to show that $(\phi^{-1})^* f^* \omega$ is a measurable $k$-form on $\phi U$. We note that $(\phi^{-1})^* f^* \omega$ only depends on the values of $\omega$ on $fU$, which is contained in $V$. Hence, $(\phi^{-1})^* f^* \omega = ((\phi^{-1})^* f^* \psi^*) (\psi^{-1})^*\omega$, and by Lemma \ref{lem:sobolev_composition_mfld_targ_cont}, we obtain $((\phi^{-1})^* f^* \psi^*) (\psi^{-1})^*\omega = (\phi^{-1} \circ f \circ \psi)^* (\psi^{-1})^* \omega$. Now, $(\psi^{-1})^* \omega$ is a measurable differential form in an Euclidean domain, and $\phi^{-1} \circ f \circ \psi$ is a non-constant quasiregular map between Euclidean domains. Hence, we have reduced the question to the Euclidean case.
	
	Consider then the Euclidean case, where $M$ and $N$ are domains in $\R^n$. In this case, we may write $\omega = \sum_I \omega_I \eps_I$, where $\omega_I \colon M \to \R$ are measurable functions and $\eps_I = \dd x_{I_1} \wedge \dots \wedge \dd x_{I_k}$ are the standard basis vectors of differential $k$-forms in $\R^n$. Hence, the pull-back $f^* \omega$ can be written in the Euclidean setting as
	\[
		f^*\omega = \sum_I (\omega_I \circ f) f^* \eps_I
		= \sum_{I} (\omega_I \circ f) \cdot (f^*\dd x_{I_1}) \wedge \dots \wedge (f^* \dd x_{I_k}).
	\]
	It follows from Lemma \ref{lem:measurability_and_wedges} that $f^* \eps_I$ is measurable. Moreover, using the measurability of $\omega_I$ and the Lusin $(N^{-1})$ condition of $f$ given in Theorem \ref{thm:QR_lusin_N_conditions}, we obtain that $\omega_I \circ f$ is measurable. We conclude that $f^* \omega$ is measurable.
\end{proof}

\begin{rem}
	We remark briefly what happens in the trivial case when $f \colon M \to N$ is constant. In that case, $Df$ is the zero map almost everywhere on $M$. Hence, if $k \neq 0$, then $f^* \omega = 0$ almost everywhere for every $\omega \in \Gamma(\wedge^k N)$, and changing $Df$ and $\omega$ in a set of measure zero doesn't affect the outcome. 
	
	However, if $k = 0$, then $\omega$ is merely a measurable function $M \to \R$, and $f^* \omega = \omega \circ f$. In this case the result depends only on the value of $\omega$ in a single point, which is a set of measure zero. Hence, in only this specific case, we do not for obtain a well defined map when we pass to equivalence classes with respect to equality outside sets of measure zero.
\end{rem}

\subsection{The $L^p$ estimate in the conformal exponent}

We have estabilished that quasiregular pull-back is a well defined operation for measurable forms, and it also passes down to equivalence classes of measurable forms with respect to equality outside a set of measure zero. The next question is on how does it interact with the integrability condition of $L^p$ spaces. It turns out that for differential $k$-forms, there is a specific $p$ for which pull-backs by quasiregular maps locally preserve the $L^p$-space: the \emph{conformal exponent} $p = n/k$.

In case of $k = n$, this means that if $\omega$ is locally integrable and $f$ is non-constant quasiregular, then $f^* \omega$ is also locally integrable. In fact, this case is already consequence of the change of variables formula for quasiregular maps which we discussed previously. The other cases will also follow from the change of variables formula, once we obtain the following key pointwise norm estimate for quasiregular pull-backs. The result is a refined version of Lemma \ref{lem:bilip_pullback_pointwise_estimate}, and parts of it have already come up e.g.\ in the proof of Lemma \ref{lem:smooth_form_sobolev_pullback_eucl}.

\begin{lemma}\label{lem:qr_lp_pointwise_estimate}
	Let $M$ and $N$ be connected, oriented Riemannian $n$-manifolds, and let $f \colon M \to N$ be a non-constant quasiregular map. Let $\omega \in \Gamma(\wedge^k N)$, where $k \in \{0, \dots, n\}$. Then, for almost every $x \in M$, we have
	\[
		\binom{n}{k}^{-\frac{1}{2}} l(Df(x))^k \abs{\omega_{f(x)}} 
		\leq \abs{(f^* \omega)_x}
		\leq \binom{n}{k}^{\frac{1}{2}} \abs{Df(x)}^k \abs{\omega_{f(x)}}.
	\]
\end{lemma}

\begin{proof}[Proof of Lemma \ref{lem:qr_lp_pointwise_estimate}:]
	We treat the main case $k \geq 1$ first, and leave the spacial case $k = 0$ to the end. For $k \geq 1$, we show similarly as in Lemma \ref{lem:bilip_pullback_pointwise_estimate} that in fact
	\begin{equation}\label{eq:qr_pointwise_estimate_mass}
		l(Df(x))^k \abs{\omega_{f(x)}}_{\mass}
		\leq \abs{(f^* \omega)_x}_{\mass} \leq \abs{Df(x)}^k \abs{\omega_{f(x)}}_{\mass}
	\end{equation}
	for almost every $x \in M$. Then the claim follows by Lemma \ref{lem:comass_norm_conversion}.
	
	The proof of the upper bound is essentially the same as in Lemma \ref{lem:bilip_pullback_pointwise_estimate}. Namely, suppose then that $x$ is such that $(f^* \omega)_x$ is given by the standard pull-back formula using $Df(x)$. Let $v \in \wedge^k T_x M$ be simple, in which case we may write $v = v_1 \wedge \dots \wedge v_k$ with $v_i \in T_x M$, and suppose that $\abs{v} = 1$. Now, by Lemma \ref{lem:simple_k_vec_props}, we may assume that $v_i$ are pairwise orthogonal, and we consequently obtain that
	\begin{multline*}
		\abs{(Df(x)v_1) \wedge \dots \wedge (Df(x)v_k)}
		\leq \abs{Df(x)v_1} \cdots \abs{Df(x)v_k}\\
		\leq \abs{Df(x)}^k \abs{v_1} \cdots \abs{v_k}
		= \abs{Df(x)}^k \abs{v}
		= \abs{Df(x)}^k.
	\end{multline*}
	Hence, we have that
	\begin{multline*}
	(f^*\omega)_x(v)
	= \omega_{f(x)} \left((Df(x)v_1) \wedge \dots \wedge (Df(x)v_k)\right)\\
	\leq \abs{\omega_{f(x)}}_{\mass} \abs{(Df(x)v_1) \wedge \dots \wedge (Df(x)v_k)}
	\leq \abs{\omega_{f(x)}}_{\mass} \abs{Df(x)}^k.
	\end{multline*}
	By taking the supremum over all $v$, we obtain the upper bound of \eqref{eq:qr_pointwise_estimate_mass}.
	
	For the lower bound, we now select a simple $w \in \wedge^k T_{f(x)} N$ with $\abs{w} = 1$, and as before write $w = w_1 \wedge \dots \wedge w_k$ with $w_i$ pairwise orthogonal. By Lemma \ref{lem:qr_diff_almost_everywhere_bij}, we may assume that $Df(x)$ is bijective at $x$. Hence, for every $i \in \{1, \dots, k\}$ there exists a $u_i \in T_x M$ for which $Df(x) u_i = w_i$. Let $u = u_i \wedge \dots \wedge u_k \in \wedge^k T_x M$, in which case $u$ is simple. Then
	\begin{multline*}
		1 = \abs{w} = \abs{w_1} \cdots \abs{w_k}
		= \abs{Df(x)u_1} \cdots \abs{Df(x) u_k}\\
		\geq l(Df(x))^k \abs{u_1} \cdots \abs{u_k}
		\geq l(Df(x))^k \abs{u}.
	\end{multline*}
	Consequently,
	\begin{multline*}
	l(Df(x))^k \omega_{f(x)}(w) 
	= l(Df(x))^k \omega_{f(x)}((Df(x)u_i) \wedge \dots \wedge (Df(x)u_k))\\
	= l(Df(x))^k (f^*\omega)_x(u)
	\leq l(Df(x))^k \abs{u} \abs{(f^*\omega)_x}_{\mass}
	\leq \abs{(f^*\omega)_x}_{\mass}.
	\end{multline*}
	By a supremum over $w$, we obtain the lower bound of \eqref{eq:qr_pointwise_estimate_mass}.
	
	Finally, we comment on the special case $k = 0$. In this case, $\omega$ is merely a measurable function $N \to \R$, and $f^* \omega = \omega \circ f$. The claim of the lemma requires that
	\[
		l(Df(x))^0 \abs{\omega \circ f(x)}
		\leq \abs{\omega \circ f(x)}
		\leq \abs{Df(x)}^0 \abs{\omega \circ f(x)}.	
	\]
	However, this is trivially true whenever $x \in M$ is such that $l(Df(x)) > 0$, which holds for almost every $x$ by Lemma \ref{lem:qr_diff_almost_everywhere_bij}.
\end{proof}

Now, we state the local $L^p$-estimate for pull-backs of differential forms, and prove it using the tools obtained so far.

\begin{prop}\label{prop:qr_local_lp_estimate}
	Let $M$ and $N$ be connected, oriented Riemannian $n$-manifolds, and let $f \colon M \to N$ be a non-constant quasiregular map. Let $\omega \in L^{n/k}_\loc(\wedge^k N)$, where $k \in \{1, \dots, n\}$. Then $f^* \omega \in L^{n/k}_\loc(\wedge^k M)$. More precisely, if $E \subset M$ is measurable, then
	\begin{multline*}
		\binom{n}{k}^{-\frac{n}{2k}} K^{-(n+1)} \int_N N(f, \cdot, E)
			\abs{\omega}^{\frac{n}{k}} \vol_N
		\leq \int_E \abs{f^* \omega}^{\frac{n}{k}} \vol_M\\
		\leq \binom{n}{k}^{\frac{n}{2k}} K \int_N N(f, \cdot, E) 
			\abs{\omega}^{\frac{n}{k}} \vol_N.
	\end{multline*}
	
	In particular, suppose that $f \colon M \to N$ is a proper non-constant quasiregular map, and that $\omega \in L^{n/k}(\wedge^k N)$. Then $f^* \omega \in L^{n/k}(\wedge^k M)$, with
	\begin{multline*}
		\qquad \dfrac{(\deg f)^{\frac{k}{n}}}{\displaystyle\binom{n}{k}^{\frac{1}{2}} 
				K^{\frac{k(n-1)}{n}}} 
			\norm{\omega}_{\frac{n}{k}}
		\leq \norm{f^* \omega}_{\frac{n}{k}}
		\leq \binom{n}{k}^{\frac{1}{2}} K^{\frac{k}{n}} (\deg f)^{\frac{k}{n}} \norm{\omega}_{\frac{n}{k}}.\qquad
	\end{multline*}
\end{prop}
\begin{proof}
	By Lemma \ref{lem:qr_lp_pointwise_estimate}, we may estimate
	\begin{multline*}
		\binom{n}{k}^{-\frac{n}{2k}} 
			\int_E (\abs{\omega}^{\frac{n}{k}} \circ f) l(Df)^n \vol_M
		\leq \int_E \abs{f^* \omega}^{\frac{n}{k}} \vol_M\\
		\leq \binom{n}{k}^{\frac{n}{2k}} 
			\int_E (\abs{\omega}^{\frac{n}{k}} \circ f) \abs{Df}^n \vol_M.
	\end{multline*}
	By the definition of a quasiregular map as well as Lemma \ref{lem:qr_l_estimates}, we may further estimate
	\begin{multline*}
		\binom{n}{k}^{-\frac{n}{2k}} K^{-(n-1)}
			\int_E (\abs{\omega}^{\frac{n}{k}} \circ f) J_f \vol_M
		\leq \int_E \abs{f^* \omega}^{\frac{n}{k}} \vol_M\\
		\leq \binom{n}{k}^{\frac{n}{2k}} K
			\int_E (\abs{\omega}^{\frac{n}{k}} \circ f) J_f \vol_M.
	\end{multline*}
	Finally, by applying the change of variables formula for non-negative functions given in Theorem \ref{thm:change_of_variables_non-neg}, we obtain the desired estimate
	\begin{multline*}
		\binom{n}{k}^{-\frac{n}{2k}} K^{-(n+1)} \int_N N(f, \cdot, E)
			\abs{\omega}^{\frac{n}{k}} \vol_N
		\leq \int_E \abs{f^* \omega}^{\frac{n}{k}} \vol_M\\
		\leq \binom{n}{k}^{\frac{n}{2k}} K \int_N N(f, \cdot, E) 
			\abs{\omega}^{\frac{n}{k}} \vol_N.
	\end{multline*}
	
	In order to conclude local $L^{n/k}$-integrability of $f^* \omega$ from the above estimate, let $x \in M$, and using Lemmas \ref{lem:preimage_normal_neighborhoods} and \ref{lem:preimage_normal_neighborhood_properties}, select a normal neighborhood $U$ of $x$ for which $\omega \in L^{n/k}(\wedge^k fU)$. Then by Theorem \ref{thm:deg_i_sum_result} and Lemma \ref{lem:multiplicity_degree_branching}, we have that $N(f, \cdot, V) \equiv \deg f\vert_V$ almost everywhere on $fV$. Since $N(f, \cdot, V) \equiv 0$ on $N \setminus fV$, we obtain that $N(f, \cdot, V) \abs{\omega}^{n/k} \in L^{1}(\wedge^k N)$. Hence, by the above estivate with $E = V$, we obtain that $f^* \omega \in L^{n/k}(\wedge^k V)$, which implies local $L^{n/k}$-integrability.
	
	Finally, the version for proper $f$ follows by applying the integral estimate with $E = M$, raising to power $k/n$, and using the fact that $N(f, \cdot) \equiv \deg f$ almost everywhere due to Theorem \ref{thm:deg_i_sum_result} and Lemma \ref{lem:multiplicity_degree_branching}.
\end{proof}

\begin{rem}
	For $k = 0$, Proposition \ref{prop:qr_local_lp_estimate} has a natural $L^\infty$-counterpart. Namely, suppose that $f$ satisfies the assumptions of Proposition \ref{prop:qr_local_lp_estimate}, and let $\omega \in L^\infty(\wedge^0 N)$. Then $\omega$ is merely an essentially bounded measurable function $\omega \in L^\infty(N)$. The pull-back of $\omega$ by $f$ is $\omega \circ f$. By the Lusin properties of $f$ (see Theorem \ref{thm:QR_lusin_N_conditions}), a set $E \subset M$ is of measure zero if and only if $fE$ is of measure zero. Consequently, we have that
	\[
		\norm{\omega}_\infty = \norm{f^* \omega}_\infty,
	\]
	which is the omitted $k = 0$ case of the result of Proposition \ref{prop:qr_local_lp_estimate}.
\end{rem}

\begin{rem}
	Beyond Proposition \ref{prop:qr_local_lp_estimate}, the next step in obtaining $L^p$ results for quasiregular pull-backs of differential forms would be to take advantage of higher integrability of the Jacobian. Namely, for a non-constant quasiregular maps, the Jacobian $J_f$ is $L^r_\loc$ for some $r > 1$. See e.g.\ Elcrat and Meyers \cite{Meyers-Elcrat_HigherInt} or Martio \cite{Martio_HigherInt} for the Euclidean results, which can be locally transferred to manifolds using charts as we have already done with many results in these notes. We leave out an in-depth look at this topic in order to not make these notes grow too large. For the interested reader, see for example \cite[Lemma 3.1]{Kangasniemi-Pankka_PLMS} for a typical $L^p$ result using higher integrability, as well as the standard method used to obtain such results.
\end{rem}

\subsection{Sobolev forms and quasiregular maps}

In our discussion of quasi\-regular pull-backs of differential forms, the remaining main point of discussion is how much differentiability the pull-back preserves.

The main result we give here is on the interaction of quasiregular pull-back and the spaces $W^{d, p, q}(\wedge^k M)$. Namely, when the exponents of integrability $p$ and $q$ are the corresponding conformal exponents, then we obtain a well defined quasiregular pull-back map which commutes with taking weak differentials. We note that the quasiconformal case has been discussed in \cite[Lemma 2.22]{Donaldson-Sullivan_Acta} and \cite[Theorem 6.6]{Goldshtein-Troyanov_DeRham}, and the quasiregular case is not particularly different.

\begin{thm}\label{thm:qr_conformal_exponent_Sobolev_pullback}
	Let $f \colon M \to N$ be a non-constant quasiregular map between connected, oriented Riemannian $n$-manifolds. Let $k \in \{0, \dots, n-1\}$, and suppose that
	\[
		\omega \in W^{d, \frac{n}{k}, \frac{n}{k+1}}_\loc(\wedge^k N),
	\]
	where for $k = 0$, we interpret this as $\omega \in W^{d, \infty, n}_\loc(\wedge^k N)$. Then
	\[
		f^*\omega \in W^{d, \frac{n}{k}, \frac{n}{k+1}}_\loc(\wedge^k M),
	\]
	and the weak differential is given by
	\[
		d f^* \omega = f^* d\omega.
	\]
\end{thm}

We break down the proof to two lemmas. The requisite integrabilities have been handled in Proposition \ref{prop:qr_local_lp_estimate}, so the remaining work is to show that $f^* d\omega$ is a weak differential of $f^* \omega$. By Corollary \ref{cor:smooth_form_sobolev_pullback_cont}, we know that this holds if $f$ is between Euclidean domains and $\omega$ is smooth. Our strategy is to first show the result in the Euclidean setting, and then on manifolds.

We begin with the Euclidean version of the result.

\begin{lemma}\label{lem:qr_Sobolev_pullback_eucl}
	Let $U, V$ be open domains in $\R^n$. Let $f \colon U \to V$ be a non-constant quasiregular map, and let $\omega \in W^{d, n/k, n/(k+1)}_\loc(\wedge^k V)$, where $k \in \{0, \dots, n-1\}$. Then $f^* \omega \in W^{d, n/k, n/(k+1)}_\loc(\wedge^k U)$, and $d f^* \omega = f^* d \omega$.
\end{lemma}
\begin{proof}
	By Proposition \ref{prop:qr_local_lp_estimate}, we only need to show that $f^* d\omega = d f^* \omega$. This amounts to fixing $\eta \in C^\infty_0(\wedge^{n-k-1} U)$ and showing that
	\[
		\int_U f^* \omega \wedge d\eta = (-1)^{k-1} \int_U f^* d\omega \wedge \eta.
	\]
	
	We may suppose that $\spt \eta \subset D$, where $D$ is compactly contained in a bounded normal domain $U' = U_f(x, r)$ of $f$. Indeed, we may cover $\spt \eta$ with finitely many such domains, and then decompose $\eta = \eta_1 + \dots + \eta_l$ using a smooth subordinate partition of unity. Showing the desired formula for $\eta_i$ then implies it for $\eta$. 
	
	We again have due to continuity and openness of $f$ that $fD$ is compactly contained in $f U'$. Hence, multiplying $\omega$ by a smooth cutoff function yields a form $\omega' \in W^{d, n/k, n/(k+1)}(\wedge^k \R^n)$ for which $\omega' = \omega$ on $fD$ and $\spt \omega' \subset f U'$.
	
	We consider first the case $k > 0$, leaving the case $k = 0$ for later. Since we supposed that $k > 0$, we may use Proposition \ref{prop:W_dp_smooth_approx} to approximate $\omega'$ by a sequence $\omega'_i \in C^\infty(\wedge^k \R^n)$ for which $\norm{\omega' - \omega'_i}_{n/k} \to 0$ and $\norm{d\omega' - d\omega'_i}_{n/(k+1)} \to 0$.
	
	Since $D$ is contained in the normal domain $U'$ of $f$, the multiplicity function $N(f, \cdot, D)$ is bounded. Hence, the estimate of Proposition \ref{prop:qr_local_lp_estimate} yields that $\norm{f^*\omega' - f^*\omega'_i}_{n/k} \to 0$ and $\norm{f^*d\omega' - f^*d\omega'_i}_{n/(k+1)} \to 0$. By Corollary \ref{cor:smooth_form_sobolev_pullback_cont}, $f^* \omega'_i$ has the weak derivative $f^* d\omega'_i$. Then, by using a similar Hölder estimate argument as the one in the proof of Proposition \ref{prop:wedges_of_sobolev_forms}, we obtain
	\begin{multline*}
		\abs{\int_U f^* \omega \wedge d\eta + (-1)^{k} \int_U f^* d\omega \wedge \eta}\\
		\leq C
			\left( \norm{f^*\omega' - f^*\omega'_i}_{\frac{n}{k}}
				\norm{d\eta}_{\frac{n}{n-k}} + 
			\norm{f^*d\omega' - f^*d\omega'_i}_{\frac{n}{k+1}}
				\norm{\eta}_{\frac{n}{n-k-1}}\right),
	\end{multline*}
	where the right hand side tends to zero. Hence, $f^* d\omega = d f^* \omega$.
	
	What remains is the more involved case $k = 0$. Here, the exponent $\infty$ stops us from taking a smooth approximation of $\omega'$ in $W^{d, \infty, n}(\wedge^0 \R^n)$. However, since in the case $k = 0$ the forms are in fact Sobolev functionals, there exists a workaround using Sobolev inequalities, presented by Heinonen, Kilpel\"ainen and Martio in \cite[Theorem 14.28]{Heinonen-Kilpelainen-Martio_book}. We explain this method also here. 
	
	Namely, we note that $\omega' \in W^{d, n}(\wedge^0 f U') = W^{1, n}(f U')$, and $\spt \omega'$ is compactly contained in $fU'$. We therefore may approximate $\omega'$ with $\omega'_i \in C^\infty_0(fU')$ in the $W^{1,n}$-norm. Now we have by Corollary \ref{cor:smooth_form_sobolev_pullback_cont} and Proposition \ref{prop:qr_local_lp_estimate} that $f^* \omega_i' \in W^{d, \infty, n}_\loc(\wedge^0 U')$ and $\norm{f^* d\omega_i' - f^* d\omega'}_n \to 0$.
	
	Since $U' = U_f(x, r)$, we have by Lemma \ref{lem:preimage_normal_neighborhood_properties} that $f \colon U' \to fU'$ is proper. Hence, $f^* d\omega_i'$ are compactly supported in $U'$, and it follows by the Gagliardo--Nirenberg Sobolev inequality given in Theorem \ref{thm:sobolev_eucl_inequality} that
	\[
		\norm{f^* \omega_i' - f^* \omega_j'}_n \leq C m_n(U')^\frac{1}{n} \norm{f^* d\omega_i' - f^* d\omega_j'}_n. 
	\]
	Thus, $(f^* \omega_i')$ is a Cauchy sequence in $W^{1,n}(U')$. Therefore, by the completeness of $W^{1, n}(D)$ stated in Theorem \ref{thm:sobolev_eucl_banach}, $(f^* \omega_i')$ converges to a 0-form $\tau \in W^{1, n}(U')$.
	
	However, since $\norm{\omega_i' - \omega'}_n \to 0$, we have $(\omega_i' - \omega')_x \to 0$ for almost every $x \in D$. Hence, by the Lusin property of $f$ given in Theorem \ref{thm:QR_lusin_N_conditions}, and since $\omega_i'$ and $\omega'$ are 0-forms, we have that $(f^*\omega_i' - f^*\omega')_x = (\omega_i' - \omega')_{f(x)} \to 0$ for almost every $x \in D$. Since also $f^*\omega_i' - \tau \to 0$ almost everywhere on $D$, we have $\tau = f^*\omega'$ almost everywhere on $D$. Moreover, since $\norm{f^* d\omega_i' - f^* d\omega'}_n \to 0$, we have $d\tau = f^* d\omega'$ on $D$. Hence,
	\begin{multline*}
	\int_U f^* \omega \wedge d\eta 
	= \int_D f^* \omega' \wedge d\eta 
	= \int_D \tau \wedge d\eta
	= (-1)^{k-1} \int_D d\tau \wedge \eta\\ 
	= (-1)^{k-1} \int_D f^* d\omega' \wedge \eta
	= (-1)^{k-1} \int_U f^* d\omega \wedge \eta.
	\end{multline*}
	The proof of the case $k = 0$ is therefore complete.
\end{proof}

We then finally finish the proof of Theorem \ref{thm:qr_conformal_exponent_Sobolev_pullback}, where what remains is transferring the results up until now from the Euclidean setting to the manifold setting.

\begin{proof}[Proof of Theorem \ref{thm:qr_conformal_exponent_Sobolev_pullback}]
	We have by Proposition \ref{prop:qr_local_lp_estimate} that $f^*\omega \in L^{n/k}_\loc(\wedge^k M)$ and $f^* d\omega \in L^{n/(k+1)}_\loc(\wedge^k M)$. The remaining step is therefore to show that $f^* d\omega$ is a weak differential of $f^* \omega$.
	
	It suffices to verify that $f^*\omega$ satisfies \eqref{eq:weak_differential_for_forms} for every such test form $\eta \in C^\infty_0(\wedge^{n-k-1} M)$ for which there exist positively oriented smooth bilipschitz charts $\phi \colon U \to \R^n$ and $\psi \colon V \to \R^n$ in $M$ and $N$, respectively, satisfying the following conditions: $fU \subset V$, and $\spt \eta \subset U$. Indeed, the argument is again to cover $\spt \eta$ with finitely many such sets $U_1, \dots, U_l$, and use a smooth subordinate partition of unity to decompose $\eta = \eta_1 + \dots + \eta_l$ where $\spt \eta_i \subset U_i$. Then showing the claim for $\eta_i$ again implies it for $\eta$.
	
	Suppose then that $\eta$ is as above, with charts $\phi \colon U \to \R^n$ and $\psi \colon V \to \R^n$. We wish to verify that
	\[
		\int_U f^*\omega \wedge d\eta = (-1)^{k+1} \int_U f^*d\omega \wedge \eta.
	\]
	A diffeomorphic change of variables with respect to $\phi^{-1}$ reduces this to verifying that
	\[
		\int_{\phi U} (f \circ \phi^{-1})^* \omega \wedge d \eta'
		= (-1)^{k+1} \int_{\phi U} (f \circ \phi^{-1})^* d\omega \wedge \eta',
	\]
	where $\eta' = (\phi^{-1})^* \eta \in C^\infty_0(\wedge^{n-k-1} \phi U)$. Moreover, since $fU \subset V$, we may further write
	\[
		\int_{\phi U} (\psi \circ f \circ \phi^{-1})^* ((\psi^{-1})^* \omega) \wedge d \eta'
		= (-1)^{k+1} \int_{\phi U} (\psi \circ f \circ \phi^{-1})^* ((\psi^{-1})^* d\omega) \wedge \eta'.
	\]
	However, since $\omega \in W^{d, n/k, n/(k+1)}_\loc(V)$ and $\psi^{-1}$ is smooth bilipschitz, we in fact have by Proposition \ref{prop:wdpqloc_bilip_pullback_formula} that $(\psi^{-1})^* \omega \in W^{d, n/k, n/(k+1)}_\loc(\psi V)$ and $d (\psi^{-1})^* \omega = (\psi^{-1})^* d\omega$.
	
	Hence, we have reduced the problem to showing that
	\[
		\int_{U'} F^* \omega' \wedge d \eta'
		= (-1)^{k+1} \int_{U'} F^* d\omega' \wedge \eta',
	\]
	where $U' = \phi U$ and $V' = \psi V$ are domains in $\R^n$, the form $\omega' = (\psi^{-1})^* \omega$ is in $W^{d, n/k, n/(k+1)}_\loc(\wedge^k V')$, the test form $\eta$ is in $C^\infty_0(\wedge^k U')$, and the map $F = \psi \circ f \circ \phi^{-1} \colon U' \to V'$ is non-constant quasiregular. This follows immediately from Lemma \ref{lem:qr_Sobolev_pullback_eucl}. The proof is therefore complete.
\end{proof}

%%%%%%%%%%%%%%% Bibliography %%%%%%%%%%%%%%%
\bibliographystyle{abbrv}
\bibliography{sources}

\begin{thebibliography}{10}

\bibitem{Adams-Fournier_Sobolev}
R.~Adams and J.~Fournier.
\newblock {\em Sobolev spaces}.
\newblock Elsevier, 2003.

\bibitem{Astala-Iwaniec-Martin_Book}
K.~Astala, T.~Iwaniec, and G.~Martin.
\newblock {\em Elliptic partial differential equations and quasiconformal
  mappings in the plane}.
\newblock Princeton university press, 2009.

\bibitem{Benilan-et-al}
P.~B{\'e}nilan, L.~Boccardo, T.~Gallou{\"e}t, R.~Gariepy, M.~Pierre, and J.~L.
  V{\'a}squez.
\newblock An $l^{1}$-theory of existence and uniqueness of solutions of
  nonlinear elliptic equations.
\newblock {\em Annali della Scuola Normale Superiore di Pisa-Classe di
  Scienze}, 22(2):241--273, 1995.

\bibitem{Bojarski-Iwaniec}
B.~Bojarski and T.~Iwaniec.
\newblock Analytical foundations of the theory of quasiconformal mappings in
  $\mathbb{R}^n$.
\newblock {\em Ann. Acad. Sci. Fenn. Ser. AI Math}, 8:257--324, 1983.

\bibitem{Bredon_book_algtopo}
G.~E. Bredon.
\newblock {\em Topology and geometry}.
\newblock Springer, 1993.

\bibitem{Convent-VanSchaftingen_Sobolev}
A.~Convent and J.~Van~Schaftingen.
\newblock Intrinsic co-local weak derivatives and {S}obolev spaces between
  manifolds.
\newblock {\em Ann. Sc. Norm. Super. Pisa Cl. Sci. (5)}, 16(1):97--128, 2016.

\bibitem{Donaldson-Sullivan_Acta}
S.~Donaldson and D.~Sullivan.
\newblock Quasiconformal 4-manifolds.
\newblock {\em Acta Math.}, 163(1):181--252, 1989.

\bibitem{Federer_book}
H.~Federer.
\newblock {\em Geometric meeasure theory}.
\newblock Springer, 1969.

\bibitem{Gehring-Martin-Palka_book}
F.~Gehring, G.~Martin, and B.~Palka.
\newblock {\em An introduction to the theory of higher-dimensional
  quasiconformal mappings}.
\newblock American Mathematical Society, 2017.

\bibitem{Goldshtein-Troyanov_DeRham}
V.~Gol’dshtein and M.~Troyanov.
\newblock A conformal de rham complex.
\newblock {\em J. Geom. Anal.}, 20(3):651--669, 2010.

\bibitem{Goldstein-Hajlasz-Pakzad}
P.~Goldstein, P.~Haj{\l}asz, and M.~R. Pakzad.
\newblock Finite distortion {S}obolev mappings between manifolds are
  continuous.
\newblock {\em Int. Math. Res. Not. IMRN}, 2019(14):4370--4391, 2017.

\bibitem{Guo-Williams}
C.-Y. Guo and M.~Williams.
\newblock Geometric function theory: the art of pullback factorization.
\newblock 2016.
\newblock Preprint, arxiv.org/abs/1611.02478.

\bibitem{Hajlasz-Iwaniec-Maly-Onninen}
P.~Haj{\l}asz, T.~Iwaniec, J.~Mal{\`y}, and J.~Onninen.
\newblock Weakly differentiable mappings between manifolds.
\newblock {\em Mem. Amer. Math. Soc.}, 192(899), 2008.

\bibitem{Hatcher_book}
A.~Hatcher.
\newblock {\em Algebraic Topology}.
\newblock Cambridge university press, 2002.

\bibitem{Heinonen_book}
J.~Heinonen.
\newblock {\em Lectures on analysis on metric spaces}.
\newblock Springer, 2001.

\bibitem{Heinonen-Kilpelainen-Martio_book}
J.~Heinonen, T.~Kilpel{\"a}inen, and O.~Martio.
\newblock {\em Nonlinear potential theory of degenerate elliptic equations}.
\newblock Dover, 2006.

\bibitem{Heinonen-Koskela-Shanmugalingam-Tyson}
J.~Heinonen, P.~Koskela, N.~Shanmugalingam, and J.~Tyson.
\newblock {\em Sobolev spaces on metric measure spaces: an approach based on
  upper gradients}.
\newblock Cambridge university press, 2015.

\bibitem{Hencl-Koskela_book}
S.~Hencl and P.~Koskela.
\newblock {\em Lectures on Mappings of Finite Distortion}.
\newblock Springer, 2014.

\bibitem{Holopainen-Pankka_Notes}
I.~Holopainen and P.~Pankka.
\newblock $p$-{L}aplace operator, quasiregular mappings, and {P}icard-type
  theorems.
\newblock In {\em Proceedings of the International Workshop on Quasiconformal
  Mappings and their Applications}, 2007.

\bibitem{Holopainen-Rickman_Picard}
I.~Holopainen and S.~Rickman.
\newblock A {P}icard type theorem for quasiregular mappings of $\mathbb{R}^n$
  into $n$-manifolds with many ends.
\newblock {\em Rev. Math. Iberoam.}, 8(2):131--148, 1992.

\bibitem{Iwaniec-Martin_book}
T.~Iwaniec and G.~Martin.
\newblock {\em Geometric function theory and non-linear analysis}.
\newblock Clarendon Press, 2001.

\bibitem{Iwaniec-Scott-Stroffolini}
T.~Iwaniec, C.~Scott, and B.~Stroffolini.
\newblock Nonlinear {H}odge theory on manifolds with boundary.
\newblock {\em Ann. Mat. Pura. Appl. (4)}, 177(1):37--115, 1999.

\bibitem{Kangaslampi_thesis}
R.~Kangaslampi.
\newblock Uniformly quasiregular mappings on elliptic {R}iemannian manifolds.
\newblock {\em Ann. Acad. Sci. Fenn. Math. Diss.}, 151, 2008.
\newblock Dissertation, Helsinki University of Technology, Espoo, 2008.

\bibitem{Kangasniemi-Pankka_PLMS}
I.~Kangasniemi and P.~Pankka.
\newblock Uniform cohomological expansion of uniformly quasiregular mappings.
\newblock {\em Proc. London Math. Soc.}, 118:701--728, 2019.

\bibitem{Kirsila_GenManifoldsMFD}
V.~Kirsil\"a.
\newblock Mappings of finite distortion from generalized manifolds.
\newblock {\em Conf. Geom. Dyn.}, 18:229--262, 2014.

\bibitem{Klingenberg_RiemannGeoBook}
W.~Klingenberg.
\newblock {\em Riemannian geometry}.
\newblock de Gruyter, 2011.

\bibitem{Lee_RiemGeo}
J.~M. Lee.
\newblock {\em Introduction to smooth manifolds}.
\newblock Springer, 2013.

\bibitem{Lindquist-Pankka}
J.~Lindquist and P.~Pankka.
\newblock Vertical quasi-isometries and branched quasisymmetries.
\newblock 2019.
\newblock Preprint, https://arxiv.org/abs/1911.12680.

\bibitem{Martio_HigherInt}
O.~Martio.
\newblock On the integrability of the derivative of a quasiregular mapping.
\newblock {\em Math. Scand.}, 35(1):43--48, 1975.

\bibitem{Meyers-Elcrat_HigherInt}
N.~G. Meyers and A.~Elcrat.
\newblock Some results on regularity for solutions of non-linear elliptic
  systems and quasi-regular functions.
\newblock {\em Duke Math. J}, 42(1):121--136, 1975.

\bibitem{Nash_embedding}
J.~F. Nash.
\newblock The imbedding problem for {R}iemannian manifolds.
\newblock {\em Ann. of Math.}, 63:20--63, 1956.

\bibitem{Okuyama-Pankka_measure}
Y.~Okuyama and P.~Pankka.
\newblock Equilibrium measures for uniformly quasiregular dynamics.
\newblock {\em J. Lond. Math. Soc.}, 89(2):524--538, 2014.

\bibitem{Pankka_Lectnotes_degreetheory}
P.~Pankka.
\newblock Degree theory and branched covers.
\newblock Lecture notes, 2018.

\bibitem{Pankka_Lectnotes_Fribourg}
P.~Pankka.
\newblock Six hours on quasiregular mappings: Fribourg lecture notes.
\newblock 2014.

\bibitem{Reshetnyak-book}
Y.~G. Reshetnyak.
\newblock {\em {Space mappings with bounded distortion}}, volume~73 of {\em
  {Translations of Mathematical Monographs}}.
\newblock American Mathematical Society, Providence, RI, 1989.

\bibitem{Reshetnyak_metric_Sobolev}
Y.~G. Reshetnyak.
\newblock Sobolev-type classes of functions with values in a metric space.
\newblock {\em Sib. Math. J.}, 38(3):567--583, 1997.

\bibitem{Rickman_book}
S.~Rickman.
\newblock {\em Quasiregular mappings}, volume~26.
\newblock Springer-Verlag, 1993.

\bibitem{Scott_HodgeTheory}
C.~Scott.
\newblock ${L^p}$-theory of differential forms on manifolds.
\newblock {\em Trans. Amer. Math. Soc.}, 347(6):2075--2096, 1995.

\bibitem{Suominen_QC-Manifold}
K.~Suominen.
\newblock Quasiconformal maps in manifolds.
\newblock {\em Ann. Acad. Sci. Fenn. Ser. AI Math}, 393:1--39, 1967.

\end{thebibliography}

\end{document}